\PassOptionsToPackage{pagebackref=true}{hyperref}
\documentclass[11pt,a4paper,reqno]{amsart}
\usepackage[T1]{fontenc}
\usepackage{lmodern,microtype}
\usepackage[margin=24mm,headheight=14pt,headsep=7mm]{geometry}
\usepackage{amsmath,amssymb,amsthm,mathtools,mathrsfs}
\usepackage{tikz-cd,graphicx,booktabs,array,enumitem,etoolbox,placeins}
\usepackage{mathrsfs}
\usetikzlibrary{arrows.meta,calc,positioning}
\usepackage{caption,flafter}
\usepackage[hidelinks]{hyperref}
\usepackage{bookmark}
\numberwithin{equation}{section}
\newtheorem{theorem}{Theorem}[section]
\newtheorem{proposition}[theorem]{Proposition}
\newtheorem{lemma}[theorem]{Lemma}
\newtheorem{corollary}[theorem]{Corollary}
\theoremstyle{definition}
\newtheorem{definition}[theorem]{Definition}
\newtheorem{example}[theorem]{Example}
\theoremstyle{remark}
\newtheorem{remark}[theorem]{Remark}

\usepackage{tikz}
\usetikzlibrary{calc}

\usetikzlibrary{arrows.meta,calc,decorations.markings}

\usepackage{xcolor}

\renewcommand*{\backref}[1]{}

\renewcommand*{\backrefalt}[4]{%
	\ifcase #1\relax
	\or
	\space(Cited on page~#2.)%
	\else
	\space(Cited on pages~#2.)%
	\fi
}

\newtheorem{conjecture}[theorem]{Conjecture}

\DeclareMathOperator{\add}{add}
\DeclareMathOperator{\thick}{thick}
\DeclareMathOperator{\per}{per}
\DeclareMathOperator{\Hom}{Hom}
\DeclareMathOperator{\End}{End}
\DeclareMathOperator{\Ext}{Ext}
\DeclareMathOperator{\ind}{ind}
\DeclareMathOperator{\Gr}{Gr}
\DeclareMathOperator{\CM}{CM}
\DeclareMathOperator{\GP}{GP}
\DeclareMathOperator{\GI}{GI}
\DeclareMathOperator{\proj}{proj}

\DeclareMathOperator{\rad}{rad}

\newcommand{\llbracket}{[\![}
\newcommand{\rrbracket}{]\!]}
\newcommand{\kk}{\mathbb C}
\newcommand{\ZZ}{\mathbb Z}
\newcommand{\NN}{\mathbb N}
\newcommand{\cE}{\mathcal E}

\newcommand{\cC}{\mathcal C}
\newcommand{\cD}{\mathcal D}
\newcommand{\cP}{\mathcal P}
\newcommand{\cN}{\mathcal N}
\newcommand{\cI}{\mathcal I}
\newcommand{\cS}{\mathcal S}
\newcommand{\CC}{\operatorname{CC}}

\newcommand{\Db}{\mathcal D^b}
\newcommand{\Dsg}{\mathcal D_{\mathrm{sg}}}
\newcommand{\fgmod}{\operatorname{fgmod}}

\newcommand{\bB}{\widetilde B}

\DeclareMathOperator{\Cone}{Cone}
\allowdisplaybreaks[2]
\AtBeginDocument{%
\setlength{\abovedisplayskip}{7pt plus 2pt minus 1pt}%
\setlength{\belowdisplayskip}{7pt plus 2pt minus 1pt}%
\setlength{\abovedisplayshortskip}{3pt plus 1pt}%
\setlength{\belowdisplayshortskip}{3pt plus 1pt}}
\title[]{Quasi-equivalence of relabeled positroid cluster structures}
\author{Yilin Wu}
\date{}
\hypersetup{pdftitle={Quasi-equivalence of relabeled positroid cluster structures},
 pdfsubject={Research manuscript; prose revision 6, 20 September 2026},pdfauthor={}}
\begin{document}
\begin{abstract}
	We prove the Fraser--Sherman-Bennett conjecture on the
	quasi-equivalence of positroid cluster structures arising from
	admissibly relabeled plabic graphs. We construct a tilting module
	from the boundary necklace and obtain a Frobenius
	categorification of the relabeled cluster structure. By identifying
	localized characters of rank-one modules with Pl\"ucker coordinates,
	we match the stable images of the mutable face modules with the
	summands of a reachable cluster-tilting object in the
	stable category. The induced derived equivalence yields a unimodular
	transformation of relative index lattices intertwining the full
	extended exchange matrices. These comparisons show that the identity
	map on the coordinate ring is a quasi-cluster isomorphism between
	the two cluster structures.
\end{abstract}
\maketitle

\setcounter{tocdepth}{1}
\tableofcontents

\section{Introduction and main theorem}

The coordinate ring of an open positroid variety admits cluster structures
whose seeds are described by reduced plabic graphs
\cite{GL}. Each graph has source and target face labellings,
and Fraser and Sherman--Bennett \cite{FSB} construct
further cluster structures by admissibly permuting its boundary labels.
These structures have the same coefficient group, but their chosen frozen
generators can differ. The natural comparison is therefore
quasi-equivalence: after mutation, corresponding mutable variables differ
by frozen Laurent monomials and the exchange ratios agree.

Fraser and Sherman--Bennett conjecture that all cluster structures obtained
in this way are quasi-equivalent. Pressland
\cite[Theorem~6.16]{Pressland} proves the source--target case using
Frobenius categorification and the categorical comparison of Fraser and
Keller \cite[Appendix~A]{KW}. 

The general relabelling problem
requires a comparison with a new boundary necklace and is not covered by
that argument; see \cite[Remark~3.8]{Pressland}. In this paper, we
prove the general conjecture by constructing a tilting module from the
new boundary necklace and comparing the resulting Frobenius model with
the standard positroid model.

We work over $\mathbb C$ and invert all frozen variables. Let $\pi$ be
a decorated permutation defining a loopless positroid
$\mathcal M_\pi\subseteq\binom{[n]}k$. Write
$\widehat\Pi_\pi^\circ$ for the corresponding open subset of the affine
Pl\"ucker cone, and let $\mathcal I^\pi$ be the set of distinct terms of
its ordinary Grassmann necklace. Put
\[
R_\pi=\mathbb C[\widehat\Pi_\pi^\circ],
\qquad
\mathscr F_\pi=\langle\Delta_I:I\in\mathcal I^\pi\rangle_{\mathbb Z}
\subseteq R_\pi^\times.
\]
For $\rho\in S_n$, set $\mu=\rho^{-1}\pi\rho$. We call $\rho$
\emph{admissible for $\pi$} if
\[
\pi\rho\leq_\circ\pi,
\qquad
\dim\widehat\Pi_\mu^\circ=\dim\widehat\Pi_\pi^\circ,
\]
where $\leq_\circ$ is the circular weak order. If $G$ is a reduced
plabic graph with trip permutation $\mu$, let $G^\rho$ be obtained by
replacing its boundary label $a$ by $\rho(a)$. By
\cite[Theorems~4.14 and~4.21]{FSB}, the target face
coordinates of $G^\rho$ form a seed for $R_\pi$, and its distinct
boundary coordinates form a basis of $\mathscr F_\pi$.

\begin{theorem}[Theorem \ref{prop:necklace-connected-comparison}, Corollary \ref{general case}]
	\label{thm:introduction-fsb}
	Let $\rho$ be admissible for $\pi$, and let $D$ be an ordinary reduced
	Postnikov diagram with trip permutation $\pi$. Then the identity of
	$R_\pi$ is a quasi-cluster isomorphism from the target-labelled cluster
	structure of $G^\rho$ to either the source-labelled or the target-labelled
	cluster structure of $D$.
\end{theorem}

Thus these cluster structures quasi-coincide on the same coordinate
ring. In particular, the target comparison proves
\cite[Conjecture~1.3]{FSB}.

We explain the main steps for a connected positroid. Let $B$ be the
boundary order associated with $D$, and let $M_J$ be the rank-one
$B$-lattice corresponding to $J\in\mathcal M_\pi$. We first show that
\[
\Phi_s(M_J)=\Delta_J,
\]
where $\Phi_s$ is the localized character of the standard
source Frobenius categorification. The character is defined on
$D^b(\operatorname{fgmod}B)$, so this equality makes sense even when
$M_J$ does not belong to the source Frobenius category. We obtain it
from the source--target and twist formulas by extending the character
comparison to the bounded derived categories.

Let $\mathcal S$ and $\mathcal I$ be the sets of distinct target labels
of all faces and of the boundary faces of $G^\rho$, respectively.
We first compare the mutable face coordinates with the cluster
variables in the standard source structure. By
\cite[Theorem~8.9]{MMMSV}, every nonzero Pl\"ucker coordinate is a
cluster monomial in this structure. Since each mutable face coordinate
$\Delta_J$ is irreducible, we obtain
\[
\Delta_J=p_Jz_J
\qquad (J\in\mathcal S\setminus\mathcal I),
\]
where $z_J$ is a source cluster variable and $p_J$ is a frozen Laurent
monomial.

The identity $\Phi_s(M_J)=\Delta_J$ now allows us to lift this
comparison to the stable category. Indeed, since the
extended exchange matrix is full column rank, then $qM_J$ and the reachable
object corresponding to $z_J$ have the same index in the stable category.
Since both objects are rigid, they are isomorphic by the uniqueness
of rigid objects with a prescribed index.

It remains to show that these objects belong to a common reachable
cluster-tilting object. The face labels are weakly separated, so the
stable extension spaces between the corresponding rank-one objects
vanish. Through the character formula, this vanishing implies
pairwise compatibility of the variables $z_J$ in the sense of
Fu--Gyoda. Their compatibility criterion
\cite[Theorem~4.18]{FuGyoda} therefore implies that these variables are in one
source cluster. Since they are pairwise distinct and their number
equals the mutable rank, they form its entire mutable cluster.
Consequently,
\[
\bigoplus_{J\in\mathcal S\setminus\mathcal I}qM_J
\]
is a basic reachable cluster-tilting object in the stable category.

We next incorporate the new frozen variables. Set
\[
P_{\mathcal I}=\bigoplus_{I\in\mathcal I}M_I,
\qquad
B_{\mathcal I}=\operatorname{End}_B(P_{\mathcal I})^{\mathrm{op}}.
\]
We prove that $P_{\mathcal I}$ is a $1$-tilting $B$-module.
The circular weak-order condition supplies an aligned-toggle path from
the ordinary necklace to $\mathcal I$. Short exact sequences of
rank-one lattices along this path preserve the thick subcategory
generated by the boundary modules. Weak separation is needed only at
the final necklace to obtain rigidity. In particular, the proof does
not require the intermediate necklaces to be weakly separated.

The category
$\mathcal E_{\mathcal I}=\operatorname{GP}B_{\mathcal I}$ is Frobenius,
with cluster-tilting object
\[
T_{\mathcal I}
=\bigoplus_{J\in\mathcal S}
\operatorname{Hom}_B(P_{\mathcal I},M_J).
\]
Its projective-injective summands are indexed by $\mathcal I$.
We identify its extended exchange matrix with that of $G^\rho$ and
verify compatibility with mutation. The tilting equivalence induces
\[
\mathcal F_{\mathcal I}:
D^b(\mathcal E_{\mathcal I})
\xrightarrow{\sim}D^b(\operatorname{GP}B),
\]
and sends each face object to its rank-one realization in the common
derived category.

To complete the comparison, we use relative Grothendieck groups to
retain the projective-injective contributions. If $U$ is the matched
reachable object in the source model, naturality of relative indices
gives
\[
G=\begin{pmatrix}I_r&0\\ A&C_0\end{pmatrix},
\qquad
G\widetilde B_{G^\rho}=\widetilde B_U,
\qquad C_0\in\operatorname{GL}_f(\mathbb Z).
\]
Here $A$ records the frozen factors in the mutable coordinates, while
$C_0$ is the map on the Grothendieck groups of the perfect categories.
The matrix identity includes the frozen rows and identifies the exchange
ratios. The rank-one normalization identifies the resulting ring map
with the identity of $R_\pi$. This proves the source comparison;
Pressland's theorem then gives the target comparison.

For a disconnected positroid, we prove that admissibility induces
admissible relabellings of its connected components. We combine the
component comparisons through the common Pl\"ucker degree. This
homogeneous amalgamation identifies the global coefficient group and
preserves the global exchange ratios, completing the proof in general.

Section~2 recalls the combinatorial and cluster-algebraic setup.
Section~3 develops the relative index comparison, and identifies the common reachable face cluster.
Section~4 constructs the tilted Frobenius model and proves the main
theorem, first for connected positroids and then by homogeneous
amalgamation.

\subsection*{Acknowledgements}
The author thanks Chris Fraser for his comments on the manuscript.

\subsection*{Use of generative AI}
The author formulated the problem and established the mathematical results, independently verified all arguments and references, and made all final decisions concerning the manuscript. GPT-5.6 Sol was used as a supporting tool for English-language and \LaTeX{} editing, manuscript organization, discussion of the presentation of proof arguments, understanding certain combinatorial definitions and constructions related to positroids, and preliminary checking of proofs and calculations. All mathematical arguments and conclusions were independently verified by the authors, who take full responsibility for the content of the manuscript.

\section{Positroid cluster structures from relabeled plabic graphs}
\label{sec:fsb-positroid-preliminaries}

We recall the cluster-algebraic conventions and the combinatorial data
used in the paper. We then state the results of Fraser--Sherman-Bennett
on relabeled plabic graphs and formulate the quasi-equivalence problem.
All coordinate rings in this section are over $\mathbb C$, and all frozen
variables are inverted. For a positive integer $n$, put
$[n]=\{1,\ldots,n\}$. Permutation products are composed from right to left. For an integer $a$, define $[a]_{+}=\mathrm{max}\{a,0\}$.

\subsection{Seeds, mutation, and cluster structures}
\label{subsec:cluster-preliminaries}

Let $V$ be an irreducible rational affine variety over $\mathbb C$, let
$\mathbb C[V]$ be its coordinate ring, and let $\mathbb C(V)$ be its
field of rational functions.  Put $m=\dim V$.  We use geometric-type
seeds with inverted frozen variables. We follow
\cite[Sections~2.4--2.6]{FSB}, with the exchange-matrix sign convention
specified below.
\begin{definition}[Seed]
	\label{def:fsb-seed}
	A seed of rank $0\le r\le m$ in $\mathbb C(V)$ is a pair
	\[
	\Sigma=(X,\widetilde B),
	\qquad
	X=(x_1,\ldots,x_r,z_1,\ldots,z_f),
	\qquad
	f=m-r,
	\]
	with the following properties.
	\begin{enumerate}
		\item
		The coordinates in $X$ are algebraically independent and generate
		the ambient field:
		\[
		\mathbb C(V)=\mathbb C(x_1,\ldots,x_r,z_1,\ldots,z_f).
		\]
		
		\item
		The variables $x_1,\ldots,x_r$ are mutable, whereas
		$z_1,\ldots,z_f$ are frozen and are units of $\mathbb C[V]$.
		
		\item
		$\widetilde B=(b_{ij})\in M_{m\times r}(\mathbb Z)$ has
		skew-symmetric principal part
		$B=(b_{ij})_{1\leq i,j\leq r}$.
	\end{enumerate}
	We call $X$ the extended cluster and
	\[
	\mathbb P_\Sigma
	=\langle z_1,\ldots,z_f\rangle_{\mathbb Z}
	\subset \mathbb C[V]^\times
	\]
	the coefficient group of $\Sigma$.
\end{definition}

Equivalently, $\widetilde B$ is encoded by an ice quiver whose mutable
vertices are $1,\ldots,r$ and whose frozen vertices are
$r+1,\ldots,m$. We discard arrows between frozen vertices. Our sign
convention is
\begin{equation}
	\label{eq:compact-exchange}
	b_{ij}
	=\#\{j\longrightarrow i\}
	-\#\{i\longrightarrow j\},
	\qquad
	\widehat y_k
	=X^{\widetilde B e_k}
	=\prod_{i=1}^{m}X_i^{b_{ik}}.
\end{equation}
Here, for $\alpha=(\alpha_1,\ldots,\alpha_m)\in\mathbb Z^m$, we write
$X^\alpha=\prod_iX_i^{\alpha_i}$, and
$[a]_+=\max(a,0)$ is applied componentwise to vectors. Thus the
numerator of $\widehat y_k$ records arrows leaving $k$, and its
denominator records arrows entering $k$.

For a fixed abstract quiver, this ratio is the inverse of the one in
\cite[equation~(3)]{FSB}. Our geometric convention for a plabic graph,
specified below, also reverses every arrow of the quiver used there:
we place the white endpoint on the right, following
\cite[Definition~2.13]{Pressland}. These two reversals cancel.
Consequently, for the same labelled plabic graph, our exchange ratios
agree with those of Fraser--Sherman--Bennett. This choice is compatible
with the convention $A_T=\operatorname{End}(T)^{\mathrm{op}}$ for left
modules in the categorical sections.

For a mutable index $k\in[r]$, mutation replaces $x_k$ by $x_k'$,
where
\begin{equation}
	\label{eq:fsb-cluster-exchange-relation}
	x_kx_k'
	=X^{[\widetilde B e_k]_+}
	+X^{[-\widetilde B e_k]_+},
\end{equation}
and leaves the other variables unchanged.  The extended exchange matrix
is replaced by $\mu_k(\widetilde B)=(b'_{ij})$, where
\begin{equation}
	\label{eq:fsb-prelim-matrix-mutation}
	b'_{ij}
	=
	\begin{cases}
		-b_{ij},
		& i=k\text{ or }j=k,\\[2mm]
		b_{ij}+[b_{ik}]_+[b_{kj}]_+
		-[-b_{ik}]_+[-b_{kj}]_+,
		& i\neq k\text{ and }j\neq k.
	\end{cases}
\end{equation}
Mutation is an involution.  Two seeds are \emph{mutation-equivalent} if
they are connected by a finite sequence of mutations.

\begin{definition}[Cluster algebra and cluster structure]
	\label{def:fsb-cluster-structure}
	Let $\mathcal X(\Sigma)$ be the set of mutable variables in seeds
	obtained from $\Sigma$ by finite mutation sequences. Put
	\[
	\mathcal A(\Sigma)
	=\mathbb C[z_1^{\pm1},\ldots,z_f^{\pm1}]
	[\mathcal X(\Sigma)]\subseteq\mathbb C(V).
	\]
	We say that $\Sigma$ determines a \emph{cluster structure} on $V$ if
	$\mathcal A(\Sigma)=\mathbb C[V]$.
\end{definition}

In particular, $\mathcal A$ always denotes the algebra with inverted
coefficients. When needed, the algebra without frozen inverses will be
written
\[
\mathcal A^+(\Sigma)
=\mathbb C[z_1,\ldots,z_f][\mathcal X(\Sigma)],
\qquad
\mathcal A(\Sigma)
=\mathcal A^+(\Sigma)[z_1^{-1},\ldots,z_f^{-1}].
\]
A \emph{cluster monomial} is a monomial in the variables of one seed,
with nonnegative mutable exponents and arbitrary integral frozen
exponents.

By the Laurent phenomenon \cite[Theorem~3.1]{FZI}, a seed in a cluster
structure determines an open cluster torus in $V$: its coordinate ring
is
\[
\mathbb C[V][x_1^{-1},\ldots,x_r^{-1}]
=\mathbb C[x_1^{\pm1},\ldots,x_r^{\pm1},z_1^{\pm1},\ldots,z_f^{\pm1}].
\]
When $V$ and the seed coordinates are defined over $\mathbb R$, we write
$V_{>0}$ for the locus where all extended-cluster coordinates are
positive. The exchange relations show that this locus does not depend
on the seed in the mutation class. 

\subsection{Frozen rescaling and quasi-equivalence}
\label{def:fsb-quasi}

Mutation equivalence is often too rigid when the same coordinate ring is
presented with different choices of frozen variables.  The appropriate
weaker relation keeps the exchange ratios $\widehat y_k$ fixed while
allowing a change of basis in the coefficient group and a rescaling of
mutable variables by frozen Laurent monomials.  This is the form of
Fraser's quasi-homomorphism formalism \cite{Fra16} used by
Fraser--Sherman-Bennett.

\begin{definition}[Quasi-equivalence {\cite[Definition~2.17]{FSB}}]
	\label{def:fsb-quasi-equivalent-seeds}
	Let $\Sigma=(X,\widetilde B)$ and
	$\Sigma'=(X',\widetilde B')$ be seeds of the same rank in
	$\mathbb C(V)$.  After fixing a matching of their mutable indices,
	we write $\Sigma\sim\Sigma'$ if the following conditions hold.
	\begin{enumerate}
		\item
		Their coefficient groups coincide:
		$\mathbb P_\Sigma=\mathbb P_{\Sigma'}$.
		
		\item
		Corresponding mutable variables are proportional by coefficients:
		for every $j\in[r]$,
		\[
		x_j'=p_jx_j
		\qquad\text{for some }p_j\in\mathbb P_\Sigma.
		\]
		
		\item
		Their exchange ratios agree:
		\[
		\widehat y_j(\Sigma')=\widehat y_j(\Sigma)
		\qquad(j\in[r]).
		\]
	\end{enumerate}
	The seeds are related by a \emph{quasi-cluster transformation} if there
	is a mutation sequence $\boldsymbol\mu$ such that
	$\boldsymbol\mu(\Sigma)\sim\Sigma'$.  Two cluster structures are
	\emph{quasi-equivalent} if their initial seeds are related by a
	quasi-cluster transformation.
\end{definition}

Seed quasi-equivalence is stable under simultaneous mutation:
\begin{equation}
	\label{eq:fsb-quasi-mutation-stability}
	\Sigma\sim\Sigma'
	\quad\Longleftrightarrow\quad
	\mu_k(\Sigma)\sim\mu_k(\Sigma')
	\qquad(k\in[r]).
\end{equation}
Consequently, quasi-equivalence compares entire seed patterns, not only
one pair of clusters.  Geometrically, it changes the parametrization of
a cluster torus by an integral Laurent monomial transformation but does
not change its image.  

If two quasi-equivalent seed patterns both give
cluster structures on $V$, then they determine the same set of cluster
monomials and the same totally positive part of $V$
\cite[Definition~2.17 and Lemma~2.19]{FSB}.

\begin{proposition}
	\label{prop:compact-coefficients}
	Let $\Sigma_X=(X,\widetilde B_X)$ and
	$\Sigma_\Xi=(\Xi,\widetilde B_\Xi)$ be seeds in the same ambient field,
	with extended clusters
	\[
	X=(x_1,\ldots,x_r,z_1,\ldots,z_f),
	\qquad
	\Xi=(\xi_1,\ldots,\xi_r,w_1,\ldots,w_f).
	\]
	Suppose that there exists an $(r+f)\times(r+f)$-integer matrix
	$G=\begin{pmatrix}I_r&0\\A&C\end{pmatrix}$
	such that
	$\Xi=X^G$, i.e. \[
	\xi_i=x_i\prod_{p=1}^f z_p^{a_{pi}}
	\quad(1\le i\le r),
	\qquad
	w_j=\prod_{p=1}^f z_p^{c_{pj}}
	\quad(1\le j\le f).
	\]

	For $\alpha=\binom{u}{v}\in\mathbb Z^r\oplus\mathbb Z^f$, the
	corresponding monomial substitution is
	\[
	\Xi^\alpha=\xi^uw^v=x^uz^{Au+Cv}=X^{G\alpha}.
	\]
	Then we have
	\begin{equation}
		\label{eq:fsb-matrix-quasi-criterion}
		\Sigma_X\sim\Sigma_\Xi
		\quad\Longleftrightarrow\quad
		C\in\operatorname{GL}_f(\mathbb Z)
		\quad\text{and}\quad
		G\widetilde B_\Xi=\widetilde B_X.
	\end{equation}
	If these conditions hold, then after every common mutation sequence
	$\boldsymbol\mu$,  there is an integer matrix
	\[
	G_{\boldsymbol\mu}
	=\begin{pmatrix}I_r&0\\A_{\boldsymbol\mu}&C\end{pmatrix}
	\]
	such that
	$\Xi_{\boldsymbol\mu}^{\alpha}
	=X_{\boldsymbol\mu}^{G_{\boldsymbol\mu}\alpha},
	G_{\boldsymbol\mu}\widetilde B_{\boldsymbol\mu(\Sigma_\Xi)}
	=\widetilde B_{\boldsymbol\mu(\Sigma_X)}.$
\end{proposition}

\begin{proof}

We first compare the coefficient groups
\[
\mathbb P_X=\{z^a:a\in\ZZ^f\},
\qquad
\mathbb P_\Xi=\{w^b:b\in\ZZ^f\}.
\]

Since $z_1,\ldots,z_f$ are algebraically independent, distinct
Laurent monomials in these variables are distinct elements of the
ambient field. Thus
\[
z^a=z^{a'}
\quad\Longleftrightarrow\quad
a=a',
\]
and the map
$\ZZ^f\longrightarrow\mathbb P_X,
\quad
a\mapsto z^a,$
is an isomorphism from the additive group $\ZZ^f$ to the
multiplicative group $\mathbb P_X$. For $b=(b_1,\ldots,b_f)^t$,
the expressions for the $w_j$ give
\[
w^b
=\prod_{j=1}^f
\left(\prod_{p=1}^f z_p^{c_{pj}}\right)^{b_j}
=\prod_{p=1}^f z_p^{\sum_j c_{pj}b_j}
=z^{Cb}.
\]
Consequently, under this isomorphism, $\mathbb P_\Xi$ corresponds
to the subgroup $C\ZZ^f\subseteq\ZZ^f$. In particular, we have
\[
\mathbb P_\Xi=\mathbb P_X
\quad\Longleftrightarrow\quad
C\ZZ^f=\ZZ^f
\quad\Longleftrightarrow\quad
C\in\operatorname{GL}_f(\ZZ)\Longleftrightarrow\det C=\pm1.
\]

The formulas for the $\xi_i$ already show that the matched
mutable variables differ by coefficient monomials. It remains
to compare their hatted variables. For any exponent vector
$v\in\ZZ^{r+f}$, the relation $\Xi=X^G$ gives
\[
\begin{aligned}
	\Xi^v
	&=\prod_{j=1}^{r+f}\left(X^{Ge_j}\right)^{v_j}\\
	&=X^{\sum_j v_jGe_j}
	=X^{Gv}.
\end{aligned}
\]
Applying this identity to the $k$-th exchange column yields
\[
\widehat y_{k,\Xi}
=\Xi^{\widetilde B_\Xi e_k}
=X^{G\widetilde B_\Xi e_k},
\qquad
\widehat y_{k,X}
=X^{\widetilde B_Xe_k}.
\]
All coordinates of $X$ are algebraically independent. Hence two
Laurent monomials in $X$ are equal precisely when their exponent
vectors are equal. We therefore obtain
\[
\widehat y_{k,\Xi}=\widehat y_{k,X}
\quad\Longleftrightarrow\quad
G\widetilde B_\Xi e_k=\widetilde B_Xe_k.
\]
Requiring this equality for every mutable index $k$ is equivalent
to the matrix identity
\[
G\widetilde B_\Xi=\widetilde B_X.
\]
Thus unimodularity of $C$ identifies the coefficient groups,
while the matrix identity identifies the hatted variables.
Together with the prescribed coefficient rescalings of the
mutable variables, these are exactly the required
quasi-equivalence conditions. 

Assume these conditions and fix a mutable direction $k$. Put
	\begin{equation}
		\label{eq:fsb-prelim-complete-compatibility}
		\beta=\widetilde B_\Xi e_k,
		\qquad
		\gamma=\widetilde B_Xe_k=G\beta.
	\end{equation}
	Since $\beta=[\beta]_+-[-\beta]_+$, we have
	\[
	G[\beta]_+-[\gamma]_+
	=G[-\beta]_+-[-\gamma]_+
	=\binom{0}{d_k}
	\]
	for some $d_k\in\mathbb Z^f$. The mutable component is zero because
	$G$ has upper block $(I_r\ 0)$, so the principal exchange matrices
	coincide. The exchange relations give
	\begin{equation}
		\label{eq:fsb-prelim-mutated-factor}
		\begin{aligned}
			\xi_k\xi_k'
			&=\Xi^{[\beta]_+}+\Xi^{[-\beta]_+}\\
			&=z^{d_k}\bigl(X^{[\gamma]_+}+X^{[-\gamma]_+}\bigr)
			=z^{d_k}x_kx_k',\\
			\xi_k'&=z^{d_k-a_k}x_k'.
		\end{aligned}
	\end{equation}
	Hence the new exponent matrix is
	\[
	G'=\begin{pmatrix}I_r&0\\A'&C\end{pmatrix},
	\qquad a_j'=a_j\ (j\ne k),\qquad a_k'=d_k-a_k.
	\]
	For either seed, the hatted variables mutate by
	\[
	\widehat y_k'=\widehat y_k^{-1},\qquad
	\widehat y_j'
	=\widehat y_j\widehat y_k^{[b_{kj}]_+}
	(1+\widehat y_k)^{-b_{kj}}
	\quad(j\ne k).
	\]
	Their equality therefore persists after mutation. By the algebraic independence
	of the mutated coordinates, we have
	\[
	G'\mu_k(\widetilde B_\Xi)=\mu_k(\widetilde B_X).
	\]
	By repeating this argument along the mutation sequence, we prove the desired result.
\end{proof}

\subsection{Grassmannians, Grassmann necklaces, and open positroid varieties}
\label{subsec:positroid-preliminaries}

Fix $0<k<n$. We work with loopless positroids, but do not assume
connectedness. Equivalently, all fixed points of the indexing decorated
permutations are colored white; see \cite[Section~2.1]{FSB}.

For a full-rank $k\times n$ matrix $M$ and a $k$-subset
$I\in\binom{[n]}k$, let $\Delta_I(M)$ be the corresponding maximal
minor.  The Pl\"ucker embedding identifies the Grassmannian with a
projective variety
\[
\operatorname{Gr}(k,n)
\hookrightarrow
\mathbb P^{\binom nk-1},
\qquad
x\longmapsto
\bigl(\Delta_I(x)\bigr)_{I\in\binom{[n]}k}.
\]
We write $\widehat{\operatorname{Gr}}(k,n)$ for the affine cone over
this embedding. For a collection
$\mathcal S\subseteq\binom{[n]}k$, put
$\Delta(\mathcal S)=\{\Delta_I:I\in\mathcal S\}$. For a necklace, this
notation means the set of its distinct Pl\"ucker coordinates.

A collection $\mathcal M\subseteq\binom{[n]}k$ is a \emph{positroid} if
it is the column matroid of a full-rank real matrix whose maximal
minors are all nonnegative. It is \emph{loopless} if every element of
$[n]$ belongs to at least one member of $\mathcal M$.  The
associated closed positroid variety is
\[
\Pi_{\mathcal M}
=
\left\{
x\in\operatorname{Gr}(k,n):
\Delta_J(x)=0\text{ for every }J\notin\mathcal M
\right\}.
\]

For $a\in[n]$, let $<_a$ be the cyclic order
\[
a<_a a+1<_a\cdots<_a n<_a1<_a\cdots<_a a-1.
\]
If
$I=\{i_1<_a\cdots<_a i_k\},
\qquad
J=\{j_1<_a\cdots<_a j_k\},$
we write $I\leq_aJ$ when $i_s\leq_aj_s$ for every $s$.

Let $\sigma\in S_n$, with every fixed point colored white. Its bounded
affine lift is the unique bijection $f_\sigma:\mathbb Z\to\mathbb Z$
satisfying $f_\sigma(a+n)=f_\sigma(a)+n$ and
\[
f_\sigma(a)\equiv\sigma(a)\pmod n,
\qquad
a<f_\sigma(a)\leq a+n.
\]
Thus a white fixed point is lifted by $f_\sigma(a)=a+n$.  We say that
$\sigma$ has type $(k,n)$ if
\[
\operatorname{av}(f_\sigma)
:=\frac1n\sum_{a=1}^n\bigl(f_\sigma(a)-a\bigr)=k.
\]

For a permutation $\sigma$ of type $(k,n)$, its forward Grassmann
necklace is $\mathbf I^\sigma=(I_a^\sigma)_{a\in[n]}$, where
\begin{equation}
	\label{eq:fsb-forward-necklace}
	I_a^\sigma
	=\{j\in[n]:j\leq_a\sigma^{-1}(j)\}.
\end{equation}
Equivalently,
\begin{equation}
	\label{eq:fsb-forward-necklace-recursion}
	a\in I_a^\sigma,
	\qquad
	I_{a+1}^\sigma
	=I_a^\sigma\setminus\{a\}\cup\{\sigma(a)\}.
\end{equation}
All necklace indices are read modulo $n$. The corresponding positroid
is recovered by the cyclic Gale inequalities
\cite[Section~2.1]{FSB}:
\begin{equation}
	\label{eq:fsb-positroid-gale}
	\mathcal M_\sigma
	=
	\left\{
	J\in\binom{[n]}k:
	I_a^\sigma\leq_aJ
	\text{ for every }a\in[n]
	\right\}.
\end{equation}
We write
$\mathcal I^\sigma
=\{I_a^\sigma:a\in[n]\}$
for the set of distinct necklace terms.

From now on, fix a permutation $\pi$ of type $(k,n)$.  The associated
open positroid variety is
\[
\Pi_\pi^\circ
=
\left\{
x\in\Pi_{\mathcal M_\pi}:
\Delta_I(x)\neq0
\text{ for every }I\in\mathcal I^\pi
\right\}.
\]
Let $\widehat\Pi_\pi^\circ$ be the corresponding open subvariety of the
affine cone.  Its localized homogeneous coordinate ring is
\begin{equation}
	\label{eq:fsb-positroid-coordinate-ring}
	\begin{split}
		R_\pi
		&=\mathbb C[\widehat\Pi_\pi^\circ]\\
		&=
		\left(
		\mathbb C[\widehat{\operatorname{Gr}}(k,n)]/
		(\Delta_J:J\notin\mathcal M_\pi)
		\right)
		[\Delta_I^{-1}:I\in\mathcal I^\pi],\\
		\mathscr F_\pi
		&=
		\langle\Delta_I:I\in\mathcal I^\pi\rangle_{\mathbb Z}
		\subset R_\pi^\times.
	\end{split}
\end{equation}
The group $\mathscr F_\pi$ is the coefficient group of the ordinary
target plabic seeds.

\subsection{Plabic graphs, weak separation, and Pl\"ucker seeds}
\label{subsec:plabic-seeds}

A plabic graph is a planar bicolored graph embedded in a disk, with
boundary vertices $1,\ldots,n$ of degree one, in clockwise order. We use
reducedness in the sense of \cite{Postnikov} and work
only with reduced graphs having no black lollipops. This is the graph
convention corresponding to loopless positroids.  A trip starting at a boundary vertex $a$ turns maximally
left at every white vertex and maximally right at every black vertex.
If it ends at $b$, we put $\pi_G(a)=b$ and call $\pi_G$ the trip
permutation of $G$.

\begin{example}
	\label{ex:fsb-square-graph}
	Let $G$ be the square graph below. Its internal vertices $v_1,v_3$
	are black and $v_2,v_4$ are white; boundary vertex $a$ is joined to
	$v_a$. The four trips are listed on the right.
	\begin{center}
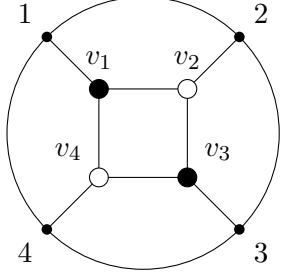
		\begin{minipage}[c]{0.45\linewidth}
			\centering
			\begin{tikzpicture}[scale=0.9,
				blackvertex/.style={circle,draw,fill=black,inner sep=2.5pt},
				whitevertex/.style={circle,draw,fill=white,inner sep=2.5pt},
				boundary/.style={circle,fill=black,inner sep=1.4pt}]
				\draw (0,0) circle (2);
				\coordinate (b1) at (135:2);
				\coordinate (b2) at (45:2);
				\coordinate (b3) at (-45:2);
				\coordinate (b4) at (-135:2);
				\coordinate (v1) at (-0.65,0.65);
				\coordinate (v2) at (0.65,0.65);
				\coordinate (v3) at (0.65,-0.65);
				\coordinate (v4) at (-0.65,-0.65);
				\draw (v1)--(v2)--(v3)--(v4)--cycle;
				\foreach \a in {1,2,3,4}{\draw (b\a)--(v\a);}
				\node[blackvertex,label=above:$v_1$] at (v1) {};
				\node[whitevertex,label=above:$v_2$] at (v2) {};
				\node[blackvertex,label=above right:$v_3$] at (v3) {};
				\node[whitevertex,label=above left:$v_4$] at (v4) {};
				\node[boundary,label=above left:$1$] at (b1) {};
				\node[boundary,label=above right:$2$] at (b2) {};
				\node[boundary,label=below right:$3$] at (b3) {};
				\node[boundary,label=below left:$4$] at (b4) {};
			\end{tikzpicture}
		\end{minipage}\hfill
		\begin{minipage}[c]{0.53\linewidth}
			\[
			\begin{array}{c|l}
				\text{source}&\text{trip}\\\hline
				1&1\to v_1\to v_4\to v_3\to3\\
				2&2\to v_2\to v_3\to v_4\to4\\
				3&3\to v_3\to v_2\to v_1\to1\\
				4&4\to v_4\to v_1\to v_2\to2
			\end{array}
			\]
		\end{minipage}
	\end{center}
	Thus $\pi_G=(1\ 3)(2\ 4)$, or $3412$.
	The graph has four boundary faces and one internal face.
\end{example}

Let $G$ be a reduced plabic graph with trip permutation $\pi$ and let $F$ be a face of $G$.
We orient every trip from its source to its target.  The source and target
labels of $F$ are
\begin{equation}
	\label{eq:compact-face-labels}
	\begin{split}
		I_F^{\rm src}
		&=
		\left\{
		a\in[n]:
		F\text{ lies to the left of the trip }
		a\rightsquigarrow\pi(a)
		\right\},\\
		I_F^{\rm tgt}
		&=
		\left\{
		\pi(a):
		F\text{ lies to the left of the trip }
		a\rightsquigarrow\pi(a)
		\right\}
		=\pi(I_F^{\rm src}).
	\end{split}
\end{equation}
Every face lies to the left of exactly $k$ trips, so these labels are
$k$-subsets.  If $F_a$ is the boundary face immediately before the
boundary vertex $a$ in clockwise order, then
\begin{equation}
	\label{eq:fsb-boundary-forward-necklace}
	I_{F_a}^{\rm tgt}=I_a^\pi.
\end{equation}
Thus the boundary target labels form the forward Grassmann necklace.

\begin{example}
	\label{ex:fsb-square-face-labels}
	We continue with the square graph in
	Example~\ref{ex:fsb-square-graph}, whose trip permutation is
	$\pi=(1\ 3)(2\ 4).$ The two drawings show the source and target labels of the same graph;
	we write $ij$ for the subset $\{i,j\}$. The face $F_a$ is immediately
	before boundary vertex $a$ in clockwise order.
	
	
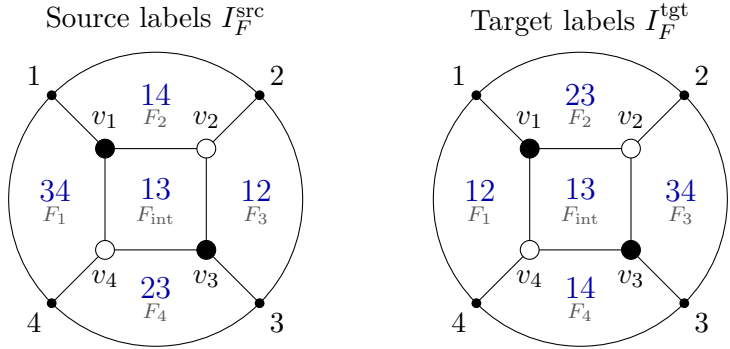
\begin{center}
		\begin{tikzpicture}[
			scale=1.08,
			blackvertex/.style={
				circle,draw,fill=black,inner sep=2.5pt
			},
			whitevertex/.style={
				circle,draw,fill=white,inner sep=2.5pt
			},
			boundary/.style={
				circle,fill=black,inner sep=1.3pt
			},
			facelabel/.style={
				font=\large,text=blue!65!black
			},
			facename/.style={
				font=\scriptsize,text=black!65
			}
			]
			\foreach \offset/\heading/\labtype/\labLeft/\labTop/\labRight/\labBottom in {
				0/Source/src/34/14/12/23,
				5.2/Target/tgt/12/23/34/14
			}{
				\begin{scope}[xshift={\offset cm}]
					\node at (0,2.18)
					{\heading\ labels $I_F^{\mathrm{\labtype}}$};
					
					\draw (0,0) circle (1.8);
					\coordinate (b1) at (135:1.8);
					\coordinate (b2) at (45:1.8);
					\coordinate (b3) at (-45:1.8);
					\coordinate (b4) at (-135:1.8);
					
					\coordinate (v1) at (-0.62,0.62);
					\coordinate (v2) at (0.62,0.62);
					\coordinate (v3) at (0.62,-0.62);
					\coordinate (v4) at (-0.62,-0.62);
					
					\draw (v1)--(v2)--(v3)--(v4)--cycle;
					\foreach \a in {1,2,3,4}{
						\draw (b\a)--(v\a);
						\node[boundary] at (b\a) {};
					}
					
					\node[above left]  at (b1) {$1$};
					\node[above right] at (b2) {$2$};
					\node[below right] at (b3) {$3$};
					\node[below left]  at (b4) {$4$};
					
					\node[blackvertex,label=above:{$v_1$}] at (v1) {};
					\node[whitevertex,label=above:{$v_2$}] at (v2) {};
					\node[blackvertex,label=below:{$v_3$}] at (v3) {};
					\node[whitevertex,label=below:{$v_4$}] at (v4) {};
					
					\node[facelabel] at (-1.23,0.10) {$\labLeft$};
					\node[facename]  at (-1.23,-0.20) {$F_1$};
					
					\node[facelabel] at (0,1.28) {$\labTop$};
					\node[facename]  at (0,0.99) {$F_2$};
					
					\node[facelabel] at (1.23,0.10) {$\labRight$};
					\node[facename]  at (1.23,-0.20) {$F_3$};
					
					\node[facelabel] at (0,-1.08) {$\labBottom$};
					\node[facename]  at (0,-1.38) {$F_4$};
					
					\node[facelabel] at (0,0.12) {$13$};
					\node[facename]  at (0,-0.20) {$F_{\mathrm{int}}$};
				\end{scope}
			}
		\end{tikzpicture}
	\end{center}
	
	For example, $F_1$ lies to the left of the trips starting at $3$
	and $4$. Since these trips end at $1$ and $2$, respectively, we have
	\[
	I_{F_1}^{\rm src}=\{3,4\},
	\qquad
	I_{F_1}^{\rm tgt}=\{\pi(3),\pi(4)\}=\{1,2\}.
	\]
	The internal face lies to the left of the trips starting at $1$
	and $3$, whose targets are $3$ and $1$. Thus
	\[
	I_{F_{\mathrm{int}}}^{\rm src}
	=I_{F_{\mathrm{int}}}^{\rm tgt}=\{1,3\}.
	\]
	Reading the target labels of $F_1,F_2,F_3,F_4$ in clockwise order
	gives the forward Grassmann necklace
	\[
	\mathcal I^\pi
	=\bigl(\{1,2\},\{2,3\},\{3,4\},\{1,4\}\bigr).
	\]
\end{example}

\bigskip

Two $k$-subsets $I,J\subseteq[n]$ are \emph{weakly separated} if there
is no cyclically ordered quadruple $a,b,c,d$ such that
\[
a,c\in I\setminus J,
\qquad
b,d\in J\setminus I.
\]
A collection is weakly separated if its members are pairwise weakly
separated.  

The target face labels of a reduced plabic graph with trip
permutation $\pi$ form a maximal weakly separated subcollection of
$\mathcal M_\pi$ containing $\mathcal I^\pi$; conversely, every such
collection is the target collection of a reduced plabic graph.  Two
such collections are connected by square moves that keep the necklace
terms fixed \cite[Theorems~2.9 and~2.11]{FSB}.

The dual ice quiver $Q(G)$ has one vertex for each face of $G$; boundary
faces are frozen and internal faces are mutable. If an edge $e$
separates faces $F$ and $F'$ and has internal endpoints of opposite
colors, we orient the dual arrow $F\to F'$ so that the white endpoint
lies on the right, and the black endpoint on the left, while crossing
$e$ from $F$ to $F'$. This is the convention of
\cite[Definition~2.13]{Pressland}. Thus our quiver is opposite to the
one attached to the same graph in \cite[Section~2.5]{FSB}.
Reversing all arrows commutes with mutation, since
$\mu_k(-\widetilde B)=-\mu_k(\widetilde B).$

It also interchanges the two exchange monomials, so the exchange
relations and the resulting cluster variables are unchanged.

For the seed, we cancel oriented $2$-cycles. The full dual quiver
will be used to construct the dimer algebra in the categorical
setting. We order the internal faces first and the boundary faces
last, and define $\widetilde B_G$ by \eqref{eq:compact-exchange}.
The target and source plabic seeds are
\begin{equation}
	\label{eq:fsb-ordinary-target-seed}
	\Sigma_G^{\rm tgt}
	=
	\left(
	(\Delta_{I_F^{\rm tgt}})_{F\in F(G)},
	\widetilde B_G
	\right),
	\qquad
	\Sigma_G^{\rm src}
	=
	\left(
	(\Delta_{I_F^{\rm src}})_{F\in F(G)},
	\widetilde B_G
	\right).
\end{equation}

The compatibility between square moves and cluster mutation is encoded
by a three-term Pl\"ucker relation.  Namely, if a square move replaces
the label $Sac$ by $Sbd$, where $a,b,c,d$ occur in cyclic order and
$S$ is a $(k-2)$-subset disjoint from them, we write
$Sab=S\cup\{a,b\}$ and similarly for the other labels. Then we have
\begin{equation}
	\label{eq:fsb-square-plucker}
	\Delta_{Sac}\Delta_{Sbd}
	=
	\Delta_{Sab}\Delta_{Scd}
	+
	\Delta_{Sad}\Delta_{Sbc}.
\end{equation}
This is the exchange relation at the corresponding mutable face.
Thus square moves induce seed mutations; contractions of monochromatic
edges and insertion or removal of bivalent vertices do not change the
seed data.

\begin{example}[The seed in $\operatorname{Gr}(2,4)$]
	\label{ex:fsb-square-seed}
	For the graph in Example~\ref{ex:fsb-square-graph}, the target labels
	of the boundary faces, starting just before vertex $1$, are
	\[
	\mathbf I_F^{\rm tgt}=\{12,23,34,14\},
	\]
	and the internal target label is $13$. With the order
	\[
	X=(\Delta_{13},\Delta_{12},\Delta_{23},\Delta_{34},\Delta_{14}),
	\]
	the corresponding ice quiver is given as follows
	
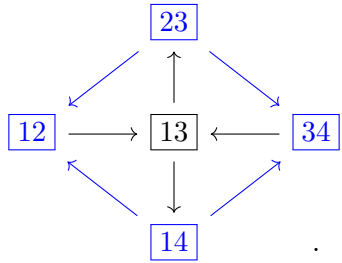
\[
	\begin{tikzcd}
		&\color{blue}\boxed{23}\arrow[dl,blue]\arrow[dr,blue]&\\
		\color{blue}\boxed{12}\arrow[r]&\boxed{13}\arrow[u]\arrow[d]&\color{blue}\boxed{34}\arrow[l]\\
		&\color{blue}\boxed{14}\arrow[ul,blue]\arrow[ur,blue]&.
	\end{tikzcd}
	\]
	
	Hence the extended exchange matrix and exchange ratio are
	\[
	\widetilde B_G=(0,-1,1,-1,1)^{\mathsf T},
	\qquad
	\widehat y_1=\frac{\Delta_{23}\Delta_{14}}{\Delta_{12}\Delta_{34}}.
	\]
	Mutation replaces $\Delta_{13}$ by $\Delta_{24}$, with
	\[
	\Delta_{13}\Delta_{24}
	=\Delta_{12}\Delta_{34}+\Delta_{14}\Delta_{23}.
	\]
	Equivalently, with the chosen exchange ratio, we have
	\[
	\Delta_{24}
	=\frac{\Delta_{12}\Delta_{34}}{\Delta_{13}}
	(1+\widehat y_1).
	\]
\end{example}

For every reduced plabic graph with trip permutation $\pi$, both
$\Sigma_G^{\rm src}$ and $\Sigma_G^{\rm tgt}$ determine cluster
structures on the same ring:
\begin{equation}
	\label{eq:fsb-source-target-cluster-rings}
	\mathcal A(\Sigma_G^{\rm src})
	=R_\pi
	=\mathcal A(\Sigma_G^{\rm tgt}).
\end{equation}
All target seeds with trip permutation $\pi$ are mutation-equivalent,
and the same is true of all source seeds.  In general, however, the
source and target seeds need not be mutation-equivalent to one another.
This is the first place where quasi-equivalence is needed
\cite[Theorem~2.14 and Remark~2.16]{FSB}.

\subsection{Relabeled plabic graphs and Grassmannlike necklaces}
\label{subsec:relabelled-plabic}

Let $G$ be a reduced plabic graph with trip permutation $\mu$ of
type $(k',n)$, and let $\rho\in S_n$. We obtain the relabeled graph
$G^\rho$ by replacing each boundary label $a$ with $\rho(a)$,
while keeping the embedded bicolored graph unchanged. This gives
a canonical identification of the dual ice quivers of $G$ and
$G^\rho$. The trip permutation and target face labels change
according to
\begin{equation}
	\label{eq:fsb-relabelled-data}
	\pi=\rho\mu\rho^{-1},
	\qquad
	I_F^{\rm tgt}(G^\rho)
	=\rho\bigl(I_F^{\rm tgt}(G)\bigr).
\end{equation}
We also introduce the insertion permutation
\begin{equation}
	\label{eq:fsb-three-permutations}
	\iota=\pi\rho=\rho\mu,
	\qquad
	\mu=\rho^{-1}\pi\rho=\rho^{-1}\iota.
\end{equation}
Thus the trip starting at the boundary vertex labeled $\rho(a)$
ends at the vertex labeled $\iota(a)$.

Write $(k,n)$ for the type of $\pi$. Relabeling preserves the
cardinality of each face label, so the target labels of $G^\rho$
are $k'$-subsets of $[n]$. However, arbitrary conjugation need not
preserve type, and hence $k'$ may differ from $k$. The admissibility
conditions introduced below will ensure that $k'=k$ and that the
relabeled face coordinates define a seed for the positroid
associated with $\pi$. Until these conditions have been verified,
we use $\Sigma_{G^\rho}^{\rm tgt}$ to denote the candidate pair
consisting of the target face coordinates and the extended
exchange matrix; see \cite[Section~3.1]{FSB}.

\begin{example}[Relabeling need not preserve type]
	\label{ex:fsb-relabeling-changes-type}
	Let $G$ consist of one white internal vertex joined to four
	boundary vertices, labeled $1,2,3,4$ in clockwise order.
	This is a reduced plabic graph whose trips run from $a$ to
	$a+1$, with indices taken modulo $4$. We relabel the boundary
	by $\rho=(2\,3)$. The two graphs are shown below; the blue
	subsets are their target face labels.
	
	
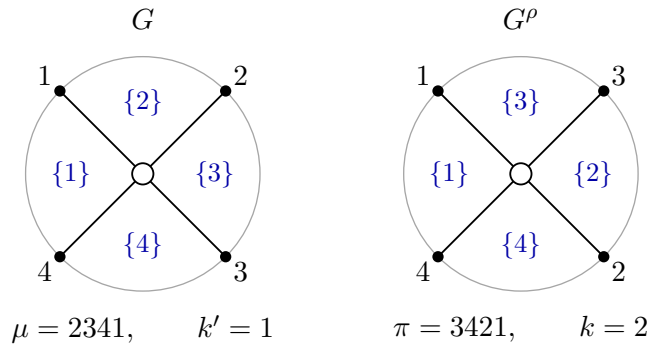
\begin{center}
		\begin{tikzpicture}[
			line width=0.7pt,
			boundary vertex/.style={
				circle,
				fill=black,
				inner sep=1.5pt
			},
			white vertex/.style={
				circle,
				draw,
				fill=white,
				minimum size=8pt,
				inner sep=0pt
			},
			face label/.style={
				text=blue!65!black,
				font=\small
			}
			]
			\foreach \shift/\graphname/\btwo/\bthree/\perm/\typeparam in {
				0/{G}/2/3/{\mu=2341}/{k'=1},
				5/{G^\rho}/3/2/{\pi=3421}/{k=2}
			}{
				\begin{scope}[xshift=\shift cm]
					\draw[gray!70,line width=0.5pt]
					(0,0) circle (1.55);
					
					\foreach \pos/\lab/\ang in {
						1/1/135,
						2/\btwo/45,
						3/\bthree/-45,
						4/4/-135
					}{
						\coordinate (b\pos) at (\ang:1.55);
						\draw (0,0) -- (b\pos);
						\node[boundary vertex] at (b\pos) {};
						\node at (\ang:1.83) {$\lab$};
					}
					
					\node[white vertex] at (0,0) {};
					
					\node[face label] at (-0.95,0)
					{$\{1\}$};
					\node[face label] at (0,0.95)
					{$\{\btwo\}$};
					\node[face label] at (0.95,0)
					{$\{\bthree\}$};
					\node[face label] at (0,-0.95)
					{$\{4\}$};
					
					\node at (0,2.05) {$\graphname$};
					\node at (0,-2.05)
					{$\perm,\qquad\typeparam$};
				\end{scope}
			}
		\end{tikzpicture}
	\end{center}
	
	The trip permutations are
	\[
	\mu=
	\begin{pmatrix}
		1&2&3&4\\
		2&3&4&1
	\end{pmatrix},
	\qquad
	\pi=\rho\mu\rho^{-1}=
	\begin{pmatrix}
		1&2&3&4\\
		3&4&2&1
	\end{pmatrix}.
	\]
	For a permutation $\sigma$ without fixed points, the first
	term of its forward Grassmann necklace is
	\[
	I_1^\sigma
	=\{a\in[4]:a<\sigma^{-1}(a)\}.
	\]
	Since $I_1^\mu=\{1\}$ and $I_1^\pi=\{1,2\}$, the permutations
	$\mu$ and $\pi$ have types $(1,4)$ and $(2,4)$, respectively.
	Thus $k'=1$ and $k=2$.
	
	Keeping the boundary faces in their original geometric order,
	starting just before the upper-left vertex, their target
	labels are
	\[
	\begin{aligned}
		G &: \quad
		\bigl(\{1\},\{2\},\{3\},\{4\}\bigr),\\
		G^\rho &: \quad
		\bigl(\{1\},\{3\},\{2\},\{4\}\bigr).
	\end{aligned}
	\]
	In particular, every face label of $G^\rho$ still has
	cardinality $k'=1$. By comparison, the forward Grassmann
	necklace of $\pi$ with respect to the standard cyclic order is
	\[
	\mathbf I^\pi
	=\bigl(\{1,2\},\{2,3\},\{3,4\},\{2,4\}\bigr).
	\]
	The distinction comes from the boundary order: in $G^\rho$,
	the labels occur clockwise as $1,3,2,4$, whereas the type of
	$\pi$ is computed using the standard order $1,2,3,4$.
	Hence arbitrary relabeling preserves the size of each face
	label but may change the type of the trip permutation.
\end{example}

\bigskip

For each \(a\in[n]\), let \(F_a\) be the boundary face
of \(G\) incident to the clockwise boundary arc from
\(a-1\) to \(a\), with indices read modulo \(n\).
We identify \(F_a\) with the same face in \(G^\rho\).
Its target label in \(G\) is \(I_a^\mu\).
Since relabelling replaces each label \(i\) by \(\rho(i)\), then
\begin{equation}
	\label{eq:fsb-relabelled-boundary-label}
	J_a
	=
	I_{F_a}^{\rm tgt}(G^\rho)
	=
	\rho(I_a^\mu).
\end{equation}
Using the Grassmann-necklace recurrence, we obtain
\begin{equation}
	\label{eq:fsb-relabelled-necklace-recursion}
	\begin{split}
		J_{a+1}
		&=\rho(I_{a+1}^\mu)\\
		&=\rho\bigl(I_a^\mu\setminus\{a\}\cup\{\mu(a)\}\bigr)\\
		&=J_a\setminus\{\rho(a)\}\cup\{\iota(a)\}.
	\end{split}
\end{equation}
This leads to the following definition.

\begin{definition}[Grassmannlike necklace]
	\label{def:compact-necklace}
	A \emph{Grassmannlike necklace} of type $(k,n)$ is an $n$-tuple
	$\mathcal J=(J_a)_{a\in[n]}$ of $k$-subsets, together with a removal
	permutation $\rho\in S_n$, such that
	\begin{equation}
		\label{eq:fsb-prelim-grassmannlike}
		\rho(a)\in J_a,
		\qquad
		J_{a+1}
		=J_a\setminus\{\rho(a)\}\cup\{\iota(a)\}
		\qquad(a\in[n]),
	\end{equation}
	Here $\iota(a)\notin J_a\setminus\{\rho(a)\}$, and the inserted
	elements define a permutation $\iota\in S_n$.  Its
	trip permutation and underlying permutation are
	\begin{equation}
		\label{eq:fsb-necklace-permutations}
		\pi=\iota\rho^{-1},
		\qquad
		\mu=\rho^{-1}\iota=\rho^{-1}\pi\rho.
	\end{equation}
	We denote the necklace by $\mathcal J_{\rho,\iota,\pi}$.
\end{definition}

Any two of $\rho,\iota,\pi$ determine the third and the necklace.
The underlying permutation $\mu$ has the type of the necklace, although
$\pi$ and $\iota$ need not. Explicitly,
\begin{equation}
	\label{eq:compact-necklace-formula}
	J_a
	=\rho(I_a^\mu)
	=
	\left\{
	j\in[n]:
	\rho^{-1}(j)\leq_a\iota^{-1}(j)
	\right\}.
\end{equation}
Conversely, every Grassmannlike necklace occurs as the clockwise list of
target labels of the boundary faces of a relabeled plabic graph
\cite[Lemmas~3.11 and~3.12]{FSB}.

Recall from \eqref{eq:fsb-positroid-coordinate-ring} that
$\mathcal I^\pi$ is the set of distinct terms of the ordinary
forward Grassmann necklace of $\pi$, and that
\[
\mathscr F_\pi
=
\left\{
\prod_{I\in\mathcal I^\pi}\Delta_I^{m_I}
:
m_I\in\mathbb Z
\right\}
\subseteq R_\pi^\times
\]
is the coefficient group of the target plabic seeds.

\begin{definition}[Unit necklace]
	\label{def:fsb-unit-necklace}
	Let $\mathcal J=(J_a)_{a\in[n]}$ be a Grassmannlike necklace
	of type $(k,n)$ with trip permutation $\pi$.
	We call $\mathcal J$ a \emph{unit necklace for $\pi$} if
	\begin{equation}
		\label{eq:fsb-unit-necklace}
		\Delta_{J_a}\in\mathscr F_\pi
		\qquad\text{for every }a\in[n].
	\end{equation}
	Thus every boundary coordinate of $\mathcal J$ is a Laurent
	monomial in the ordinary target frozen variables.
\end{definition}

\subsection{Circular weak order and admissible relabelings}
\label{subsec:circular-weak-order}

An affine permutation is a bijection $f:\mathbb Z\to\mathbb Z$ satisfying
$f(a+n)=f(a)+n$. Its average and length are
\begin{equation}
	\label{eq:fsb-affine-length}
	\begin{aligned}
		\operatorname{av}(f)&=\frac1n\sum_{a=1}^n(f(a)-a),\\
		\ell(f)&=\#\{(a,b)\in[n]\times\mathbb Z:
		a<b\text{ and }f(a)>f(b)\}.
	\end{aligned}
\end{equation}
For affine permutations $u,f$ of the same average, the right weak order
is characterized by
\begin{equation}
	\label{eq:fsb-right-weak-order}
	u\le_R f
	\quad\Longleftrightarrow\quad
	\ell(f)=\ell(u)+\ell(u^{-1}f).
\end{equation}
If $\iota,\pi$ have type $(k,n)$, with bounded affine lifts
$i=f_\iota$ and $f=f_\pi$, their circular weak order is
\begin{equation}
	\label{eq:compact-affine-order}
	\iota\le_\circ\pi\quad\Longleftrightarrow\quad i\le_R f;
\end{equation}
see \cite[Definitions~2.23--2.24]{FSB}. Put
\[
e_k(a)=a+k,\qquad
\epsilon_k(a)\equiv a+k\pmod n.
\]
The translation $e_k$ is the minimum of the right weak order in
average $k$. In particular, $\epsilon_k\le_\circ\pi$ for every $\pi$
of type $(k,n)$.

The dimension formula is
\begin{equation}
	\label{eq:fsb-positroid-dimension}
	\dim\widehat\Pi_\pi^\circ=k(n-k)+1-\ell(f_\pi).
\end{equation}
Fix $\pi$ of type $(k,n)$ and $\rho\in S_n$, and put
\[
\iota=\pi\rho,\qquad
\mu=\rho^{-1}\pi\rho=\iota^{-1}\pi\iota.
\]
If $\iota$ has type $(k,n)$ and $\iota\le_\circ\pi$, then $\mu$ also
has type $(k,n)$. Moreover, \cite[Lemma~4.17]{FSB} gives
\begin{equation}
	\label{eq:fsb-affine-admissibility}
	i\le_R f,\qquad m=i^{-1}fi,
	\qquad f=f_\pi,\quad i=f_\iota,\quad m=f_\mu.
\end{equation}

\begin{definition}[Admissible relabeling]
	\label{def:fsb-admissible}
	We call $\rho$ \emph{admissible for $\pi$} if $\iota=\pi\rho$ has
	type $(k,n)$ and
	\begin{equation}
		\label{eq:fsb-admissible-relabeling}
		\iota\le_\circ\pi,
		\qquad
		\dim\widehat\Pi_\mu^\circ=\dim\widehat\Pi_\pi^\circ.
	\end{equation}
	Under the order condition, the equality of dimensions is equivalent to
	\begin{equation}
		\label{eq:fsb-equal-length}
		\ell(m)=\ell(f).
	\end{equation}
\end{definition}

This is the terminology used in this paper for the hypotheses and
seed-size condition of \cite[Theorem~4.21]{FSB}. The order condition
is sufficient to give the correct subset size and coefficient units;
it is not asserted to be necessary for every relabeling that gives a
cluster structure.

\begin{remark}
	We explain why admissibility forces $k'=k$ in Subsection \ref{subsec:relabelled-plabic}. Let
	\[
	\mathcal J=(J_a)_{a\in[n]}
	=\mathcal J_{\rho,\iota,\pi},
	\qquad \iota=\pi\rho,
	\]
	be the boundary necklace of $G^\rho$. Since $G$ has trip
	permutation $\mu$ of type $(k',n)$, its boundary labels satisfy
	\[
	J_a=\rho(I_a^\mu),
	\qquad
	|J_a|=|I_a^\mu|=k'.
	\]
	On the other hand, admissibility requires that $\iota$ have
	type $(k,n)$ and that $\iota\le_\circ\pi$. By the Unit Necklace
	Theorem \cite[Theorem~4.14]{FSB}, the same necklace $\mathcal J$
	has type $(k,n)$, so $|J_a|=k$. Consequently, $k'=k$.
	Thus the circular weak-order condition already ensures the
	correct cardinality of every face label; the dimension condition
	is used subsequently to ensure the correct number of seed
	variables.
\end{remark}

\subsection{The Fraser--Sherman--Bennett results}
\label{subsec:fsb-results}

\begin{theorem}[Unit Necklace Theorem {
		\cite[Theorem~4.14]{FSB}}]
	\label{thm:fsb-unit-necklace}
	Let $\pi$ and $\iota$ be permutations of type $(k,n)$ satisfying
	$\iota\leq_\circ\pi$, and put $\rho=\pi^{-1}\iota$.  Then
	$\mathcal J_{\rho,\iota,\pi}$ is a Grassmannlike necklace of type
	$(k,n)$ and a unit necklace for $\pi$.  Moreover,
	\begin{equation}
		\label{eq:fsb-unit-basis}
		\Delta(\mathcal J_{\rho,\iota,\pi})
	\end{equation}
	forms a $\mathbb Z$-basis of the coefficient group $\mathscr F_\pi$,
	with repeated necklace terms counted only once.
\end{theorem}

Thus an order relation in the affine symmetric group produces a new
basis of frozen Pl\"ucker coordinates.  The remaining question is
whether the full face collection has the correct cardinality and
exchange combinatorics to be a seed.

\begin{theorem}[Fraser--Sherman--Bennett {
		\cite[Theorem~4.21 and Corollary~5.18]{FSB}}]
	\label{thm:fsb-cluster-criterion}
	Let $\pi$ and $\iota=\pi\rho$ have type $(k,n)$, and assume
	$\iota\leq_\circ\pi$. Let $G$ be a reduced plabic graph with trip
	permutation $\mu=\rho^{-1}\pi\rho$.  Let $G^\rho$ be the relabeled graph, whose
	trip permutation is $\pi$.  The following conditions are equivalent.
	\begin{enumerate}
		\item
		The target-labeled pair $\Sigma_{G^\rho}^{\rm tgt}$ is a seed in
		$\operatorname{Frac}(R_\pi)$ and
		\[
		\mathcal A(\Sigma_{G^\rho}^{\rm tgt})=R_\pi.
		\]
		
		\item
		The graph has the expected number of faces:
		\[
		|F(G^\rho)|=\dim\widehat\Pi_\pi^\circ.
		\]
		Equivalently,
		\[
		\dim\widehat\Pi_\mu^\circ
		=\dim\widehat\Pi_\pi^\circ,
		\qquad\text{or equivalently}\qquad
		\ell(f_\mu)=\ell(f_\pi).
		\]
		
		\item
		The Grassmannlike necklace of boundary target labels is weakly
		separated.  Equivalently, the complete target face collection of
		$G^\rho$ is weakly separated.
		
		\item
		The open positroid varieties
		$\widehat\Pi_\mu^\circ$ and $\widehat\Pi_\pi^\circ$ are isomorphic.
	\end{enumerate}
	When these conditions hold, positivity of the face Pl\"ucker
	coordinates of $G^\rho$ cuts out the totally positive positroid cell
	$\widehat\Pi_{\pi,>0}^\circ$.
\end{theorem}

For later use, let
\begin{equation}
	\label{eq:fsb-face-labels}
	\begin{split}
		\mathcal S
		&=
		\left\{
		\rho(I_F^{\rm tgt}):
		F\text{ is a face of }G
		\right\},\\
		\mathcal I
		&=
		\left\{
		\rho(I_F^{\rm tgt}):
		F\text{ is a boundary face of }G
		\right\},\\
		\Sigma_{G^\rho}^{\rm tgt}
		&=
		\left(
		(\Delta_J)_{J\in\mathcal S},
		\widetilde B_{G^\rho}
		\right).
	\end{split}
\end{equation}
Here $\mathcal S$ and $\mathcal I$ denote sets of distinct labels, ordered
with $\mathcal S\setminus\mathcal I$ first. If $\rho$ is admissible, then
\[
\mathcal S\subseteq\mathcal M_\pi,\qquad
\mathcal S\text{ is weakly separated},\qquad
\langle\Delta_J:J\in\mathcal I\rangle_{\mathbb Z}=\mathscr F_\pi.
\]
In particular, the mutable rank is
\[
|\mathcal S\setminus\mathcal I|
=\dim\widehat\Pi_\pi^\circ-\operatorname{rank}_{\mathbb Z}\mathscr F_\pi.
\]

\subsection{The quasi-equivalence conjecture}
\label{subsec:fsb-quasi-equivalence-conjecture}

The preceding results show that an admissible relabeling produces a new
Pl\"ucker seed for the same ring $R_\pi$ and that its boundary variables
form a different basis of the same coefficient group $\mathscr F_\pi$.
The natural comparison with the ordinary target cluster structure is
therefore quasi-equivalence rather than mutation equivalence.

\begin{conjecture}[Fraser--Sherman--Bennett
	{\cite[Conjecture~1.3]{FSB}}]
	\label{conj:fsb-quasi-equivalence}
	Let $\rho$ be admissible for $\pi$. Let $H$ and $G$ be reduced plabic
	graphs with trip permutations $\pi$ and $\rho^{-1}\pi\rho$, respectively.
	Then
	\[
	\Sigma_H^{\rm tgt}
	\quad\text{and}\quad
	\Sigma_{G^\rho}^{\rm tgt}
	\]
	are related by a quasi-cluster transformation. Equivalently, the two
	cluster structures on $R_\pi$ are quasi-equivalent.
\end{conjecture}

For any reduced plabic graph $H$ with trip permutation $\pi$, the
face-label identity gives
\begin{equation}
	\label{eq:fsb-source-as-relabelled-target}
	\Sigma_H^{\rm src}=\Sigma_{H^{\pi^{-1}}}^{\rm tgt}.
\end{equation}
However, $\rho=\pi^{-1}$ is not admissible in the above sense:
$\pi\rho=\mathrm{id}$ has type $(n,n)$ under our white-fixed-point
convention, not type $(k,n)$.

To obtain the source--target comparison as an admissible specialization,
we use the cyclic shift from \cite[Remark~6.1]{FSB}. Let \(\epsilon_k\in S_n\) be the cyclic shift defined by
\[
\epsilon_k(a)\equiv a+k\pmod n,
\qquad \epsilon_k(a)\in[n].
\]
Following \cite[Remark~6.1]{FSB}, let \(G\) be obtained
from \(H\) by replacing each boundary label \(a\) with
\(\epsilon_k^{-1}(a)\).
Put
\begin{equation}
	\label{eq:fsb-source-admissible-shift}
	\rho=\pi^{-1}\epsilon_k,
	\qquad
	\iota=\epsilon_k,
	\qquad
	\mu=\epsilon_k^{-1}\pi\epsilon_k.
\end{equation}
Then $G$ has trip permutation $\mu$ and
\[
\rho\mu\rho^{-1}=\pi,
\qquad
G^\rho=H^{\pi^{-1}}
\]
under the canonical identification of the embedded graphs. Moreover,
\[
\epsilon_k\le_\circ\pi,
\qquad
f_\mu=e_k^{-1}f_\pi e_k,
\qquad
\ell(f_\mu)=\ell(f_\pi).
\]
Thus $\rho$ is admissible and
$\Sigma_{G^\rho}^{\rm tgt}=\Sigma_H^{\rm src}$. This realizes the
Muller--Speyer source--target comparison as a special case of
Conjecture~\ref{conj:fsb-quasi-equivalence}.

Fraser--Sherman-Bennett proved the conjecture for open Schubert and
opposite open Schubert varieties \cite[Theorem~6.12]{FSB}. More generally,
if two admissible boundary necklaces are related by an aligned toggle,
then their seeds are related by a quasi-cluster transformation
\cite[Definitions~4.3, 4.5 and Theorem~6.3]{FSB}.

Pressland subsequently proved that the source and target cluster
structures quasi-coincide for every open positroid variety
\cite[Theorem~6.16]{Pressland}. His proof combines categorical models of the two cluster structures
with the interpretation of perfect matchings and twists and the
Fraser--Keller categorification of quasi-cluster morphisms.
This establishes quasi-equivalence between the source and target
structures. We address the more general problem of comparing the
standard positroid cluster structure with those arising from
arbitrary admissible relabellings.

\providecommand{\cV}{\mathcal{V}}
\providecommand{\cA}{\mathcal{A}}
\makeatletter
\@ifundefined{assumption}{%
	\theoremstyle{definition}
	\newtheorem{assumption}[theorem]{Assumption}%
}{}
\makeatother
\theoremstyle{plain}

\section{Frobenius categorifications}
\label{sec:fsb-character-input}

In this section, we first recall the Frobenius categorifications and
their cluster characters, and then extend the index and the character to
the bounded derived category. 

\subsection{Frobenius categories and cluster-tilting objects}
\label{subsec:fsb-frobenius-setup}

An \emph{exact category} is an additive category equipped with a class of
kernel--cokernel pairs satisfying Quillen's exact-category axioms. We write
such a pair as a \emph{conflation}
\[
0\longrightarrow L\xrightarrow{i}M\xrightarrow{p}N\longrightarrow0,
\]
and call $i$ an \emph{inflation} and $p$ a \emph{deflation}.
The group $\Ext^1_{\cE}(N,L)$ consists of equivalence classes of
conflations with these endpoints.
An object is projective, respectively injective, if it has the lifting
property for deflations, respectively the extension property for
inflations. The category is \emph{Frobenius} if it has enough projectives
and enough injectives and these two classes of objects coincide.

Let $\cE$ be an essentially small, $\kk$-linear, Krull--Schmidt Frobenius
exact category. Thus every object is a finite direct sum of objects with
local endomorphism rings, and this decomposition is unique up to
isomorphism and permutation. Denote by $\cP$ its full subcategory of
projective-injective objects. For objects $L,M$, let $[\cP](L,M)$ be the
subspace of morphisms factoring through an object of $\cP$.
The \emph{stable category} has the same objects as $\cE$ and morphisms
\[
\Hom_{\underline{\cE}}(L,M)
=\Hom_{\cE}(L,M)/[\cP](L,M).
\]
It is triangulated. Its suspension $\Sigma$ can be constructed by
choosing a conflation $0\to L\to P\to L^+\to0$, with $P\in\cP$,
and taking $\Sigma\underline L=\underline{L^+}$.

We assume that $\underline{\cE}$ is Hom-finite and $2$-Calabi--Yau:
its morphism spaces are finite-dimensional and there are bifunctorial
isomorphisms
\[
D\Hom_{\underline{\cE}}(L,M)
\simeq\Hom_{\underline{\cE}}(M,\Sigma^2L),
\]
where $D=\Hom_\kk(-,\kk).$
We do not assume that the morphism spaces of $\cE$ are finite-dimensional.
For $H\in\cE$, let $\add H$ be the full subcategory of direct summands
of finite direct sums of copies of $H$. An object is \emph{basic} if
its indecomposable summands are pairwise non-isomorphic, and \emph{rigid}
if $\Ext^1_{\cE}(H,H)=0$.

\begin{definition}
	\label{def:fsb-categorical-cluster-tilting}
	A morphism $V\to H$, with $V\in\cV$, is a \emph{right
		$\cV$-approximation} if every morphism from an object of $\cV$ to $H$
	factors through it. Left approximations are defined dually. The
	subcategory $\cV$ is \emph{functorially finite} if every object admits
	both approximations.
	A basic object $T\in\cE$ is \emph{cluster-tilting} if $\add T$ is
	functorially finite and
	\[
	\add T
	=\{H\in\cE\mid\Ext^1_{\cE}(T,H)=0\}
	=\{H\in\cE\mid\Ext^1_{\cE}(H,T)=0\}.
	\]
\end{definition}

Fix such an object and write
\[
T=\bigoplus_{i=1}^rT_i\oplus\bigoplus_{p=1}^sP_p,
\qquad \cP=\add\Bigl(\bigoplus_{p=1}^sP_p\Bigr),
\]
where the $T_i$ are non-projective and the $P_p$ are indecomposable
projective-injectives. Every projective-injective belongs to $\add T$
by the defining orthogonality. The summands $T_i$ are called
\emph{mutable}, and the summands $P_p$ are called \emph{frozen}.

\subsection{Cluster structures and the coefficient ring}
\label{subsec:fsb-frobenius-cluster-structure}

The existence of $T$ alone does not specify a cluster structure. We make
the following additional assumption, in the sense of
\cite[Appendix~A.2]{KW}.

\begin{assumption}[Cluster structure]
	\label{ass:fsb-frobenius-cluster}
	The cluster-tilting objects under consideration admit mutation in every
	non-projective indecomposable summand. More precisely, if
	$U=U_{\widehat k}\oplus U_k$, there is, up to isomorphism, a unique
	$U_k^*\not\simeq U_k$ such that
	\[
	\mu_kU=U_{\widehat k}\oplus U_k^*
	\]
	is cluster-tilting. There are exchange conflations
	\begin{equation}
		\label{eq:fsb-exchange-conflations}
		\begin{gathered}
			0\longrightarrow U_k^*\longrightarrow E_{U,k}^+
			\longrightarrow U_k\longrightarrow0,\\
			0\longrightarrow U_k\longrightarrow E_{U,k}^-
			\longrightarrow U_k^*\longrightarrow0.
		\end{gathered}
	\end{equation}
	Their middle terms belong to $\add U_{\widehat k}$ and have no common
	indecomposable summand. The maps to and from the endpoints are minimal
	right and left $\add U_{\widehat k}$-approximations, and
	\[
	\dim_\kk\Ext^1_{\cE}(U_k,U_k^*)
	=\dim_\kk\Ext^1_{\cE}(U_k^*,U_k)=1.
	\]
	The extended exchange matrices defined below have skew-symmetric
	principal parts and satisfy $\bB_{\mu_kU}=\mu_k(\bB_U)$.
\end{assumption}

Here minimality of $a:V\to H$ means that $a\theta=a$ implies that
$\theta\in\End(V)$ is invertible; the dual condition defines left
minimality. A cluster-tilting object is \emph{reachable from $T$} if it
is obtained from $T$ by finitely many mutations. A stable object is
\emph{reachable rigid} if it belongs to $\add\underline U$ for one such
$U$. Multiplicities and the zero object are allowed. A rigid object is
not called reachable merely because it is rigid.

For $V\in\add U$, let $[V:U_j]$ denote the multiplicity of $U_j$ in
$V$. We order the non-projective summands first and the $P_p$ last, and
put
\[
b_{jk}=[E_{U,k}^+:U_j]-[E_{U,k}^-:U_j],
\qquad \bB_U=(b_{jk})\in M_{r+s,r}(\mathbb Z).
\]
Equivalently, in the split Grothendieck group of $\add U$,
\begin{equation}
	\label{eq:fsb-exchange-column}
	\bB_Ue_k=[E_{U,k}^+]-[E_{U,k}^-].
\end{equation}
The first $r$ rows are mutable and the last $s$ rows are frozen.
The matrix mutation convention is
\[
\mu_k(\bB)_{ij}=
\begin{cases}
	-b_{ij},&i=k\text{ or }j=k,\\
	b_{ij}+[b_{ik}]_+[b_{kj}]_+-[-b_{ik}]_+[-b_{kj}]_+,&i,j\ne k,
\end{cases}
\qquad [a]_+=\max(a,0).
\]

Choose algebraically independent variables
\[
x=(x_1,\ldots,x_r,z_1,\ldots,z_s),\qquad
z_p=x_{r+p},\qquad \mathbb F_T=\kk(x_1,\ldots,x_r,z_1,\ldots,z_s).
\]
They correspond, in this order, to the indecomposable summands of $T$.
The initial seed is $\Sigma_T=(x,\bB_T)$. Mutation in a mutable
position replaces $x_k$ by $x_k'$ through
\[
x_kx_k'=x^{[\bB_Te_k]_+}+x^{[-\bB_Te_k]_+};
\]
the frozen variables remain unchanged. The same rule is used at every
seed. Denote its mutable variables by $x_{i;t}$.

\begin{definition}[The associated cluster algebras]
	\label{def:fsb-ring-R}
	The cluster algebra with non-inverted coefficients associated with
	$(\cE,\cP,T)$ is
	\[
	R=\kk[z_1,\ldots,z_s][x_{i;t}:1\le i\le r,\ t\text{ a seed}]
	\subseteq\mathbb F_T.
	\]
	Its localization at the frozen variables is
	\[
	R_{\mathrm{loc}}=R[z_1^{-1},\ldots,z_s^{-1}].
	\]
	The coefficient group is
	$\mathbb P_T=\{z^a:a\in\mathbb Z^s\}$, where
	$z^a=\prod_pz_p^{a_p}$. A cluster monomial with frozen Laurent factors
	is a monomial $z^a\prod_i x_{i;t}^{d_i}$ for one seed $t$, with
	$a\in\mathbb Z^s$ and $d_i\in\mathbb N$.
	The upper cluster algebra with inverted coefficients is
	\[
	\mathcal U_T=
	\bigcap_t\kk[x_{1;t}^{\pm1},\ldots,x_{r;t}^{\pm1},
	z_1^{\pm1},\ldots,z_s^{\pm1}]
	\subseteq\mathbb F_T.
	\]
\end{definition}
By the Laurent phenomenon, we have $R_{\mathrm{loc}}\subseteq\mathcal U_T$.

\subsection{Relative Grothendieck groups and indices}
\label{subsec:fsb-relative-groups}

Put $\cD=\Db(\cE)$ and $\cN=K^b(\cP)$. Here $\Db(\cE)$ is the
bounded homotopy category of $\cE$ modulo its acyclic complexes for the
chosen exact structure. Explicitly, an acyclic complex has factorizations
$d^i:X^i\twoheadrightarrow Z^{i+1}\rightarrowtail X^{i+1}$
with conflations $0\to Z^i\to X^i\to Z^{i+1}\to0$.
We identify $K^b(\cP)$ with its fully faithful
image in $\cD$. An object of this subcategory is called \emph{perfect}.
It is a thick subcategory: it is closed under shifts, cones and direct
summands. The stable realization of the singularity quotient gives
\[
q:\cD\longrightarrow\cD/\cN\xrightarrow{\sim}\underline{\cE}.
\]
We use $[1]$ for the derived shift and $\Sigma$ for the stable
suspension; thus $q(X[1])\simeq\Sigma qX$. For $H\in\cE$, $qH$ is
its usual stable image.

\begin{definition}[Split and relative Grothendieck groups]
	\label{def:fsb-relative-grothendieck}
	For an essentially small additive category $\cA$, its split
	Grothendieck group is
	\[
	K_0^{\mathrm{sp}}(\cA)=
	\mathbb Z\{[M]:M\in\cA\}/
	\langle[M\oplus N]-[M]-[N]\rangle,
	\]
	where generators are indexed by isomorphism classes.
	For a triangulated category $\cA$ and a thick subcategory $\cN$, put
	\[
	K(\cA,\cN)=K_0^{\mathrm{sp}}(\cA)/
	\langle[P]-[X]+[Y]:P\to X\to Y\to P[1],\ P\in\cN\rangle.
	\]
	The image of $[X]$ in this quotient is denoted by $[X]_{\mathrm{rel}}$.
	In the Frobenius setting, we write
	$K_{\cE,\cP}=K(\Db(\cE),K^b(\cP))$.
\end{definition}

Only triangles with a perfect first term are imposed as relations.
For $N\in\cN$, the triangle $N\to0\to N[1]\to N[1]$ gives
$[N[1]]_{\mathrm{rel}}=-[N]_{\mathrm{rel}}$. Rotation therefore
also gives the relation $[A]_{\mathrm{rel}}-[B]_{\mathrm{rel}}
+[N]_{\mathrm{rel}}=0$ for a triangle $A\to B\to N\to A[1]$
with perfect third term. A perfect middle term alone does not give
the usual triangle relation. Thus this group is not the usual
triangulated Grothendieck group. Since $\cE$ is Krull--Schmidt,
$K_0^{\mathrm{sp}}(\cE)$ is free on its indecomposable isomorphism
classes.

\begin{proposition}[{\cite[Proposition~A.3]{KW}}]
	\label{prop:relative-split-grothendieck}
	The inclusion $\cE\to\Db(\cE)$ induces an isomorphism
	\[
	\phi:K_0^{\mathrm{sp}}(\cE)\xrightarrow{\sim}K_{\cE,\cP}.
	\]
\end{proposition}

For $X\in\cD$, define
\begin{equation}
	\label{eq:fsb-kappa}
	\kappa(X)=\phi^{-1}([X]_{\mathrm{rel}})
	\in K_0^{\mathrm{sp}}(\cE).
\end{equation}
Thus $\kappa(H)=[H]$ for $H\in\cE$. If $P^\bullet$ is a bounded
complex of projective-injectives, the defining relations give
\[
\kappa(P^\bullet)=\sum_j(-1)^j[P^j],
\qquad \kappa(P[1])=-[P]\quad(P\in\cP).
\]

\begin{definition}
	\label{def:fsb-T-presentation}
	A \emph{$T$-presentation} of $H\in\cE$ is a conflation
	\begin{equation}
		\label{eq:fsb-T-presentation}
		0\longrightarrow T_1^H\longrightarrow T_0^H\longrightarrow H
		\longrightarrow0,\qquad T_0^H,T_1^H\in\add T.
	\end{equation}
	The index of $H$ is
	\begin{equation}
		\label{eq:fsb-index}
		\ind_T H=[T_0^H]-[T_1^H]\in L_T:=K_0^{\mathrm{sp}}(\add T).
	\end{equation}
\end{definition}

Every object admits such a presentation. Indeed, add a projective
deflation to a right $\add T$-approximation to obtain a deflation
$T_0^H\to H$. For its kernel $K$, the approximation property and
$\Ext^1(T,T_0^H)=0$ give $\Ext^1(T,K)=0$, so $K\in\add T$.

The index does not depend on the choice of presentation.
Indeed, let \(P\) be the pullback of the deflations
\(T_0^H\to H\) and \(\widetilde T_0^H\to H\).
The induced conflations
\[
T_1^H\longrightarrow P\longrightarrow\widetilde T_0^H,
\qquad
\widetilde T_1^H\longrightarrow P\longrightarrow T_0^H
\]
split because \(T\) is rigid. Thus
\[
T_1^H\oplus\widetilde T_0^H
\simeq P
\simeq\widetilde T_1^H\oplus T_0^H,
\]
and consequently
\[
[T_0^H]-[T_1^H]
=
[\widetilde T_0^H]-[\widetilde T_1^H]
\quad\text{in }K_0(\operatorname{add}T).
\]
Direct sums of presentations show that $\ind_T$ extends to a
homomorphism $K_0^{\mathrm{sp}}(\cE)\to L_T$.
We identify
\[
L_T=\bigoplus_{i=1}^r\mathbb Z[T_i]
\oplus\bigoplus_{p=1}^s\mathbb Z[P_p]\simeq\mathbb Z^{r+s}.
\]
For $\alpha=\sum_i\alpha_i[T_i]+\sum_p\alpha_{r+p}[P_p]$, set
$x^\alpha=\prod_i x_i^{\alpha_i}\prod_pz_p^{\alpha_{r+p}}$.

\begin{definition}[Localized index]
	\label{def:localized-index}
	For $X\in\cD$, its localized index is
	\begin{equation}
		\label{eq:fsb-localized-index}
		I_T(X)=\ind_T(\kappa(X))\in L_T.
	\end{equation}
	Thus $I_T$ is the composite
	$K_{\cE,\cP}\xrightarrow{\phi^{-1}}K_0^{\mathrm{sp}}(\cE)
	\xrightarrow{\ind_T}L_T$, evaluated at $[X]_{\mathrm{rel}}$.
\end{definition}

\begin{lemma}
	\label{lem:fsb-projective-correction}
	Let $X\in\cD$ and $H\in\cE$ satisfy $qX\simeq qH$.
	There are unique integers $a_p$ such that
	\begin{equation}
		\label{eq:fsb-boundary-localization}
		\kappa(X)=[H]+\sum_{p=1}^s a_p[P_p].
	\end{equation}
	Consequently,
	\begin{equation}
		\label{eq:fsb-localized-index-projective-correction}
		I_T(X)=\ind_T H+\sum_{p=1}^s a_p[P_p].
	\end{equation}
\end{lemma}

\begin{proof}
	Fix an isomorphism \(\alpha:qX\xrightarrow{\sim}qH\).
	By the calculus of fractions for the Verdier quotient, we can
	represent \(\alpha\) by a roof
	\begin{equation*}
		X\xleftarrow{s}Z\xrightarrow{t}H,
		\qquad
		\alpha=q(t)q(s)^{-1},
	\end{equation*}
	where \(N_s:=\Cone(s)\) belongs to \(\cN\).
	Since both \(\alpha\) and \(q(s)\) are invertible, so is \(q(t)\).
	Thus \(q\Cone(t)=0\). As \(\cN\) is thick, the kernel of the
	quotient functor is precisely \(\cN\), and hence
	\(N_t:=\Cone(t)\) also belongs to \(\cN\).
	
	Rotating the cone triangles gives
	\begin{equation*}
		\begin{gathered}
			N_s[-1]\longrightarrow Z\longrightarrow X
			\longrightarrow N_s,\\
			N_t[-1]\longrightarrow Z\longrightarrow H
			\longrightarrow N_t.
		\end{gathered}
	\end{equation*}
	Their first terms are perfect, so the defining relations of
	the relative Grothendieck group apply. Using
	\(
	[N[-1]]_{\mathrm{rel}}=-[N]_{\mathrm{rel}}
	\)
	for \(N\in\cN\), we obtain
	\begin{equation*}
		\begin{aligned}
			[X]_{\mathrm{rel}}-[Z]_{\mathrm{rel}}
			&=[N_s]_{\mathrm{rel}},\\
			[H]_{\mathrm{rel}}-[Z]_{\mathrm{rel}}
			&=[N_t]_{\mathrm{rel}}.
		\end{aligned}
	\end{equation*}
	Subtracting and applying \(\phi^{-1}\), we find
	\begin{equation*}
		\kappa(X)-[H]=\kappa(N_s)-\kappa(N_t).
	\end{equation*}
	
	Choose bounded complexes \(P_s^\bullet\) and \(P_t^\bullet\)
	of projective-injective objects representing \(N_s\) and \(N_t\).
	The Euler formulas give
	\begin{equation*}
		\kappa(X)-[H]
		=
		\sum_j(-1)^j\bigl([P_s^j]-[P_t^j]\bigr).
	\end{equation*}
	Each term belongs to the subgroup generated by
	\([P_1],\ldots,[P_s]\). Therefore there are integers \(a_p\)
	such that
	\begin{equation*}
		\kappa(X)=[H]+\sum_{p=1}^s a_p[P_p].
	\end{equation*}
	Since \(\cE\) is Krull--Schmidt, its split Grothendieck group
	is free on the isomorphism classes of indecomposable objects.
	In particular, \([P_1],\ldots,[P_s]\) are linearly independent,
	so the integers \(a_p\) are unique for the fixed objects
	\(X\) and \(H\). Finally, applying \(\ind_T\) and using
	\(\ind_T P_p=[P_p]\), we obtain
	\begin{equation*}
		I_T(X)=\ind_T H+\sum_{p=1}^s a_p[P_p],
	\end{equation*}
	as required.
\end{proof}

Every $qX$ has a representative $H\in\cE$, since $q$ is essentially
surjective. Thus the lemma applies to every object in $\cD$, but the
integers $a_p$ depend on the chosen representative $H$.

\subsection{Cluster characters and their localization}
\label{subsec:fsb-character-definitions}

Put
\[
\Gamma_T=\End_{\underline{\cE}}(qT),\qquad
M_T(Y)=\Hom_{\underline{\cE}}(qT,\Sigma Y)
\quad(Y\in\underline{\cE}).
\]
The algebra $\Gamma_T$ and the module $M_T(Y)$ are finite-dimensional.
We regard $M_T(Y)$ as a \emph{right} $\Gamma_T$-module by
precomposition: $m\cdot a=m\circ a$. Equivalently, it is a left
$\Gamma_T^{\mathrm{op}}$-module. This convention agrees with the
use of left modules over boundary orders below. Let $\epsilon_i$ be the
primitive idempotent for $qT_i$. Its dimension vector is
\[
\underline{\dim}M=(\dim_\kk M\epsilon_1,\ldots,
\dim_\kk M\epsilon_r)\in\mathbb N^r.
\]
There are no components for $P_p$, since $qP_p=0$.

\begin{definition}[Submodule Grassmannians and $F$-polynomials]
	\label{def:fsb-Grassmannian-polynomial}
	For a finite-dimensional right $\Gamma_T$-module $M$ and
	$e\in\mathbb N^r$, let
	\[
	\Gr_e(M)=\{N\subseteq M:N\text{ is a right }\Gamma_T
	\text{-submodule},\ \underline{\dim}N=e\}.
	\]
	This is a closed subvariety of
	$\prod_i\Gr(e_i,M\epsilon_i)$, and is empty when a component of $e$
	is too large. We use the compactly supported topological Euler
	characteristic
	$\chi(V)=\sum_j(-1)^j\dim_\mathbb Q H_c^j(V,\mathbb Q)$.
	For $y=(y_1,\ldots,y_r)$, put
	\begin{equation}
		\label{eq:fsb-F-polynomial}
		F_T^Y(y)=\sum_{e\in\mathbb N^r}\chi(\Gr_e(M_T(Y)))y^e,
		\qquad y^e=\prod_i y_i^{e_i}.
	\end{equation}
	For $H\in\cE$, we abbreviate $F_T^{qH}$ to $F_T^H$.
\end{definition}

The sum is finite, and $F_T^Y(0)=1$ because the zero submodule is
unique. The coefficient of $y^{\underline{\dim}M_T(Y)}$ is also one.
For each mutable $j$, define the hatted variable
\begin{equation}
	\label{eq:fsb-hat-y}
	\widehat y_j=x^{\bB_Te_j}.
\end{equation}

\begin{definition}[Cluster characters]
	\label{def:fsb-cluster-character-axioms}
	Let $S$ be a commutative integral domain. A \emph{cluster character}
	$\Phi:\operatorname{obj}(\cE)\to S$ is invariant under isomorphism,
	satisfies $\Phi(0)=1$ and
	$\Phi(L\oplus M)=\Phi(L)\Phi(M)$, and satisfies
	\[
	\Phi(L)\Phi(M)=\Phi(E)+\Phi(E')
	\]
	whenever $\dim_\kk\Ext^1_{\cE}(L,M)=1$ and
	$0\to L\to E\to M\to0$, $0\to M\to E'\to L\to0$ are non-split.
	The character associated with $T$ is given, in the realizations used
	here, by
	\begin{equation}
		\label{eq:fsb-full-character}
		\CC_T(H)=x^{\ind_T H}F_T^H(\widehat y)
		\in\kk[x_1^{\pm1},\ldots,x_r^{\pm1},z_1^{\pm1},\ldots,z_s^{\pm1}].
	\end{equation}
\end{definition}

\begin{assumption}[Character realization]
	\label{ass:fsb-character-realization}
	For the Frobenius categorification and cluster-tilting objects used below,
	\eqref{eq:fsb-full-character} is the standard cluster character and
	satisfies Definition~\ref{def:fsb-cluster-character-axioms}.
	Moreover, it gives the usual correspondence between reachable
	indecomposable non-projective objects and mutable cluster variables,
	and between reachable cluster-tilting objects and seeds.
\end{assumption}

This is the character input supplied by Frobenius categorification;
see \cite[Appendix~A.2]{KW}. In the Hom-infinite setting it requires
an appropriate character realization, not just Hom-finiteness of the
stable category. It holds for the standard positroid models recalled
below. We keep it separate from the elementary definition of a
cluster-tilting object. In particular, one has
\[
\CC_T(T_i)=x_i,\qquad \CC_T(P_p)=z_p.
\]
Our signs can be checked on an exchange pair:
$\ind_T T_k^*=[E_k^-]-[T_k]$ and $F_T^{T_k^*}=1+y_k$, whence
\[
\CC_T(T_k^*)=
\frac{x^{[E_k^-]}+x^{[E_k^+]}}{x_k}.
\]

\begin{proposition}[Localized character]
	\label{prop:fsb-localized-character}
	For $X\in\cD$, put
	\begin{equation}
		\label{eq:fsb-derived-character-formula}
		\CC_{T,\mathrm{loc}}(X)=x^{I_T(X)}F_T^{qX}(\widehat y).
	\end{equation}
	This extends $\CC_T$. If $qX\simeq qH$ and
	\eqref{eq:fsb-boundary-localization} holds, then
	\begin{equation}
		\label{eq:fsb-localized-character-factor}
		\CC_{T,\mathrm{loc}}(X)=\CC_T(H)\prod_{p=1}^s z_p^{a_p}.
	\end{equation}
	It is the localization of the character in
	\cite[Theorem~A.4]{KW}.
\end{proposition}

\begin{proof}
	The stable isomorphism identifies $M_T(qX)$ with $M_T(qH)$.
	Substitute the index identity of
	Lemma~\ref{lem:fsb-projective-correction} into
	\eqref{eq:fsb-derived-character-formula}. The resulting formula is
	the extension formula of \cite[Theorem~A.4 and Remark~A.5]{KW}.
\end{proof}

More generally, let $\Phi:\cE\to S$ be a cluster character with
$\Phi(P_p)\ne0$, and set
$S_{\mathrm{loc}}=S[\Phi(P_1)^{-1},\ldots,\Phi(P_s)^{-1}]$.
For $qX\simeq qH$ and the integers in
Lemma~\ref{lem:fsb-projective-correction}, its localized extension is
\begin{equation}
	\label{eq:fsb-general-localized-character}
	\Phi_{\mathrm{loc}}(X)=\Phi(H)\prod_p\Phi(P_p)^{a_p}.
\end{equation}
This expression is independent of $H$. Indeed, an alternative expression
$\kappa(X)=[H']+\sum_pa_p'[P_p]$ has the same non-projective
indecomposable multiplicities. After decomposing $H$ and $H'$ into
their common non-projective part and their projective parts, the
equality of the remaining projective coefficients and direct-sum
multiplicativity of $\Phi$ give the same value. This argument only
inverts the values on projective-injectives.

The extension is invariant under isomorphism and satisfies
\[
\Phi_{\mathrm{loc}}(0)=1,\qquad
\Phi_{\mathrm{loc}}(X\oplus Y)
=\Phi_{\mathrm{loc}}(X)\Phi_{\mathrm{loc}}(Y).
\]
For a triangle $P\to X\to Y\to P[1]$ with $P$ in  $K^b(\cP)$, it satisfies
\[
\Phi_{\mathrm{loc}}(X)=\Phi_{\mathrm{loc}}(P)\Phi_{\mathrm{loc}}(Y),
\qquad
\Phi_{\mathrm{loc}}(P^\bullet)=\prod_j\Phi(P^j)^{(-1)^j}.
\]
These identities follow from \cite[Theorem~A.4 and Remark~A.5]{KW}. In particular,
\[
\CC_{T,\mathrm{loc}}(P_p[1])=z_p^{-1}.
\]
For a non-projective object the shift has a different description:
if $0\to H\to P\to H^+\to0$ is a conflation with $P\in\cP$,
then the rotated triangle implies
\[
\Phi_{\mathrm{loc}}(H[1])=\Phi(H^+)\Phi(P)^{-1}.
\]

\subsection{Triangle defects and presentations}

Let $\mathcal A$ be an algebraic triangulated category, let
$\mathcal N\subseteq\mathcal A$ be a thick subcategory, and let
$q:\mathcal A\to\mathcal A/\mathcal N$ be the quotient functor.
We use $K(\mathcal A,\mathcal N)$ as in
Definition~\ref{def:fsb-relative-grothendieck}. For a triangle
\[
\Delta=(A\longrightarrow B\longrightarrow D\longrightarrow A[1]),
\]
we put
\[
\partial\Delta
=[A]_{\mathrm{rel}}-[B]_{\mathrm{rel}}+[D]_{\mathrm{rel}}
\in K(\mathcal A,\mathcal N).
\]

\begin{lemma}
	\label{lem:fsb-defect-invariance}
	Let $\Delta$ and $\Delta'$ be triangles in $\mathcal A$. If
	$q\Delta\simeq q\Delta'$, then
	$\partial\Delta=\partial\Delta'$.
\end{lemma}

\begin{proof}
	Put
	\[
	S=\{s\mid\Cone(s)\in\mathcal N\}.
	\]
	The calculus of fractions gives
	\[
	q(s)\text{ is invertible}\iff s\in S,
	\qquad
	q(h)=0\implies hw=0\text{ for some }w\in S.
	\]
	Let $f:A\to B$ and $f':A'\to B'$ be the first morphisms of
	$\Delta$ and $\Delta'$. Represent the first two components of the
	isomorphism $q\Delta\simeq q\Delta'$ by roofs
	\[
	A\xleftarrow{s}A_0\xrightarrow{a}A',
	\qquad
	B\xleftarrow{t}B_0\xrightarrow{b}B',
	\qquad s,a,t,b\in S.
	\]
	Choose $A_0\xleftarrow{u}A_1\xrightarrow{v}B_0$, with $u\in S$,
	representing $(qt)^{-1}(qf)(qs)$. We have
	\[
	q(tv-fsu)=0,
	\qquad
	q(bv-f'au)=0.
	\]
	Precompose first with an element of $S$ killing the first difference,
	and then with one killing the second. The first equality remains true
	after the second precomposition. Replacing $A_1,u,v$ by this common
	refinement, both equalities hold in $\mathcal A$.
	We obtain the commutative diagram
	\begin{equation}
		\label{eq:fsb-defect-common-arrow}
		\begin{tikzcd}[column sep=large,row sep=small]
			A\arrow[r,"f"] & B\\
			A_1\arrow[u,"su"']\arrow[r,"v"]\arrow[d,"au"']
			&B_0\arrow[u,"t"]\arrow[d,"b"]\\
			A'\arrow[r,"f'"]&B',
		\end{tikzcd}
		\qquad su,au,t,b\in S.
	\end{equation}
	
	Let $\Delta_0$ be a cone triangle of $v$. We compare it with
	$\Delta$ using the upper square. In a dg enhancement, choose
	representatives $\ell=su$, $r=t$, and a homotopy
	$rv-f\ell=dh$, where $|h|=-1$. The induced cone map is
	\[
	c_h=
	\begin{pmatrix}
		r&h\\
		0&\ell[1]
	\end{pmatrix}:
	\Cone(v)\longrightarrow\Cone(f).
	\]
	Thus we have a morphism of triangles
	\[
	\begin{tikzcd}[column sep=large,row sep=small]
		A_1\arrow[r,"v"]\arrow[d,"\ell"']
		&B_0\arrow[r]\arrow[d,"r"]
		&D_0\arrow[r]\arrow[d,"c_h"]
		&A_1[1]\arrow[d,"{\ell[1]}"]\\
		A\arrow[r,"f"]&B\arrow[r]&D\arrow[r]&A[1].
	\end{tikzcd}
	\]
	For this cone map, the dg cone construction gives a triangle
	\begin{equation}
		\label{eq:fsb-defect-three-cones}
		\begin{gathered}
			N_A=\Cone(\ell),\qquad
			N_B=\Cone(r),\qquad
			N_D=\Cone(c_h),\\
			N_A\longrightarrow N_B\longrightarrow N_D\longrightarrow N_A[1].
		\end{gathered}
	\end{equation}
	Here $N_A,N_B\in\mathcal N$ because $\ell,r\in S$, and
	$N_D\in\mathcal N$ follows from this triangle and thickness.
	Rotating the vertical
	cone triangles and using $[N[-1]]_{\mathrm{rel}}=-[N]_{\mathrm{rel}}$
	for $N\in\mathcal N$, the defining relations give
	\[
	\begin{aligned}
		\partial\Delta-\partial\Delta_0
		&=([A]_{\mathrm{rel}}-[A_1]_{\mathrm{rel}})
		-([B]_{\mathrm{rel}}-[B_0]_{\mathrm{rel}})
		+([D]_{\mathrm{rel}}-[D_0]_{\mathrm{rel}})\\
		&=[N_A]_{\mathrm{rel}}-[N_B]_{\mathrm{rel}}+[N_D]_{\mathrm{rel}}
		=0.
	\end{aligned}
	\]
	Applying the same argument to the lower square gives
	$\partial\Delta'=\partial\Delta_0$.
\end{proof}


We return to $\cD=\Db(\cE)$ and $\cN=K^b(\cP)$.

\begin{lemma}
	\label{lem:fsb-presentation-index}
	Let $A\to B\to X\to A[1]$ be a triangle in $\cD$ such that
	$qA,qB\in\add(qT)$. Then
	\[
	I_T(X)=I_T(B)-I_T(A).
	\]
\end{lemma}

\begin{proof}
	Choose $A_0,B_0\in\add T$ with $qA_0\simeq qA$ and
	$qB_0\simeq qB$. Since $\cE\to\underline{\cE}$ is full, the
	first morphism of the quotient triangle lifts to a morphism
	$a:A_0\to B_0$. Choose an inflation $i:A_0\to P$ with $P\in\cP$.
	The conflation
	\[
	0\longrightarrow A_0\xrightarrow{(a,i)}B_0\oplus P
	\longrightarrow Y\longrightarrow0
	\]
	is a $T$-presentation whose stable image is isomorphic to the
	quotient triangle. Hence
	\[
	\ind_TY=[B_0]+[P]-[A_0].
	\]
	By Lemma~\ref{lem:fsb-defect-invariance},
	\[
	\kappa(A)-\kappa(B)+\kappa(X)
	=[A_0]-[B_0]-[P]+[Y].
	\]
	Applying $\ind_T$ gives
	\[
	I_T(A)-I_T(B)+I_T(X)
	=[A_0]-[B_0]-[P]+\ind_TY=0.
	\qedhere
	\]
\end{proof}

\subsection{Mutations}
A quiver with potential $(Q,W)$
consists of a finite quiver and a (possibly completed) linear combination
of cyclic paths, considered up to cyclic equivalence. Its Jacobian
algebra is the completed path algebra modulo the closed ideal generated
by the cyclic derivatives $\partial_aW$. It is \emph{Jacobi-finite}
if that algebra is finite-dimensional. The potential is
\emph{non-degenerate} if every iterated reduced mutation in the
permitted vertices has no oriented $2$-cycles.

\begin{assumption}[Mutation-compatible QP realization]
	\label{ass:fsb-QP-realization}
	The stable category $\overline{\mathcal{E}}$ admits a realization by a generalized cluster
	category of a Jacobi-finite non-degenerate quiver with potential.
	The realizations at the cluster-tilting objects under consideration
	are compatible with QP mutation. In particular, the decorated
	representation associated with a stable object $Y$ has module
	$M_U(Y)=\Hom(qU,\Sigma Y)$, and categorical mutation corresponds to
	mutation of decorated representations with this convention.
\end{assumption}

Here a decorated representation is a pair $(M,V)$ consisting of a
finite-dimensional Jacobian module $M$ and a finite-dimensional
vertex-graded vector space $V$. 

For $T'=\mu_kT$, let $x'$ denote the coordinates of its seed and let
$\mu_k^x:\mathbb F_{T'}\to\mathbb F_T$ be the field isomorphism
which sends $x_k'$ to its exchange expression and fixes the other
coordinates. We use $\mu_k^y$ for the associated substitution in the
independent coefficient variables, displayed in the proof below.

\begin{lemma}
	\label{lem:fsb-seed-covariance}
	Assume Assumption~\ref{ass:fsb-QP-realization}, and let
	$T'=\mu_kT$. For each $X\in\Db(\cE)$, we have
	\[
	\mu_k^x\bigl(\CC_{T',\mathrm{loc}}(X)\bigr)
	=\CC_{T,\mathrm{loc}}(X).
	\]
\end{lemma}

\begin{proof}
	First let $H\in\cE$. Put $L=T_k$, $L^*=T_k^*$, and choose a
	minimal $T$-presentation
	\[
	0\longrightarrow T_1^0\oplus L^\beta\xrightarrow{f}
	T_0^0\oplus L^\alpha\longrightarrow H\longrightarrow0,
	\qquad f\in\rad,
	\]
	where neither $T_1^0$ nor $T_0^0$ has a summand isomorphic to $L$.
	Such a presentation is obtained by cancelling isomorphism blocks
	from any $T$-presentation.
	Write
	\[
	g=\ind_TH,\qquad g'=\ind_{T'}H,\qquad
	c^\pm=[\pm\bB_Te_k]_+=[E_k^\pm].
	\]
	In particular, $g_k=\alpha-\beta$.
	
	Put $V_0=T_0^0\oplus(E_k^+)^\alpha$. The first exchange
	conflation gives a deflation
	$V_0\twoheadrightarrow T_0^0\oplus L^\alpha$ with kernel
	$(L^*)^\alpha$. Pull it back along $f$, and denote the pullback
	by $W$. The resulting conflations are
	\[
	\begin{gathered}
		0\longrightarrow W\longrightarrow V_0\longrightarrow H
		\longrightarrow0,\\
		0\longrightarrow(L^*)^\alpha\longrightarrow W
		\longrightarrow T_1^0\oplus L^\beta\longrightarrow0.
	\end{gathered}
	\]
	The second conflation splits. Its extension class is the pullback
	of the exchange class along $f$. Its restriction to $T_1^0$ is zero
	because $T_1^0,L^*\in\add T'$. Its restriction to $L^\beta$
	is zero because the $L$-blocks of $f$ lie in
	$\rad\End_{\cE}(L)$, which annihilates the one-dimensional
	space $\Ext^1_{\cE}(L,L^*)$. We may therefore choose
	\[
	W\simeq(L^*)^\alpha\oplus T_1^0\oplus L^\beta.
	\]
	The second exchange conflation now gives an inflation
	\[
	W\longrightarrow
	W'=(L^*)^\alpha\oplus T_1^0\oplus(E_k^-)^\beta
	\]
	with cokernel $(L^*)^\beta$. Push out the inflation $W\to V_0$
	along this map. The pushout $V$ fits into conflations
	\[
	\begin{gathered}
		0\longrightarrow W'\longrightarrow V\longrightarrow H
		\longrightarrow0,\\
		0\longrightarrow V_0\longrightarrow V
		\longrightarrow(L^*)^\beta\longrightarrow0.
	\end{gathered}
	\]
	The second conflation splits since $V_0,(L^*)^\beta\in\add T'$.
	Thus $V\simeq V_0\oplus(L^*)^\beta$, and the first conflation
	is a $T'$-presentation. Consequently
	\begin{equation}
		\label{eq:fsb-full-index-mutation}
		\begin{gathered}
			g'=[V]-[W'],\qquad g'_k=-g_k,\\
			g'_j=g_j+\alpha c_j^+-\beta c_j^-
			\qquad(j\ne k).
		\end{gathered}
	\end{equation}
	Here $j$ ranges over the mutable and frozen summands.
	
	Under the QP realization, the decorated representation associated
	with $H$ has module $\Hom_{\underline{\cE}}(qT,\Sigma qH)$ and
	mutation parameters
	\[
	h_k=-\beta,\qquad h_k'=-\alpha
	\qquad\text{\cite[Section~4.1, Corollary~4.10]{PlaApplications}}.
	\]
	Here $\alpha,\beta$ can be read from the original minimal
	presentation after passing to the stable category. Indeed, for a
	non-projective indecomposable $T_i$, the ideal
	$[\cP](T_i,T_i)$ lies in $\rad\End_{\cE}(T_i)$: an invertible
	map factoring through a projective would make $T_i$ projective.
	A radical map between mutable indecomposables therefore cannot
	become invertible in the stable category.
	By \cite[Lemma~5.2]{DWZII},
	\begin{equation}
		\label{eq:fsb-all-object-f-mutation}
		\begin{aligned}
			\mu_k^y(F_{T'}^H)
			&=(1+y_k^{-1})^{-h_k'}(1+y_k)^{h_k}F_T^H\\
			&=y_k^{-\alpha}(1+y_k)^{g_k}F_T^H,
		\end{aligned}
	\end{equation}
	where
	\[
	\mu_k^y(y_k')=y_k^{-1},\qquad
	\mu_k^y(y_j')
	=y_jy_k^{[b_{kj}]_+}(1+y_k)^{-b_{kj}}
	\quad(j\ne k).
	\]
	
	Set
	\begin{equation}
		\label{eq:fsb-seed-substitution}
		P=x^{c^+},\qquad Q=x^{c^-},\qquad
		\widehat y_k=P/Q,\qquad
		\mu_k^x(x_k')=(P+Q)/x_k.
	\end{equation}
	The other coordinates are unchanged, and
	\[
	\mu_k^x(\widehat y_j')
	=\left.\mu_k^y(y_j')\right|_{y=\widehat y}.
	\]
	Equation~\eqref{eq:fsb-full-index-mutation} gives
	\[
	\begin{aligned}
		\frac{\mu_k^x((x')^{g'})}{x^g}
		&=Q^{-g_k}(1+\widehat y_k)^{-g_k}
		x^{\alpha c^+-\beta c^-}\\
		&=\widehat y_k^\alpha(1+\widehat y_k)^{-g_k}.
	\end{aligned}
	\]
	Multiplying this identity by
	\eqref{eq:fsb-all-object-f-mutation}, evaluated at $y=\widehat y$,
	yields $\mu_k^x(\CC_{T'}(H))=\CC_T(H)$.
	
	For $X\in\Db(\cE)$, write
	\[
	\kappa(X)=[H]+\sum_pa_p[P_p].
	\]
	Since $z_p$ is the frozen variable associated with $P_p$,
	\[
	\begin{aligned}
		\mu_k^x\bigl(\CC_{T',\mathrm{loc}}(X)\bigr)
		&=\mu_k^x\bigl(\CC_{T'}(H)\bigr)\prod_pz_p^{a_p}\\
		&=\CC_T(H)\prod_pz_p^{a_p}
		=\CC_{T,\mathrm{loc}}(X),
	\end{aligned}
	\]
	since mutation fixes the frozen variables.
\end{proof}

\begin{corollary}
	\label{cor:fsb-upper-valued-character}
	Under Assumption~\ref{ass:fsb-QP-realization}, the localized character
	of every object in $\cD$ belongs to $\mathcal U_T$.
\end{corollary}
\begin{proof}
	Iterating Lemma~\ref{lem:fsb-seed-covariance} expresses the same character
	as a Laurent polynomial in every seed. This is the defining intersection
	for $\mathcal U_T$.
\end{proof}

Let $\cE'$ be another Frobenius category with the same
cluster-structure and character hypotheses. Write
\[
\cD'=\Db(\cE'),\qquad \cN'=K^b(\cP'),\qquad
q':\cD'\longrightarrow\underline{\cE'}.
\]
We use $\kappa'$ for its relative-class map and $I_U$ for its localized
index at a cluster-tilting object $U$. The ring
$R_{\mathrm{loc}}'$ is its cluster algebra with inverted coefficients.
For an integer matrix $H$ with $r$ columns, the notation $u^H$
means the tuple $(u^{He_1},\ldots,u^{He_r})$.

\begin{proposition}
	\label{prop:fsb-relative-recognition}
	Let $F:\Db(\cE)\to\Db(\cE')$ be a $\kk$-linear triangle functor
	such that $F(\cN)\subseteq\cN'$ . Assume that it induces an equivalence
	\[
	\overline F:\underline{\cE}\xrightarrow{\sim}\underline{\cE'}.
	\]
	Let
	$U=\bigoplus_{i=1}^rU_i\oplus\bigoplus_{p=1}^{s'}P_p'$
	be a basic cluster-tilting object of $\cE'$, with non-projective
	indecomposable summands $U_i$ and all indecomposable
	projective-injective summands $P_p'$, ordered so that
	\[
	q'U_i\simeq\overline F(qT_i)\qquad(1\le i\le r).
	\]
Then we have
	\begin{itemize}
		\item[(1)] There is an integer matrix
		\[
		G=\begin{pmatrix}I_r&0\\A&C_0\end{pmatrix},
		\qquad A\in M_{s',r}(\ZZ),\quad C_0\in M_{s',s}(\ZZ),
		\]
		such that
		\[
		G\bB_T=\bB_U,
		\qquad I_U(FX)=GI_T(X)
		\qquad(X\in\Db(\cE)).
		\]
		Write $u=(u_1,\ldots,u_r,z'_1,\ldots,z'_{s'})$ for the
		coordinates of $U$. The Laurent-ring homomorphism
		\[
		f_G:\kk[x_1^{\pm1},\ldots,x_r^{\pm1},z_1^{\pm1},\ldots,z_s^{\pm1}]
		\longrightarrow
		\kk[u_1^{\pm1},\ldots,u_r^{\pm1},(z'_1)^{\pm1},\ldots,(z'_{s'})^{\pm1}],\, x^v\longmapsto u^{Gv},
		\]
		satisfies
		\begin{equation}
			\label{eq:fsb-matched-naturality}
			f_G\bigl(\CC_{T,\mathrm{loc}}(X)\bigr)
			=\CC_{U,\mathrm{loc}}(FX).
		\end{equation}
		\item[(2)] Fix a basic cluster-tilting object \(T'\in\cE'\) whose seed
		defines the chosen cluster structure on \(R_{\mathrm{loc}}'\).
		Assume that \(U\) is reachable from \(T'\). Thus there is a finite
		sequence of mutable vertices \((k_1,\ldots,k_\ell)\) such that
		\begin{equation*}
			U\simeq\mu_{k_\ell}\cdots\mu_{k_1}(T').
		\end{equation*}
		We order the resulting seed according to the chosen summands of
		\(U\), and regard its coordinates
		\(u_1,\ldots,u_r,z'_1,\ldots,z'_{s'}\) as elements of
		\(R_{\mathrm{loc}}'\).
		Assume also that Assumption~\ref{ass:fsb-QP-realization} holds
		throughout the mutation class of \(T'\).
		
		Write \(A=(a_{pi})\) and \(C_0=(c_{pj})\).
		Let
		\(\psi:R_{\mathrm{loc}}\to R_{\mathrm{loc}}'\)
		be a \(\kk\)-algebra homomorphism satisfying
		\begin{equation*}
			\begin{aligned}
				\psi(x_i)
				&=u_i\prod_{p=1}^{s'}(z'_p)^{a_{pi}}
				&& (1\leq i\leq r),\\
				\psi(z_j)
				&=\prod_{p=1}^{s'}(z'_p)^{c_{pj}}
				&& (1\leq j\leq s).
			\end{aligned}
		\end{equation*}
		Then \(\psi\) is a quasi-cluster morphism, with the initial
		source seed \(\Sigma_T\) matched to the target seed \(\Sigma_U\).
		The displayed formulas send frozen Laurent monomials to frozen
		Laurent monomials and rescale each matched mutable variable by
		such a monomial. Moreover, the identity
		\(G\bB_T=\bB_U\) implies that the corresponding Laurent
		substitution preserves the hatted variables:
		\begin{equation*}
			f_G(\widehat y_{k,T})
			=u^{G\bB_Te_k}
			=u^{\bB_Ue_k}
			=\widehat y_{k,U}
			\qquad (1\leq k\leq r).
		\end{equation*}
		If, in addition, \(\psi\) is a ring isomorphism,
		\(s=s'\), and \(C_0\in\operatorname{GL}_s(\ZZ)\), then
		\(\psi\) is a quasi-cluster isomorphism.
	\end{itemize}
\end{proposition}

\begin{proof}
	We first prove (1). Since $F(\cN)\subseteq\cN'$, it induces
	a homomorphism
	\begin{equation*}
		K(F):K(\cD,\cN)\longrightarrow K(\cD',\cN'),
		\qquad [X]_{\mathrm{rel}}\longmapsto[FX]_{\mathrm{rel}}.
	\end{equation*}
	Indeed, the image of a triangle whose first term is perfect is
	again a triangle with perfect first term. By Lemma~\ref{lem:fsb-projective-correction}, there are unique
	integers \(a_{pi}\) and \(c_{pj}\) such that
	\begin{equation}
		\label{eq:fsb-functor-columns}
		\begin{aligned}
			\kappa'(FT_i)&=[U_i]+\sum_{p=1}^{s'}a_{pi}[P'_p],
			&&1\leq i\leq r,\\
			\kappa'(FP_j)&=\sum_{p=1}^{s'}c_{pj}[P'_p],
			&&1\leq j\leq s.
		\end{aligned}
	\end{equation}
	For the second identity, we use \(q'FP_j=0\) and choose the zero
	object as the Frobenius representative.
	Set \(A=(a_{pi})\), \(C_0=(c_{pj})\), and
	\begin{equation*}
		G=\begin{pmatrix}I_r&0\\ A&C_0\end{pmatrix}.
	\end{equation*}
	Because \(\ind_U U_i=[U_i]\) and \(\ind_U P'_p=[P'_p]\), the
	columns of this matrix satisfy
	\begin{equation*}
		Ge_i=I_U(FT_i),\qquad Ge_{r+j}=I_U(FP_j).
	\end{equation*}
	Thus, for every \(V\in\add T\), additivity gives
	\begin{equation*}
		I_U(FV)=G[V],
	\end{equation*}
	where \([V]\) is expressed in the ordered basis of \(L_T\).
	
	We next compare exchange triangles. 
	We briefly show that the
	minimal approximation maps defining an exchange conflation remain
	minimal after passage to the stable category. Let
	\(a:E\to T_k\) be a minimal right
	\(\add T_{\widehat k}\)-approximation. The approximation property
	descends because the stable quotient is full. If
	\(q(a)\theta=q(a)\), lift \(\theta\) to an endomorphism
	\(\widetilde\theta\) of \(E\). Then
	\begin{equation*}
		a-a\widetilde\theta=bh
	\end{equation*}
	for maps \(h:E\to P\) and \(b:P\to T_k\), with \(P\in\cP\).
	Since \(P\in\add T_{\widehat k}\), write \(b=ac\).
	It follows that \(a(\widetilde\theta+ch)=a\), so minimality makes
	\(\widetilde\theta+ch\) invertible. Its stable image is \(\theta\),
	which is therefore invertible. The argument for left approximations
	is dual.
	
	Consequently, the stable images of the exchange conflations are
	the exchange triangles at \(qT_k\). The equivalence
	\(\overline F\) preserves these approximation properties, so it
	carries them to the exchange triangles at \(q'U_k\), with the same
	orientation of the endpoints. In particular,
	\(\overline F(qT_k^*)\simeq q'U_k^*\).
	Write these exchange conflations as
	\begin{equation*}
		\begin{gathered}
			0\longrightarrow T_k^*\longrightarrow E_k^+
			\longrightarrow T_k\longrightarrow0,
			\qquad
			0\longrightarrow T_k\longrightarrow E_k^-
			\longrightarrow T_k^*\longrightarrow0,\\
			0\longrightarrow U_k^*\longrightarrow E_{U,k}^+
			\longrightarrow U_k\longrightarrow0,
			\qquad
			0\longrightarrow U_k\longrightarrow E_{U,k}^-
			\longrightarrow U_k^*\longrightarrow0.
		\end{gathered}
	\end{equation*}
	Applying Lemma~\ref{lem:fsb-defect-invariance} in the target derived
	category and identifying its relative group with
	\(K_0^{\mathrm{sp}}(\cE')\), we obtain
	\begin{equation*}
		\begin{aligned}
			\kappa'(FT_k^*)-\kappa'(FE_k^+)+\kappa'(FT_k)
			&=[U_k^*]-[E_{U,k}^+]+[U_k],\\
			\kappa'(FT_k)-\kappa'(FE_k^-)+\kappa'(FT_k^*)
			&=[U_k]-[E_{U,k}^-]+[U_k^*].
		\end{aligned}
	\end{equation*}
	By subtracting cancels the endpoint classes, we have
	\begin{equation*}
		\kappa'(FE_k^+)-\kappa'(FE_k^-)
		=[E_{U,k}^+]-[E_{U,k}^-].
	\end{equation*}
	Applying \(\ind_U\), we find
	\begin{equation}
		\label{eq:fsb-full-matrix-naturality}
		\begin{aligned}
			G\bB_Te_k
			&=G\bigl([E_k^+]-[E_k^-]\bigr)\\
			&=I_U(FE_k^+)-I_U(FE_k^-)\\
			&=[E_{U,k}^+]-[E_{U,k}^-]
			=\bB_Ue_k.
		\end{aligned}
	\end{equation}
	This proves \(G\bB_T=\bB_U\).
	
	For \(H\in\cE\), choose a \(T\)-presentation
	\begin{equation*}
		0\longrightarrow T_1^H\longrightarrow T_0^H
		\longrightarrow H\longrightarrow0.
	\end{equation*}
	Its image under \(F\) is a triangle whose first two terms have
	stable images in \(\add(q'U)\). Hence
	Lemma~\ref{lem:fsb-presentation-index}, applied in \(\cE'\), gives
	\begin{equation*}
		\begin{aligned}
			I_U(FH)
			&=I_U(FT_0^H)-I_U(FT_1^H)\\
			&=G\bigl([T_0^H]-[T_1^H]\bigr)
			=G\ind_T H.
		\end{aligned}
	\end{equation*}
	The homomorphisms \(I_U\circ K(F)\) and \(G\circ I_T\) thus agree
	on the classes of objects of \(\cE\). These classes generate
	\(K(\cD,\cN)\), since the canonical map
	\(K_0^{\mathrm{sp}}(\cE)\to K(\cD,\cN)\) is an isomorphism.
	Therefore
	\begin{equation}
		\label{eq:fsb-full-index-naturality}
		I_U(FX)=GI_T(X)\qquad(X\in\cD).
	\end{equation}
	
	Choose isomorphisms
	\(\overline F(qT_i)\simeq q'U_i\) compatible with the given
	ordering. They induce an algebra isomorphism
	\begin{equation*}
		\Gamma_T=\End_{\underline{\cE}}(qT)
		\xrightarrow{\sim}
		\End_{\underline{\cE'}}(q'U)=\Gamma_U
	\end{equation*}
	preserving the primitive idempotents indexed by the mutable
	summands. Since \(\overline F\) is a triangle equivalence, it also
	gives a compatible isomorphism of right modules
	\begin{equation*}
		M_T(qX)=\Hom_{\underline{\cE}}(qT,\Sigma qX)
		\xrightarrow{\sim}
		\Hom_{\underline{\cE'}}(q'U,\Sigma q'FX)=M_U(q'FX).
	\end{equation*}
	Thus the submodule Grassmannians of every dimension vector are
	isomorphic, and
	\begin{equation*}
		F_T^{qX}(y)=F_U^{q'FX}(y).
	\end{equation*}
	By combining this with the matrix and index identities, we obtain
	\begin{equation*}
		\begin{aligned}
			f_G\bigl(\CC_{T,\mathrm{loc}}(X)\bigr)
			&=u^{GI_T(X)}F_T^{qX}(u^{G\bB_T})\\
			&=u^{I_U(FX)}F_U^{q'FX}(u^{\bB_U})\\
			&=\CC_{U,\mathrm{loc}}(FX),
		\end{aligned}
	\end{equation*}
	which proves (1).
	
	We now prove (2). Reachability ensures that \(\Sigma_U\) is a
	seed in the cluster pattern defining \(R_{\mathrm{loc}}'\).
	The Laurent phenomenon gives
	\begin{equation*}
		R_{\mathrm{loc}}[x_1^{-1},\ldots,x_r^{-1}]
		=\kk[x_1^{\pm1},\ldots,x_r^{\pm1},
		z_1^{\pm1},\ldots,z_s^{\pm1}].
	\end{equation*}
	Since every \(\psi(x_i)\) is nonzero, \(\psi\) extends to a
	homomorphism from this localization to
	\(\operatorname{Frac}(R_{\mathrm{loc}}')\). The prescribed images
	of the initial coordinates identify this extension with \(f_G\),
	where the \(u\)-coordinates are interpreted in the target ambient
	field. In particular,
	\begin{equation*}
		f_G(\widehat y_{k,T})
		=u^{G\bB_Te_k}=u^{\bB_Ue_k}=\widehat y_{k,U}.
	\end{equation*}
	The upper block of \(G\) is \((I_r\ 0)\), so the principal parts
	of \(\bB_T\) and \(\bB_U\) coincide.
	
	We verify that this comparison persists under simultaneous mutation.
	Write \(a_i=Ae_i\), so that \(\psi(x_i)=u_i(z')^{a_i}\).
	Fix a mutable vertex \(k\), and set
	\begin{equation*}
		\beta=\bB_Te_k,\qquad \gamma=\bB_Ue_k=G\beta.
	\end{equation*}
	Using \(\beta=[\beta]_+-[-\beta]_+\), we obtain
	\begin{equation*}
		G[\beta]_+-[\gamma]_+
		=G[-\beta]_+-[-\gamma]_+.
	\end{equation*}
	The mutable part of this vector vanishes, because the mutable
	part of \(Gv\) equals that of \(v\). Hence there is
	\(d_k\in\ZZ^{s'}\) such that
	\begin{equation*}
		G[\beta]_+-[\gamma]_+
		=G[-\beta]_+-[-\gamma]_+
		=\binom{0}{d_k}.
	\end{equation*}
	Let \(x_k^{\mathrm{new}}\) and \(u_k^{\mathrm{new}}\) be the
	mutated variables. Applying \(\psi\) to the exchange relation gives
	\begin{equation*}
		\begin{aligned}
			\psi(x_k)\psi(x_k^{\mathrm{new}})
			&=u^{G[\beta]_+}+u^{G[-\beta]_+}\\
			&=(z')^{d_k}\bigl(u^{[\gamma]_+}+u^{[-\gamma]_+}\bigr)\\
			&=(z')^{d_k}u_ku_k^{\mathrm{new}}.
		\end{aligned}
	\end{equation*}
	Since \(\psi(x_k)=u_k(z')^{a_k}\), by cancellation in the target
	fraction field, one has
	\begin{equation*}
		\psi(x_k^{\mathrm{new}})
		=u_k^{\mathrm{new}}(z')^{d_k-a_k}.
	\end{equation*}
	Thus the new exponent matrix has the same form:
	\begin{equation*}
		G^{\mathrm{new}}
		=\begin{pmatrix}I_r&0\\ A^{\mathrm{new}}&C_0\end{pmatrix},
		\qquad
		a_j^{\mathrm{new}}=
		\begin{cases}
			a_j,&j\ne k,\\
			d_k-a_k,&j=k.
		\end{cases}
	\end{equation*}
	
	We must also check the hatted variables at the mutated seeds.
	Writing \((b_{ij})\) for their common principal exchange matrix
	before mutation, the hatted-variable mutation rule is
	\begin{equation*}
		\widehat y_j^{\mathrm{new}}=
		\begin{cases}
			\widehat y_k^{-1},&j=k,\\
			\widehat y_j\widehat y_k^{[b_{kj}]_+}
			(1+\widehat y_k)^{-b_{kj}},&j\ne k.
		\end{cases}
	\end{equation*}
	The same formula applies to both seeds, so equality of their
	hatted variables is preserved.
	By algebraic independence of the mutated target coordinates, the
	equality of hatted monomials implies
	\begin{equation*}
		G^{\mathrm{new}}\mu_k(\bB_T)=\mu_k(\bB_U).
	\end{equation*}
	The same argument can therefore be applied at the next pair of
	matched seeds. Induction along any mutation sequence shows that
	every mutable source variable is sent to the corresponding target
	variable multiplied by a frozen Laurent monomial. The frozen
	coefficient map remains \(z^b\mapsto(z')^{C_0b}\) throughout.
	
	Finally, specialize all frozen variables to \(1\). The map
	\(\psi\) sends the ideal generated by the \(z_j-1\) into the
	ideal generated by the \(z'_p-1\). By \cite[Lemma~A.1]{KW}, the
	corresponding quotients are the coefficient-free cluster algebras.
	The two
	coefficient-free seed patterns have the same initial principal
	exchange matrix, and the induced map takes the initial source
	variables to the mutable variables of \(\Sigma_U\).
	The map between their rational function fields determined by these
	algebraically independent clusters is an isomorphism. Since it
	intertwines each exchange relation in both directions, it restricts
	to an isomorphism of the coefficient-free cluster algebras.
	Together with the coefficient-group map and the hatted-variable
	identities, this proves that \(\psi\) is a quasi-cluster morphism.
	
	Suppose now that \(\psi\) is a ring isomorphism, \(s=s'\), and
	\(C_0\in\operatorname{GL}_s(\ZZ)\). Then
	\begin{equation*}
		G^{-1}
		=\begin{pmatrix}I_r&0\\-C_0^{-1}A&C_0^{-1}\end{pmatrix},
		\qquad G^{-1}\bB_U=\bB_T.
	\end{equation*}
	The inverse ring map satisfies
	\begin{equation*}
		\psi^{-1}(u_i)=x_i z^{-C_0^{-1}a_i},
		\qquad
		\psi^{-1}(z'_p)=z^{C_0^{-1}e_p}.
	\end{equation*}
	Applying the same exchange calculation to these inverse
	substitutions shows that \(\psi^{-1}\) is also a quasi-cluster
	morphism. Therefore \(\psi\) is a quasi-cluster isomorphism.
\end{proof}

\subsection{Boundary orders and the two positroid models}
\label{sec:fsb-rank-one}
\label{sec:fsb-reachability}

We now specialize to a connected ordinary Postnikov diagram $D$ of type
$(k,n)$, where $0<k<n$. Here $D$ is the strand diagram obtained by drawing
the trips of a reduced plabic graph with the ordinary cyclic boundary
labelling. Denote its positroid by $\pi$, and put
\[
R_\pi=\kk[\widehat\Pi_\pi^\circ],\qquad Z=\kk\llbracket t\rrbracket.
\]
As in the preceding section, the hat denotes the open part of the affine
Pl\"ucker cone and the necklace coordinates are inverted in $R_\pi$.

\begin{definition}[Dimer algebra and boundary order]
	\label{def:fsb-boundary-order}
	Let $(Q_D,F_D)$ be the dual ice quiver of $D$, with boundary vertices
	frozen. We retain the arrows between frozen vertices when constructing
	the algebra, even though they do not occur in an extended exchange
	matrix. Let $W_D$ be the signed sum of the oriented dimer face cycles,
	with opposite signs for the two orientations. The complete dimer algebra
	and its boundary algebra are
	\[
	A_D=\widehat{\kk Q_D}/
	\overline{(\partial_aW_D:a\notin(F_D)_1)},
	\qquad B=eA_De,
	\qquad e=\sum_{j\in(F_D)_0}e_j.
	\]
	Here completion and closure are taken in the arrow-adic topology;
	$\partial_a$ is the cyclic derivative, obtained by deleting each
	occurrence of $a$ from a cycle and summing the resulting paths.
\end{definition}

We use the orientation and potential conventions of
\cite[Definitions~2.13 and~5.1]{Pressland}: a dual arrow has the
white endpoint of the crossed plabic edge on its right. Together with
$b_{jk}=\#(k\to j)-\#(j\to k)$, this agrees with
\eqref{eq:fsb-exchange-column} and the left-module convention above.

For connected $D$, $A_D$ and $B$ are
finite free $Z$-algebras; the central parameter $t$ is represented by the
sum of the basic face cycles based at the quiver vertices. 

A finite free
$Z$-algebra of this kind is called a \emph{$Z$-order} here.
The boundary algebra is Iwanaga--Gorenstein of injective dimension at most
two \cite[Corollary~10.5]{CKP}; this means it is Noetherian on both sides
and its regular left and right modules have finite injective dimension.

\begin{definition}[Lattices and the circle order]
	\label{def:fsb-lattices-circle}
	For a $Z$-order $\Lambda$, let $\CM(\Lambda)$ be the category of
	finitely generated left $\Lambda$-modules which are free of finite rank
	over $Z$. Its objects are called \emph{lattices}, or Cohen--Macaulay
	modules. Its conflations are short exact sequences of $\Lambda$-modules
	whose terms are lattices. Such sequences split over $Z$, but need not
	split over $\Lambda$.
	
	Let $Q_n$ have vertices $\mathbb Z/n\mathbb Z+\tfrac12$ and arrows
	\[
	\mathsf x_i:i-\tfrac12\longrightarrow i+\tfrac12,
	\qquad
	\mathsf y_i:i+\tfrac12\longrightarrow i-\tfrac12.
	\]
	Put $\mathsf x=\sum_i\mathsf x_i$ and
	$\mathsf y=\sum_i\mathsf y_i$. The \emph{circle order} is
	\[
	C=C_{k,n}=\widehat{\kk Q_n}/
	\overline{(\mathsf x\mathsf y-\mathsf y\mathsf x,
		\mathsf y^k-\mathsf x^{n-k})},
	\qquad t=\mathsf x\mathsf y.
	\]
	The generators $\mathsf x_i,\mathsf y_i$ are arrows, not cluster
	variables. For a $C$-lattice, its rank is the common $Z$-rank of its
	vertex components; equality follows after inverting $t$.
\end{definition}

\begin{definition}[Rank-one modules]
	\label{def:fsb-rank-one-module}
	For $I\in\binom{[n]}k$, let $M_I$ be the $C$-module with
	\[
	(M_I)_a=Z\quad\text{at every vertex},\qquad
	\mathsf x_i=t^{\mathbf1_I(i)},\qquad
	\mathsf y_i=t^{1-\mathbf1_I(i)},
	\]
	where the displayed powers denote multiplication maps and
	$\mathbf1_I(i)$ is $1$ for $i\in I$ and $0$ otherwise.
	Thus $M_I$ has rank one. These maps satisfy the relations of $C$,
	and every rank-one $C$-lattice is isomorphic to a unique $M_I$.
\end{definition}

By \cite[Propositions~5.6 and~5.8]{Pressland}, we have the restriction functor $\CM(B)\lhook\joinrel\longrightarrow\CM(C)$ is exact
and fully faithful and
\begin{equation}
	\label{eq:fsb-boundary-restriction}
	M_I\text{ is in its essential image}
	\iff\Delta_I\ne0\text{ in }R_\pi.
\end{equation}
For such $I$, we also write $M_I$ for its unique $B$-lattice lift,
up to isomorphism. This is the meaning of $M_I\in\CM(B)$ below.

\begin{definition}[The source and target Frobenius categories]
	\label{def:fsb-source-target-Frobenius}
	Set $B^\vee=\Hom_Z(B,Z)$, with left action
	$(b\lambda)(a)=\lambda(ab)$. We define the following exact subcategories
	\[
	\begin{aligned}
		\cE_{\mathrm s}
		&=\GP\CM(B)
		=\{M\in\CM(B):\Ext_B^j(M,B)=0\text{ for all }j>0\},\\
		\cE_{\mathrm t}
		&=\GI\CM(B)
		=\{M\in\CM(B):\Ext_B^j(B^\vee,M)=0\text{ for all }j>0\}.
	\end{aligned}
	\]
	Their exact structures are inherited from $\CM(B)$, and their
	projective-injective subcategories are $\add B$ and $\add B^\vee$,
	respectively. Their standard cluster-tilting objects are
	\[
	T^{\mathrm s}=eA_D,\qquad
	T^{\mathrm t}=(A_De)^\vee.
	\]
	The summands indexed by a face $j$ are
	$eA_De_j$ and $(e_jA_De)^\vee$, respectively.
\end{definition}


These are the standard Frobenius and character realizations of the
source- and target-labelled seeds. Under restriction to $C$, we have
\[
eA_De_j\simeq M_{I_j^{\mathrm{src}}},\qquad
(e_jA_De)^\vee\simeq M_{I_j^{\mathrm{tgt}}}.
\]
By\cite[Theorems~5.12--5.15]{Pressland}, assumptions~\ref{ass:fsb-frobenius-cluster} and
\ref{ass:fsb-character-realization} hold by the established Frobenius
categorifications.

Let \(\fgmod B\) be the category of finitely generated left
\(B\)-modules. Put
\[
\proj B=\add B,
\qquad
\operatorname{per}B=K^b(\proj B),
\qquad
\cC=\Dsg(B)=\Db(\fgmod B)/\operatorname{per}B,
\]
and let \(q:\Db(\fgmod B)\to\cC\) be the quotient functor.
The inclusions of \(\cE_{\mathrm s}\) and \(\cE_{\mathrm t}\)
into \(\fgmod B\) induce triangle equivalences
\[
j_\varepsilon:
\Db(\cE_\varepsilon)\xrightarrow{\sim}\Db(\fgmod B),
\qquad \varepsilon\in\{\mathrm s,\mathrm t\}.
\]

We first check that both equivalences send perfect objects
to \(\operatorname{per}B\).
The projective-injective subcategories are
\[
\cP_{\mathrm s}=\add B,
\qquad
\cP_{\mathrm t}=\add B^\vee,
\qquad
B^\vee=\operatorname{Hom}_Z(B,Z).
\]
Since \(B\) is Iwanaga--Gorenstein on both sides, there are
finite exact sequences
\[
\begin{gathered}
	0\longrightarrow B\longrightarrow J^0
	\longrightarrow\cdots\longrightarrow J^d
	\longrightarrow0,
	\qquad J^i\in\add B^\vee,\\
	0\longrightarrow Q_e\longrightarrow\cdots
	\longrightarrow Q_0\longrightarrow B^\vee
	\longrightarrow0,
	\qquad Q_i\in\add B.
\end{gathered}
\]
The first is a finite injective coresolution of \(B\)
in \(\CM(B)\); the second is obtained by applying
\(Z\)-duality to the corresponding coresolution
of the right \(B\)-module \(B\).
These sequences imply
\[
B\in\operatorname{thick}(B^\vee),
\qquad
B^\vee\in\operatorname{thick}(B),
\]
and therefore
\[
\begin{aligned}
	j_{\mathrm s}\bigl(K^b(\cP_{\mathrm s})\bigr)
	&=\operatorname{thick}(B)
	=\operatorname{per}B,\\
	j_{\mathrm t}\bigl(K^b(\cP_{\mathrm t})\bigr)
	&=\operatorname{thick}(B^\vee)
	=\operatorname{per}B.
\end{aligned}
\]

Thus \(j_{\mathrm s}\) and \(j_{\mathrm t}\) descend
to equivalences of the stable categories with \(\cC\).
Writing \(q_{\mathrm s}\) and \(q_{\mathrm t}\) for the
canonical quotient functors, we obtain a diagram
commuting up to natural isomorphism:
\begin{equation}
	\label{eq:fsb-standard-derived}
	\begin{tikzcd}[column sep=large,row sep=large]
		\Db(\cE_{\mathrm t})
		\arrow[r,"j_{\mathrm t}","\sim"']
		\arrow[d,"q_{\mathrm t}"']
		&
		\Db(\fgmod B)
		\arrow[d,"q"]
		&
		\Db(\cE_{\mathrm s})
		\arrow[l,"j_{\mathrm s}"',"\sim"]
		\arrow[d,"q_{\mathrm s}"]
		\\
		\underline{\cE}_{\mathrm t}
		\arrow[r,"\overline j_{\mathrm t}","\sim"']
		&
		\cC
		&
		\underline{\cE}_{\mathrm s}
		\arrow[l,"\overline j_{\mathrm s}"',"\sim"]
	\end{tikzcd}
\end{equation}
In particular, the equivalences
\[
\varphi=j_{\mathrm s}^{-1}j_{\mathrm t},
\qquad
\overline\varphi
=\overline j_{\mathrm s}^{-1}\overline j_{\mathrm t}
\]
satisfy
\[
q_{\mathrm s}\circ\varphi
\simeq
\overline\varphi\circ q_{\mathrm t}.
\]

We also explain why Assumption~\ref{ass:fsb-QP-realization} holds.
Let \(\overline Q_D\) be the full subquiver of the reduced
dimer quiver \(Q_D\) on the mutable vertices, and let
\(\overline W_D\) be obtained from \(W_D\) by discarding
the cyclic paths that pass through a frozen vertex.
The completed Jacobian algebra satisfies
\[
J(\overline Q_D,\overline W_D)
\simeq A_D/A_DeA_D,
\]
and is finite-dimensional by
\cite[Proposition~4.4]{PresslandCY}.

The results of \cite[Theorem~4.18 and Section~7]{KW}
are stated for right modules, whereas
\(\cE_{\mathrm s}=\GP\CM(B)\) consists of left
\(B\)-modules. We therefore apply these results to the
opposite ice quiver with potential
\[
(Q_D^{\mathrm{op}},F_D^{\mathrm{op}},W_D^{\mathrm{op}}),
\]
where all arrows and cyclic paths are reversed.
Its relative Jacobian algebra is \(A_D^{\mathrm{op}}\),
and its boundary algebra is
\[
eA_D^{\mathrm{op}}e=B^{\mathrm{op}}.
\]
Under the identification of right \(B^{\mathrm{op}}\)-modules
with left \(B\)-modules, the canonical cluster-tilting module
\(A_D^{\mathrm{op}}e\) corresponds to
\[
eA_D=T^{\mathrm s}.
\]

Consequently, the cited results give a triangle equivalence
\[
\Psi_{\mathrm s}:
\underline{\cE}_{\mathrm s}
\xrightarrow{\sim}
\cC(\overline Q_D^{\mathrm{op}},
\overline W_D^{\mathrm{op}})
\]
satisfying
\[
\Psi_{\mathrm s}(qT^{\mathrm s})\simeq U_D.
\]
Here \(U_D\) is the image of the regular Ginzburg dg module
in the generalized cluster category.
By taking endomorphism algebras, we have
\[
\begin{aligned}
	\Gamma_{T^{\mathrm s}}
	&=
	\End_{\underline{\cE}_{\mathrm s}}(qT^{\mathrm s})\\
	&\simeq
	\End_{\cC(\overline Q_D^{\mathrm{op}},
		\overline W_D^{\mathrm{op}})}(U_D)\\
	&\simeq
	J(\overline Q_D^{\mathrm{op}},
	\overline W_D^{\mathrm{op}})\\
	&\simeq
	(A_D/A_DeA_D)^{\mathrm{op}}.
\end{aligned}
\]

The mutation argument in \cite[Section~5 and Theorem~5.7]{PresslandCY}
shows that iterated ice-QP mutation at mutable vertices has no loops
or oriented $2$-cycles incident with a mutable vertex after reduction.
Deleting frozen vertices commutes with mutation at the remaining
vertices, up to reduction. Hence the unfrozen potential is
non-degenerate. The categorical realization and the associated
decorated representations are compatible with mutation
\cite[Proposition~4.1]{PlaApplications}.

For the case of target Frobenius category, by \cite[Theorem~6.4]{Pressland}, we have
\[
\overline\varphi(qT^{\mathrm t})\simeq\Sigma^2qT^{\mathrm s}.
\]
Define
\[
\Psi_{\mathrm t}
=
\Psi_{\mathrm s}\circ\Sigma^{-2}\circ\overline\varphi:
\underline{\cE}_{\mathrm t}
\xrightarrow{\sim}
\cC(\overline Q_D^{\mathrm{op}},
\overline W_D^{\mathrm{op}}).
\]
Since
$\overline\varphi(qT^{\mathrm t})
\simeq\Sigma^2qT^{\mathrm s},$
we obtain
\[
\Psi_{\mathrm t}(qT^{\mathrm t})
\simeq
\Psi_{\mathrm s}(qT^{\mathrm s})
\simeq U_D.
\]
Thus both \(\Psi_{\mathrm s}\) and \(\Psi_{\mathrm t}\)
send the respective initial cluster-tilting objects to \(U_D\).

Triangle equivalences preserve exchange triangles.
Hence the mutation compatibility of \(\Psi_{\mathrm s}\)
also holds for \(\Psi_{\mathrm t}\), i.e. for every mutation
sequence \(\boldsymbol\mu\), one has
\[
\Psi_\varepsilon\bigl(\boldsymbol\mu(qT^\varepsilon)\bigr)
\simeq
\boldsymbol\mu(U_D),
\qquad \varepsilon\in\{\mathrm s,\mathrm t\}.
\]
Consequently, the source and target categorifications satisfy
Assumption~\ref{ass:fsb-QP-realization} at every reachable
cluster-tilting object.

For \(\varepsilon\in\{\mathrm s,\mathrm t\}\), let
\(\mathcal A_\varepsilon\) and \(\mathcal U_\varepsilon\)
be the abstract cluster algebra and upper cluster algebra
associated with the corresponding initial seed, with all
frozen variables inverted.
By \cite[Theorem~2.15 and Proposition~2.16]{Pressland}, we have
\[
\mathcal A_\varepsilon=\mathcal U_\varepsilon,
\qquad
\eta_\varepsilon:
\mathcal A_\varepsilon\xrightarrow{\sim}R_\pi,
\]
where \(\eta_{\mathrm s}\) and \(\eta_{\mathrm t}\)
are the source- and target-labelled cluster structures.

For every \(X\in\Db(\fgmod B)\), we have
\[
\CC_{T^\varepsilon,\mathrm{loc}}
\bigl(j_\varepsilon^{-1}X\bigr)
\in\mathcal U_\varepsilon=\mathcal A_\varepsilon.
\]
We may therefore define
\[
\Phi^\varepsilon:
\operatorname{obj}\Db(\fgmod B)\longrightarrow R_\pi,
\qquad
\Phi^\varepsilon(X)
=
\eta_\varepsilon\left(
\CC_{T^\varepsilon,\mathrm{loc}}
\bigl(j_\varepsilon^{-1}X\bigr)
\right).
\]
In particular, for \(H\in\cE_\varepsilon\),
\[
\Phi^\varepsilon(j_\varepsilon H)
=
\eta_\varepsilon\bigl(\CC_{T^\varepsilon}(H)\bigr).
\]
Every finitely generated \(B\)-module is regarded as a complex
concentrated in degree zero. Thus \(\Phi^{\mathrm s}(M_I)\)
is defined even when \(M_I\notin\cE_{\mathrm s}\).

For a triangle
\[
P\longrightarrow X\longrightarrow Y\longrightarrow P[1]
\]
in \(\Db(\fgmod B)\) with \(P\in\operatorname{per}B\),
the localized product formula gives
\[
\Phi^\varepsilon(X)
=
\Phi^\varepsilon(P)\Phi^\varepsilon(Y).
\]
For each \(\varepsilon\), Pressland's corresponding derived
character has the same restriction to
\(j_\varepsilon(\cE_\varepsilon)\) and satisfies this product
formula. By the uniqueness in
\cite[Theorem~5.28]{Pressland}, it coincides with
\(\Phi^\varepsilon\).
Hence \(\Phi^{\mathrm s}\) and \(\Phi^{\mathrm t}\)
are precisely the source and target derived characters
used in \cite[Sections~6--7]{Pressland}.

Let $\tau:R_\pi\to R_\pi$ be the coordinate-ring
\emph{left twist} automorphism of Muller--Speyer, in the convention of
\cite[Section~7]{Pressland}.

\subsection{Rank-one normalization}

The cited source--target and twist formulas first compare the initial
cluster-tilting summands. The next lemma extends that comparison to all objects in $\Db(\fgmod B)$.

\begin{lemma}
	\label{lem:fsb-all-twist}
	For each $X\in\Db(\fgmod B)$, we have
	\begin{align}
		\Phi^{\mathrm s}(X[-1])
		&=\tau\bigl(\Phi^{\mathrm t}(X)\bigr),
		\label{eq:fsb-all-twist}\\
		\Phi^{\mathrm s}(X)
		&=\Phi^{\mathrm t}(X).
		\label{eq:fsb-all-source-target}
	\end{align}
\end{lemma}

\begin{proof}
	For \(\varepsilon\in\{0,1\}\), put
	\begin{equation*}
		F_\varepsilon=[-\varepsilon]\circ\varphi,
		\qquad
		\xi_\varepsilon=
		\begin{cases}
			\mathrm{id}_{R_\pi},&\varepsilon=0,\\
			\tau,&\varepsilon=1.
		\end{cases}
	\end{equation*}
	Then
	\begin{equation*}
		\begin{gathered}
			j_{\mathrm s}F_\varepsilon
			\simeq[-\varepsilon]j_{\mathrm t},
			\qquad
			F_\varepsilon(\cN_{\mathrm t})=\cN_{\mathrm s},\\
			\overline F_\varepsilon
			=\Sigma^{-\varepsilon}\overline\varphi,
		\end{gathered}
	\end{equation*}
	where \(\cN_{\mathrm s}=K^b(\add B)\) and
	\(\cN_{\mathrm t}=K^b(\add B^\vee)\) in their respective
	Frobenius derived categories. The equality of perfect
	subcategories follows from their common image
	\(\operatorname{per}B\).
	
	By \cite[Proposition~6.3 and Theorem~6.4]{Pressland},
	we can choose a mutation sequence \(\boldsymbol k_\varepsilon\)
	and a cluster-tilting object $U_\varepsilon
		=\mu_{\boldsymbol k_\varepsilon}T^{\mathrm s}$
	such that, after ordering its mutable summands,
	\begin{equation*}
		\begin{aligned}
			q_{\mathrm s}U_\varepsilon
			&\simeq
			\Sigma^{2-\varepsilon}q_{\mathrm s}T^{\mathrm s}\\
			&\simeq
			\Sigma^{-\varepsilon}
			\overline\varphi(q_{\mathrm t}T^{\mathrm t})\\
			&=
			\overline F_\varepsilon(q_{\mathrm t}T^{\mathrm t}).
		\end{aligned}
	\end{equation*}
	Here \(q_{\mathrm s}\) and \(q_{\mathrm t}\) denote the stable
	quotient functors of the two Frobenius models.
	
	Write
	\(T^{\mathrm t}=\bigoplus_{j=1}^{m}T_j^{\mathrm t}\),
	including the projective-injective summands, and let
	\(x_j^{\mathrm t}\) be its initial extended-cluster coordinates.
	Set
	\begin{equation*}
		L_{\mathrm t}
		=
		\kk[(x_1^{\mathrm t})^{\pm1},\ldots,
		(x_m^{\mathrm t})^{\pm1}],
		\qquad
		K=\operatorname{Frac}(R_\pi).
	\end{equation*}
	We extend \(\eta_{\mathrm s}\), \(\eta_{\mathrm t}\), and
	\(\xi_\varepsilon\) to their fraction fields.
	
	Proposition~\ref{prop:fsb-relative-recognition} gives a
	Laurent-ring map \(f_\varepsilon=f_{G_\varepsilon}\) satisfying
	\begin{equation*}
		f_\varepsilon
		\bigl(\CC_{T^{\mathrm t},\mathrm{loc}}(Y)\bigr)
		=
		\CC_{U_\varepsilon,\mathrm{loc}}(F_\varepsilon Y)
		\qquad
		(Y\in\Db(\cE_{\mathrm t})).
	\end{equation*}
	Let
		$\mu_\varepsilon^x:
		\mathbb F_{U_\varepsilon}
		\xrightarrow{\sim}
		\mathbb F_{T^{\mathrm s}}$

	be the seed substitution along \(\boldsymbol k_\varepsilon\).
	Lemma~\ref{lem:fsb-seed-covariance} gives
	\begin{equation*}
		\mu_\varepsilon^x
		\bigl(\CC_{U_\varepsilon,\mathrm{loc}}(Z)\bigr)
		=
		\CC_{T^{\mathrm s},\mathrm{loc}}(Z)
		\qquad
		(Z\in\Db(\cE_{\mathrm s})).
	\end{equation*}
	Define
		$\Theta_\varepsilon
		=
		\eta_{\mathrm s}\circ\mu_\varepsilon^x
		\circ f_\varepsilon:
		L_{\mathrm t}\longrightarrow K.$

	For every \(Y\in\Db(\cE_{\mathrm t})\), we then have
	\begin{equation*}
		\begin{aligned}
			\Theta_\varepsilon
			\bigl(\CC_{T^{\mathrm t},\mathrm{loc}}(Y)\bigr)
			&=
			\eta_{\mathrm s}\mu_\varepsilon^x
			\bigl(
			\CC_{U_\varepsilon,\mathrm{loc}}(F_\varepsilon Y)
			\bigr)\\
			&=
			\eta_{\mathrm s}
			\bigl(
			\CC_{T^{\mathrm s},\mathrm{loc}}(F_\varepsilon Y)
			\bigr)\\
			&=
			\Phi^{\mathrm s}(j_{\mathrm s}F_\varepsilon Y)\\
			&=
			\Phi^{\mathrm s}
			\bigl((j_{\mathrm t}Y)[-\varepsilon]\bigr).
		\end{aligned}
	\end{equation*}
	
	For \(1\leq j\leq m\), the initial-summand identities of
	\cite[Corollary~6.15 and Proposition~7.1]{Pressland} give
	\begin{equation}
		\label{eq:fsb-twist-calibration}
		\begin{aligned}
			\Theta_\varepsilon(x_j^{\mathrm t})
			&=
			\Phi^{\mathrm s}
			\bigl(
			(j_{\mathrm t}T_j^{\mathrm t})[-\varepsilon]
			\bigr)\\
			&=
			\xi_\varepsilon
			\bigl(
			\Phi^{\mathrm t}(j_{\mathrm t}T_j^{\mathrm t})
			\bigr)\\
			&=
			\xi_\varepsilon\eta_{\mathrm t}(x_j^{\mathrm t}).
		\end{aligned}
	\end{equation}

	Their values are nonzero, so
	\begin{equation*}
		\begin{aligned}
			\Theta_\varepsilon
			\bigl((x_j^{\mathrm t})^{-1}\bigr)
			&=
			\bigl(
			\xi_\varepsilon\eta_{\mathrm t}(x_j^{\mathrm t})
			\bigr)^{-1}\\
			&=
			\xi_\varepsilon\eta_{\mathrm t}
			\bigl((x_j^{\mathrm t})^{-1}\bigr).
		\end{aligned}
	\end{equation*}
	Therefore
		$\Theta_\varepsilon
		=
		\left.
		(\xi_\varepsilon\circ\eta_{\mathrm t})
		\right|_{L_{\mathrm t}}.$

	For arbitrary \(X\in\Db(\fgmod B)\), put
	\(Y=j_{\mathrm t}^{-1}X\). Since
	\(\CC_{T^{\mathrm t},\mathrm{loc}}(Y)\in L_{\mathrm t}\),
	we obtain
	\begin{equation*}
		\begin{aligned}
			\Phi^{\mathrm s}(X[-\varepsilon])
			&=
			\Theta_\varepsilon
			\bigl(\CC_{T^{\mathrm t},\mathrm{loc}}(Y)\bigr)\\
			&=
			\xi_\varepsilon\eta_{\mathrm t}
			\bigl(\CC_{T^{\mathrm t},\mathrm{loc}}(Y)\bigr)\\
			&=
			\xi_\varepsilon
			\bigl(\Phi^{\mathrm t}(X)\bigr).
		\end{aligned}
	\end{equation*}
	Taking \(\varepsilon=1\) and \(\varepsilon=0\), respectively,
	proves the two asserted identities.
\end{proof}

\begin{theorem}
	\label{thm:fsb-rank-one-normalization}
	For every $M_I\in\CM(B)$, we have
	\begin{equation}
		\label{eq:fsb-all-rank-one}
		\Phi^{\mathrm s}(M_I)
		=\Phi^{\mathrm t}(M_I)=\Delta_I.
	\end{equation}
\end{theorem}

\begin{proof}
	Choose a projective cover, that is, a right-minimal projective
	deflation onto $M_I$, and denote its kernel (the first syzygy) by
	$\Omega_BM_I$:
	\begin{equation}
		\label{eq:fsb-rank-one-cover}
		0\longrightarrow\Omega_BM_I\longrightarrow Q_I
		\longrightarrow M_I\longrightarrow0,
		\qquad Q_I\in\add B.
	\end{equation}
	By \cite[Lemma~10.4]{CKP}, $\Omega_BM_I\in\cE_{\mathrm s}$.
	Rotating the associated triangle gives
	\[
	Q_I[-1]\longrightarrow M_I[-1]
	\longrightarrow\Omega_BM_I\longrightarrow Q_I.
	\]
	Since $Q_I[-1]$ is perfect, relative localization yields
	\begin{equation}
		\label{eq:fsb-rank-one-ratio}
		\begin{aligned}
			\Phi^{\mathrm s}(M_I[-1])
			&=\Phi^{\mathrm s}(Q_I[-1])\Phi^{\mathrm s}(\Omega_BM_I)\\
			&=\frac{\Phi^{\mathrm s}(\Omega_BM_I)}{\Phi^{\mathrm s}(Q_I)}.
		\end{aligned}
	\end{equation}
	The rank-one twist formula \cite[Theorem~12.2]{CKP} identifies this
	ratio as
	\begin{equation}
		\label{eq:fsb-known-rank-one-twist}
		\frac{\Phi^{\mathrm s}(\Omega_BM_I)}{\Phi^{\mathrm s}(Q_I)}
		=\tau(\Delta_I).
	\end{equation}
	On the other hand, by Lemma~\ref{lem:fsb-all-twist}, we have
	\[
	\Phi^{\mathrm s}(M_I[-1])
	=\tau\bigl(\Phi^{\mathrm t}(M_I)\bigr).
	\]
	Since $\tau$ is an automorphism \cite{MS}, then we haveg
	$\Phi^{\mathrm t}(M_I)=\Delta_I$.
	Equation~\eqref{eq:fsb-all-source-target} gives the source identity.
\end{proof}

\subsection{Recovering rigid objects from characters}

Put $T=T^{\mathrm s}$, and let
\[
\bB=\bB_T\in M_{m,r}(\ZZ),
\qquad m=r+s,
\]
where $r$ and $s$ are the numbers of non-projective and
projective-injective indecomposable summands of $T$, respectively.
Reachability is understood with respect to $T$.

An object $Y\in\cC$ is called rigid if
$\Hom_{\cC}(Y,\Sigma Y)=0$.
For $H\in\cE_{\mathrm s}$, choose a triangle
\[
A_1\longrightarrow A_0\longrightarrow qH
\longrightarrow\Sigma A_1,
\qquad A_0,A_1\in\add(qT).
\]
The stable index of $qH$ with respect to $qT$ is
\[
\ind_{qT}(qH)
=[A_0]-[A_1]
\in K_0^{\mathrm{sp}}(\add(qT))\simeq\ZZ^r.
\]
It is independent of the chosen presentation triangle and satisfies
\[
\ind_{qT}(qH)=\operatorname{pr}_{\mathrm{mut}}(\ind_T H),
\]
where
$\operatorname{pr}_{\mathrm{mut}}:\ZZ^{r+s}\to\ZZ^r$
forgets the projective-injective coordinates.

\begin{lemma}[Full column rank]
	\label{lem:fsb-full-rank}
Let $G$ be a reduced plabic graph whose boundary vertices
are labelled $1,\ldots,n$ in clockwise order.
Let $\Sigma$ be its source-labelled or target-labelled seed,
or any seed obtained from one of these seeds by a finite
sequence of mutations. If $\Sigma$ has
	$r$ mutable variables and $s$ frozen variables, then its
	extended exchange matrix
	\[
	\bB_\Sigma\in M_{r+s,r}(\ZZ)
	\]
	has rank $r$. Equivalently, the homomorphism
	\[
	\bB_\Sigma:\ZZ^r\longrightarrow\ZZ^{r+s}
	\]
	is injective.
\end{lemma}

\begin{proof}
	We first notice that matrix mutation preserves rank \cite[Lemma~3.2]{BFZIII}. By \cite[Theorem~6.5]{Postnikov}, choose a Le-diagram representing the same positroid as $G$,
	and let $G_{\mathrm{Le}}$ be its associated reduced plabic
	graph \cite[Section~20]{Postnikov}. By \cite[Proposition~4.9]{GL}, the extended exchange
	matrix $\bB_{\mathrm{Le}}^{\mathrm{deh}}$ of the Le seed has full column rank $r$.
	
	In the conventions of \cite[Section~1.3]{GL}, the distinguished
	boundary face $F_0$ is omitted from the quiver. Restoring this
	face gives the homogeneous seed used here and adds one frozen
	row to the exchange matrix. Thus, after ordering the faces
	and choosing the same sign convention,
	\[
	\bB_{\mathrm{Le}}
	=
	\begin{pmatrix}
		\bB_{\mathrm{Le}}^{\mathrm{deh}}\\
		b_0
	\end{pmatrix}
	\]
	for some row vector $b_0$. Consequently,
	\[
	r
	=\operatorname{rank}\bB_{\mathrm{Le}}^{\mathrm{deh}}
	\leq\operatorname{rank}\bB_{\mathrm{Le}}
	\leq r,
	\]
	and hence $\operatorname{rank}\bB_{\mathrm{Le}}=r$.
	
	The graphs $G$ and $G_{\mathrm{Le}}$ have the same decorated
	trip permutation. By \cite[Theorem~13.4]{Postnikov}, they are
	connected by plabic moves. Square moves induce matrix mutation
	at the corresponding mutable faces, while the other moves
	preserve the extended exchange matrix under the natural
	identification of faces. Combined with the fact that mutation preserves rank, this yields
	\[
	\operatorname{rank}\bB_G=r.
	\]
	Source and target labellings assign coordinates to the same
	faces of the same dual ice quiver. Thus their exchange matrices
	agree when the faces are ordered identically, and the rank
	statement applies to both seeds. Finally, every additional seed mutation preserves rank.
	Therefore $\operatorname{rank}\bB_\Sigma=r$ for every seed
	$\Sigma$ in the statement.
\end{proof}

\begin{lemma}
	\label{lem:fsb-pointed-recovery}
	Let $\bB\in M_{m,r}(\ZZ)$ have full column rank, and let
	$x_1,\ldots,x_m$ be algebraically independent indeterminates.
	For $v=(v_1,\ldots,v_m)\in\ZZ^m$, write
	\[
	x^v=\prod_{i=1}^m x_i^{v_i},
	\qquad
	x^{\bB}=(x^{\bB e_1},\ldots,x^{\bB e_r}),
	\]
	where $e_1,\ldots,e_r$ are the standard basis vectors of
	$\ZZ^r$.
	
	Suppose that $g,h\in\ZZ^m$ and
	$F,H\in\ZZ[y_1,\ldots,y_r]$ satisfy
$	F(0,\ldots,0)=H(0,\ldots,0)=1.$
	If
	\begin{equation}
		\label{eq:fsb-pointed-equality}
		x^gF(x^{\bB})=x^hH(x^{\bB})
		\quad\text{in }\ZZ[x_1^{\pm1},\ldots,x_m^{\pm1}],
	\end{equation}
	then
	\[
	g=h
	\qquad\text{and}\qquad
	F=H.
	\]
	Thus an expression $x^gF(x^{\bB})$ with $F(0)=1$
	uniquely determines both its monomial $x^g$
	and its polynomial factor $F$.
\end{lemma}

\begin{proof}
	Write
	\[
	F(y)=\sum_{u\in\mathbb N^r}c_u y^u,
	\qquad
	H(y)=\sum_{v\in\mathbb N^r}d_v y^v,
	\]
	where only finitely many coefficients are nonzero and
	$c_0=d_0=1$.
	Since $\bB$ has full column rank, the map
	\[
	\ZZ^r\longrightarrow\ZZ^m,
	\qquad u\longmapsto\bB u
	\]
	is injective. Consequently, distinct monomials in $F$ or $H$
	remain distinct after the substitution $y\mapsto x^{\bB}$.
	
	The assumed equality becomes
	\[
	\sum_{u\in\mathbb N^r}c_u x^{g+\bB u}
	=
	\sum_{v\in\mathbb N^r}d_v x^{h+\bB v}.
	\]
	The coefficient of $x^g$ on the left is $c_0=1$.
	Hence there exists $v\in\mathbb N^r$ with $d_v\ne0$ such that
	\[
	g=h+\bB v.
	\]
	Similarly, the coefficient of $x^h$ on the right is $d_0=1$,
	so there exists $u\in\mathbb N^r$ with $c_u\ne0$ such that
	\[
	h=g+\bB u.
	\]
	Adding these identities gives
	$\bB(u+v)=0.$ By injectivity of $\bB$, we have $u+v=0$. Since every coordinate
	of $u$ and $v$ is nonnegative, we obtain $u=v=0$, and therefore
	$g=h$.
	
	By cancelling the invertible monomial $x^g$, we obtain
	\[
	F(x^{\bB})=H(x^{\bB}).
	\]
	The Laurent monomials $x^{\bB e}$, for $e\in\mathbb N^r$,
	are pairwise distinct and hence linearly independent over $\ZZ$.
	Comparing their coefficients gives $c_e=d_e$ for every $e$,
	so $F=H$.
\end{proof}

\begin{lemma}
	\label{lem:fsb-singular-rigidity}
	For $M,N\in\CM(B)$, there are natural maps
	\begin{equation}
		\label{eq:fsb-ext-surjection}
		\Ext_C^1(M,N)\hookleftarrow\Ext_B^1(M,N)
		\twoheadrightarrow\Hom_{\cC}(qM,\Sigma qN).
	\end{equation}
	If $I$ and $J$ are weakly separated, then
	\[
	\Hom_{\cC}(qM_I,\Sigma qM_J)=0.
	\]
	In particular, every $qM_I$ is rigid.
\end{lemma}

\begin{proof}
	Suppose an extension
	\[
	0\longrightarrow N\longrightarrow E\xrightarrow{p}M
	\longrightarrow0
	\]
	splits after restriction to $C$. By full faithfulness, its
	$C$-linear splitting is $B$-linear. This proves the injectivity
	of the first map.
	
	Choose a projective cover
	$0\to\Omega N\to Q_N\to N\to0$. By \cite[Lemma~10.4]{CKP}, $\Omega M$ and $\Omega N$ are
	Gorenstein-projective. Thus
	\[
	\Ext_B^2(M,Q_N)=\Ext_B^1(\Omega M,Q_N)=0.
	\]
	The associated long exact sequence and dimension shifting give
	\[
	\begin{aligned}
		\Ext_B^1(M,N)
		&\twoheadrightarrow\Ext_B^2(M,\Omega N)\\
		&\simeq\Ext_B^1(\Omega M,\Omega N)\\
		&\simeq\Hom_{\cC}(q\Omega M,\Sigma q\Omega N)\\
		&\simeq\Hom_{\cC}(qM,\Sigma qN).
	\end{aligned}
	\]
	Here we use the stable Gorenstein-projective realization of $\cC$.
	For weakly separated $I,J$, the vanishing
	$\Ext_C^1(M_I,M_J)=0$ follows from \cite{JKS}.
	Equation~\eqref{eq:fsb-ext-surjection} gives the assertion.
\end{proof}
We now use cluster characters to detect reachability in the stable
category. Each source cluster monomial is the character of a
reachable rigid object. If the character of another object agrees
with this monomial up to a frozen Laurent factor,
Lemma~\ref{lem:fsb-pointed-recovery} shows that their indices have
the same mutable coordinates, and hence that their stable indices
coincide. Since rigid objects are determined by their stable
indices, this gives the following recognition criterion.
\begin{proposition}
	\label{prop:fsb-character-reachability}
	Let $X\in\Db(\fgmod B)$. If $qX$ is rigid and
	$\Phi^{\mathrm s}(X)$ is a source cluster monomial multiplied by
	a frozen Laurent monomial, then $qX$ is reachable rigid.
	In particular, this holds for every rank-one module
	$M_I\in\CM(B)$.
\end{proposition}

\begin{proof}
	Put $T=T^{\mathrm s}$, and write
	\[
	x=(x_1,\ldots,x_r,z_1,\ldots,z_s),
	\qquad
	\bB=\bB_T.
	\]
	For $a\in\ZZ^s$, denote its inclusion into the frozen
	coordinates by
	\[
	\iota(a)=(0,\ldots,0,a_1,\ldots,a_s)\in\ZZ^{r+s}.
	\]
	Thus $x^{\iota(a)}=z^a$.
	
	By hypothesis, there exist a source seed $t$,
	integers $d_i\geq 0$, and $b\in\ZZ^s$ such that
	\[
	\eta_{\mathrm s}^{-1}\bigl(\Phi^{\mathrm s}(X)\bigr)
	=
	z^b\prod_{i=1}^r x_{i;t}^{d_i}.
	\]
	Choose a reachable cluster-tilting object
	$U=\bigoplus_{i=1}^r U_i
	\oplus\bigoplus_{p=1}^s P_p$
	corresponding to this seed, with $U_i$ corresponding to
	$x_{i;t}$, and put
$	L=\bigoplus_{i=1}^r U_i^{\oplus d_i}.$

	The character correspondence and multiplicativity give
	\[
	\eta_{\mathrm s}^{-1}\bigl(\Phi^{\mathrm s}(L)\bigr)
	=
	\prod_{i=1}^r x_{i;t}^{d_i}.
	\]
	Moreover, $qL\in\add(qU)$, so $qL$ is reachable rigid.
	
	Choose $H\in\cE_{\mathrm s}$ with $qH\simeq qX$.
	Lemma~\ref{lem:fsb-projective-correction} gives
	$a\in\ZZ^s$ such that
	\[
	\kappa\bigl(j_{\mathrm s}^{-1}X\bigr)
	=
	[H]+\sum_{p=1}^s a_p[P_p].
	\]
	By the localized character formula, we have
	\[
	\Phi^{\mathrm s}(X)
	=
	\eta_{\mathrm s}(z^a)\Phi^{\mathrm s}(H).
	\]
	Combining these identities and applying
	$\eta_{\mathrm s}^{-1}$ yields
	\[
	z^a\CC_T(H)=z^b\CC_T(L).
	\]
	Expanding both characters in the initial source seed, we obtain
	\begin{equation*}
		x^{\ind_T H+\iota(a)}
		F_T^H(x^{\bB})
		=
		x^{\ind_T L+\iota(b)}
		F_T^L(x^{\bB}).
	\end{equation*}
	
	By Lemma~\ref{lem:fsb-full-rank}, $\bB$ has full
	column rank. Since
$	F_T^H(0)=F_T^L(0)=1,$ then
	Lemma~\ref{lem:fsb-pointed-recovery} implies
	\[
	\ind_T H+\iota(a)
	=
	\ind_T L+\iota(b).
	\]
	Le
	$\operatorname{pr}_{\mathrm{mut}}:
	\ZZ^{r+s}\longrightarrow\ZZ^r$
	be the projection onto the mutable coordinates.
	Applying this projection gives
	\begin{align*}
		\ind_{qT}(qH)
		&=
		\operatorname{pr}_{\mathrm{mut}}(\ind_T H)\\
		&=
		\operatorname{pr}_{\mathrm{mut}}(\ind_T L)
		=
		\ind_{qT}(qL).
	\end{align*}
	The object $qH$ is rigid because $qH\simeq qX$,
	and $qL$ is rigid because it belongs to $\add(qU)$.
	Rigid objects are determined up to isomorphism by their
	indices; see \cite[Section~2.3]{DehyKeller}.
	Consequently,
	\[
	qX\simeq qH\simeq qL.
	\]
	Since $qL$ is reachable rigid, this proves the first assertion.
	
	Now let $M_I\in\CM(B)$ be a rank-one module.
	By Lemma~\ref{lem:fsb-singular-rigidity},we conclude that $qM_I$
	is rigid, while
	Theorem~\ref{thm:fsb-rank-one-normalization} gives
	$\Phi^{\mathrm s}(M_I)=\Delta_I.$
	
	Since $M_I\in\CM(B)$, the coordinate $\Delta_I$ is nonzero
	in $R_\pi$ by \eqref{eq:fsb-boundary-restriction}.
	By \cite[Theorem~8.9]{MMMSV}, it can be written as
	\[
	\Delta_I=p_Ic_I,
	\qquad p_I\in\mathscr F_\pi,
	\]
	where $c_I$ is a monomial in mutable variables belonging
	to a single source cluster.
	Thus $M_I$ satisfies both hypotheses of the first assertion,
	and $qM_I$ is reachable rigid.
\end{proof}

\begin{lemma}
	\label{lem:fsb-mutable-irreducible}
	Let $\mathcal A=\mathcal A(\Sigma)$ be a cluster algebra
	with inverted frozen variables, as in
	Definition~\ref{def:fsb-cluster-structure}.
	Write $z_1,\ldots,z_s$ for its frozen variables.
	Assume that the extended exchange matrix of every seed
	has no zero column.
	Then every mutable cluster variable is an irreducible
	nonunit of $\mathcal A$. More precisely, if $x$ is a
	mutable cluster variable, then
	\[
	x\notin\mathcal A^\times,
	\]
	and every factorization
	$$x=fg,\qquad f,g\in\mathcal A,$$
	has at least one factor in $\mathcal A^\times$.
	Here $\mathcal A^\times$ denotes the group of units
	of $\mathcal A$.
\end{lemma}

\begin{proof}
	Fix a seed containing the given mutable cluster variable,
	and denote its extended cluster by
	\[
	(x_1,\ldots,x_r,z_1,\ldots,z_s).
	\]
	We prove the assertion for $x_k$.
	Put
	\[
	\mathcal L
	=
	\kk[z_1^{\pm1},\ldots,z_s^{\pm1}]
	[x_1^{\pm1},\ldots,x_r^{\pm1}].
	\]
	For each mutable index $\ell$, set
	\[
	\mathcal L_\ell
	=
	\kk[z_1^{\pm1},\ldots,z_s^{\pm1}]
	[x_j^{\pm1}:1\leq j\leq r,\ j\neq\ell],
	\]
	and write the exchange relation at $\ell$ as
	\[
	x_\ell x_\ell'=E_\ell,
	\qquad E_\ell\in\mathcal L_\ell.
	\]
	Since the $\ell$-th extended exchange column is nonzero,
	the two monomials in $E_\ell$ are distinct.
	The units of a Laurent polynomial ring over $\kk$ are
	precisely the nonzero scalar multiples of Laurent
	monomials. Hence
	$E_\ell\notin\mathcal L_\ell^\times.$
	
	The Laurent ring of the mutated seed is
	\[
	\mathcal L^{(\ell)}
	=
	\mathcal L_\ell[(x_\ell')^{\pm1}].
	\]
	By the Laurent phenomenon, we have
	\[
	\mathcal A\subseteq\mathcal L,
	\qquad
	\mathcal A\subseteq\mathcal L^{(\ell)}
	\quad(1\leq\ell\leq r).
	\]
	
	First, $x_k$ is not a unit of $\mathcal A$.
	Indeed, if $x_k^{-1}\in\mathcal A$, then
	\[
	x_k^{-1}=\frac{x_k'}{E_k}
	\in\mathcal L_k[(x_k')^{\pm1}].
	\]
	Uniqueness of Laurent expansion in the indeterminate $x_k'$
	would imply $E_k^{-1}\in\mathcal L_k$, contradicting
	$E_k\notin\mathcal L_k^\times$.
	
	Now suppose that
	$x_k=fg, f,g\in\mathcal A.$
	Since $x_k$ is a unit of $\mathcal L$, both $f$ and $g$
	are units of $\mathcal L$. Thus there exist
	$c\in\kk^\times$, $a\in\ZZ^r$, and $b\in\ZZ^s$ such that
	\[
	f=cx^az^b,
	\qquad
	g=c^{-1}x^{e_k-a}z^{-b}.
	\]
	
	Fix \(\ell\in[r]\), and let \(\mu_\ell(\Sigma)\)
	be the seed obtained from \(\Sigma\) by mutation at \(x_\ell\).
	Expressing \(f\) in the cluster coordinates of
	\(\mu_\ell(\Sigma)\), we obtain
	\[
	f
	=
	\left(
	cz^b\prod_{j\neq\ell}x_j^{a_j}
	\right)
	E_\ell^{a_\ell}(x_\ell')^{-a_\ell}.
	\]
	The expression in parentheses is a unit of
	$\mathcal L_\ell$.
	Since $f\in\mathcal L^{(\ell)}$, uniqueness of Laurent
	expansion in $x_\ell'$ implies
	\[
	E_\ell^{a_\ell}\in\mathcal L_\ell.
	\]
	If $a_\ell<0$, this would make $E_\ell$ a unit of
	$\mathcal L_\ell$. Therefore $a_\ell\geq0$.
	Applying the same argument to $g$ gives
	$\delta_{k\ell}-a_\ell\geq0,$ where \(\delta_{k\ell}\) denotes the Kronecker delta.
	
	Consequently,
	\[
	0\leq a_\ell\leq\delta_{k\ell}
	\qquad(1\leq\ell\leq r).
	\]
	Thus $a=0$ or $a=e_k$.
	In the first case, $f=cz^b$ is a unit of $\mathcal A$;
	in the second case, $g=c^{-1}z^{-b}$ is a unit.
	This proves the irreducibility of $x_k$.
\end{proof}

\subsection{Compatibility of face objects}

Let $G^\rho$ be an admissibly relabelled plabic graph for the positroid
$\pi$. Write
\[
\cS=\{I_F^{\mathrm{tgt}}(G^\rho):F\text{ a face}\},\qquad
\cI=\{I_F^{\mathrm{tgt}}(G^\rho):F\text{ a boundary face}\},
\]
as sets of distinct labels. We use the following consequences of
admissibility from \cite[Theorems~4.14 and~4.21]{FSB}:
\begin{itemize}
	\item the face
	coordinates form a seed for $R_\pi$.
	\item the boundary coordinates are a
	$\mathbb Z$-basis of the coefficient group $\mathscr F_\pi$.
	\item the
	collection $\cS$ is weakly separated and $|\cS\setminus\cI|=r$.
\end{itemize}

Here two $k$-subsets $I,J$ are weakly separated if there are no four
cyclically ordered elements $a,b,c,d$ with
$a,c\in I\setminus J$ and $b,d\in J\setminus I$.

We recall the combinatorial $F$-polynomials and the compatibility
criterion that we will use to place the face objects in one cluster.

\begin{definition}[Combinatorial $F$-polynomials]
	\label{def:fsb-combinatorial-F}
	Fix a seed $u$ with mutable variables
	$(u_1,\ldots,u_r)$, and let $B_u$ be the principal part
	of its extended exchange matrix. Consider the cluster pattern with principal coefficients. Assume that the initial
	mutable variables $(v_1,\ldots,v_r)$, frozen variables
	$(y_1,\ldots,y_r)$, and extended exchange matrix
	\[
	\begin{pmatrix}
		B_u\\
		I_r
	\end{pmatrix}.
	\]
	This is the cluster pattern with principal coefficients
	at $v$.
	
	Let $z$ be a mutable cluster variable of the original
	cluster algebra. Choose a mutation sequence from $u$
	to a seed containing $z$, and perform the same mutations
	in the cluster pattern with principal coefficients. Denote the corresponding
	cluster variable by
	\[
	\mathcal X^{\mathrm{prin}}_{z,u}(v_1,\ldots,v_r;
	y_1,\ldots,y_r).
	\]
	The \emph{combinatorial $F$-polynomial of $z$ relative to $u$}
	is
	\[
	F^{\mathrm{comb}}_{z,u}(y_1,\ldots,y_r)
	=
	\left.
	\mathcal X^{\mathrm{prin}}_{z,u}(v_1,\ldots,v_r;
	y_1,\ldots,y_r)
	\right|_{v_1=\cdots=v_r=1}
	\in\ZZ[y_1,\ldots,y_r].
	\]
	It is independent of the chosen mutation sequence
	representing $z$; see \cite[Remark~2.19]{FuGyoda}.
\end{definition}

These polynomials satisfy
\[
F^{\mathrm{comb}}_{z,u}(0,\ldots,0)=1,
\qquad
F^{\mathrm{comb}}_{u_i,u}=1
\quad(1\leq i\leq r);
\]
see \cite[Proposition~2.14]{FuGyoda}.
The superscript ``$\mathrm{comb}$'' distinguishes them from
the Grassmannian polynomials defined by cluster characters.

To relate this construction to the original coefficients, write
\[
\mathbf u=(u_1,\ldots,u_r,w_1,\ldots,w_s),
\qquad
\widehat y_{j,u}=\mathbf u^{\bB_u e_j},
\]
where $w_1,\ldots,w_s$ are the frozen variables and $\bB_u$
is the full extended exchange matrix.
The separation formula
\cite[Corollary~6.3]{FZIV} gives
\[
z=
\left(\prod_{i=1}^r u_i^{g_i}\right)
\left(\prod_{p=1}^s w_p^{a_p}\right)
F^{\mathrm{comb}}_{z,u}
(\widehat y_{1,u},\ldots,\widehat y_{r,u})
\]
for some $g\in\ZZ^r$ and $a\in\ZZ^s$.

\begin{definition}[$f$-compatibility degree]
	\label{def:fsb-f-compatibility}
	Let $z_0$ and $z$ be mutable cluster variables.
	Choose a seed $u$ containing $z_0$, and write $z_0=u_i$.
	Their \emph{$f$-compatibility degree} is
	\[
	(z_0\Vert z)_f
	=
	\deg_{y_i}F^{\mathrm{comb}}_{z,u}(y_1,\ldots,y_r).
	\]
	By \cite[Theorem~3.3(2)]{FuGyoda}, this integer depends
	only on the ordered pair $(z_0,z)$, and not on the
	choice of the seed containing $z_0$.
\end{definition}

The following theorem turns the vanishing of these degrees
into a criterion for belonging to a common cluster.

\begin{theorem}[{\cite[Theorem~4.18]{FuGyoda}}]
	\label{thm:fsb-f-compatibility}
	Let $\mathcal Z$ be a finite set of distinct mutable
	cluster variables. Then $\mathcal Z$ is contained in
	a single cluster if and only if
	\[
	(z\Vert z')_f=0
	\qquad
	\text{for all distinct }z,z'\in\mathcal Z.
	\]
\end{theorem}

Recall that
\[
\cC=\Dsg(B)=\Db(\fgmod B)/\operatorname{per}B,
\]
and that \(q:\Db(\fgmod B)\to\cC\) denotes the
canonical quotient functor.
We identify \(\underline{\cE}_{\mathrm s}\) with \(\cC\)
via the equivalence in \eqref{eq:fsb-standard-derived}.
\begin{proposition}
	\label{prop:fsb-common-face-cluster}
	Let $G^\rho$ be an admissibly relabelled plabic graph
	for $\pi$, with face-label set $\cS$ and boundary-label
	set $\cI$. There exists a seed in the source cluster structure
	whose mutable cluster variables can be indexed as
	$\{z_J:J\in\cS\setminus\cI\},$
	such that, for each $J\in\cS\setminus\cI$, there is
	a frozen Laurent monomial $p_J\in\mathscr F_\pi$ satisfying
	\begin{equation}
		\label{eq:fsb-face-source-variable}
		\Delta_J=p_Jz_J.
	\end{equation}
	
	Under the source categorification, the variable $z_J$
	corresponds to the stable object $qM_J$.
	Consequently, the object
	$\bigoplus_{J\in\cS\setminus\cI}qM_J$
	is a basic cluster-tilting object of $\cC$,
	reachable from $qT^{\mathrm s}$.
\end{proposition}

\begin{proof}
	Put $T=T^{\mathrm s}$.
	By admissibility, the relabelled seed and the source seeds
	define cluster structures on the same ring $R_\pi$.
	We therefore regard all their cluster variables as elements
	of this ring.
	
	We first identify each mutable face coordinate with a
	single source cluster variable, up to a frozen Laurent factor.
	Fix $J\in\cS\setminus\cI$.
	By \cite[Theorem~8.9]{MMMSV}, we can write
	\[
	\Delta_J
	=
	p_J\prod_{\nu=1}^r\zeta_\nu^{a_\nu},
	\qquad
	p_J\in\mathscr F_\pi,\quad a_\nu\in\NN,
	\]
	where $\zeta_1,\ldots,\zeta_r$ are the mutable variables
	of one source seed.
	
	In the relabelled cluster structure, $\Delta_J$ is a mutable
	cluster variable. The extended exchange matrix of this seed
	agrees with that of the underlying ordinarily labelled graph.
	It has full column rank by Lemma~\ref{lem:fsb-full-rank},
	and this rank is preserved under mutation.
	Thus Lemma~\ref{lem:fsb-mutable-irreducible} applies and
	shows that $\Delta_J$ is an irreducible nonunit of $R_\pi$.
	The same lemma shows that each source variable $\zeta_\nu$
	is a nonunit, whereas $p_J$ is a unit. Hence
	\[
	\sum_{\nu=1}^r a_\nu=1.
	\]
	Indeed, a sum of zero would make $\Delta_J$ a unit,
	while a sum greater than one would give a factorization
	into two nonunits.
	Consequently, there is a source cluster variable $z_J$ with
	\[
	\Delta_J=p_Jz_J.
	\]
	
	These source variables are pairwise distinct.
	Suppose that $z_J=z_K$. Then
	\[
	\frac{\Delta_J}{\Delta_K}
	=
	\frac{p_J}{p_K}
	\in\mathscr F_\pi.
	\]
	By admissibility, the boundary coordinates
	$\{\Delta_I:I\in\cI\}$ form a $\ZZ$-basis of
	$\mathscr F_\pi$. Thus there exist integers $c_I$ such that
	\[
	\frac{\Delta_J}{\Delta_K}
	=
	\prod_{I\in\cI}\Delta_I^{c_I}.
	\]
	Since the extended-cluster coordinates
	$\{\Delta_I:I\in\cS\}$ are algebraically independent,
	this Laurent monomial identity forces $J=K$.
	
	We next identify the stable objects corresponding to these
	variables. For each $J$, choose a reachable non-projective
	indecomposable object $L_J\in\cE_{\mathrm s}$ satisfying
	\[
	\Phi^{\mathrm s}(L_J)=z_J.
	\]
	By Theorem \ref{thm:fsb-rank-one-normalization}, we have
	\[
	\Phi^{\mathrm s}(M_J)
	=
	\Delta_J
	=
	p_J\Phi^{\mathrm s}(L_J).
	\]
	We apply the index comparison from the proof of
	Proposition~\ref{prop:fsb-character-reachability}
	to this equality.
	
	More explicitly, choose $H_J\in\cE_{\mathrm s}$ with
	$qH_J\simeq qM_J$.
	Let $x$ be the initial extended cluster at $T$, and write
	\[
	\iota:\ZZ^s\longrightarrow\ZZ^{r+s},
	\qquad
	\iota(a)=(0,a),
	\]
	for the inclusion into the frozen coordinates.
	By Lemma~\ref{lem:fsb-projective-correction},
	the localization identity
	\eqref{eq:fsb-localized-character-factor},
	and the cluster character formula
	\eqref{eq:fsb-full-character},
	there exist $a_J,b_J\in\ZZ^s$ such that, after applying
	$\eta_{\mathrm s}^{-1}$, the preceding equality becomes
	\[
	x^{\ind_T H_J+\iota(a_J)}
	F_T^{H_J}(x^{\bB_T})
	=
	x^{\ind_T L_J+\iota(b_J)}
	F_T^{L_J}(x^{\bB_T}).
	\]
	Both $F$-polynomials have constant term one.
	Lemmas~\ref{lem:fsb-full-rank}
	and~\ref{lem:fsb-pointed-recovery} therefore give
	\[
	\ind_T H_J+\iota(a_J)
	=
	\ind_T L_J+\iota(b_J).
	\]
	Projecting onto the mutable coordinates yields
	\[
	\ind_{qT}(qH_J)=\ind_{qT}(qL_J).
	\]
	By Lemma~\ref{lem:fsb-singular-rigidity}, $qH_J\simeq qM_J$
	is rigid, and $qL_J$ is rigid by construction.
	The uniqueness of rigid objects with a prescribed index
	\cite[Section~2.3]{DehyKeller} now implies
	\[
	qM_J\simeq qH_J\simeq qL_J.
	\]
	
	Admissibility also gives weak separation of the face labels.
	Lemma~\ref{lem:fsb-singular-rigidity} consequently implies
	\[
	\Hom_{\cC}(qL_J,\Sigma qL_K)
	\simeq
	\Hom_{\cC}(qM_J,\Sigma qM_K)
	=0
	\qquad
	(J,K\in\cS\setminus\cI).
	\]
	We now translate this vanishing into compatibility of
	the corresponding source cluster variables.
	
	Fix distinct labels $J,K\in\cS\setminus\cI$.
	Choose a reachable cluster-tilting object
	$U\in\cE_{\mathrm s}$ having $L_J$ as a summand,
	and order its non-projective summands so that $U_i=L_J$.
	Let $u$ be the corresponding source seed, with extended
	coordinates
	\[
	\mathbf u=(u_1,\ldots,u_r,w_1,\ldots,w_s).
	\]
	In particular, $u_i=z_J$.
	Regard these coordinates as algebraically independent
	elements of $\operatorname{Frac}(R_\pi)$.
	
	Put
$	V=\Hom_{\cC}(qU,\Sigma qL_K).$
	If $\epsilon_i$ denotes the idempotent corresponding to
	$qU_i$, then
	\[
	V\epsilon_i
	=
	\Hom_{\cC}(qU_i,\Sigma qL_K)
	=
	\Hom_{\cC}(qL_J,\Sigma qL_K)
	=0.
	\]
	Hence every submodule of $V$ has zero $i$-th component.
	It follows directly from the Grassmannian formula that
	\[
	\deg_{y_i}F_U^{L_K}(y)=0.
	\]
	
	To identify this polynomial with the combinatorial
	$F$-polynomial, we use
	Lemma~\ref{lem:fsb-seed-covariance}.
	Expressing the character of the fixed Frobenius object
	$L_K$ in the seed $u$ gives
	\[
	z_K
	=
	\Phi^{\mathrm s}(L_K)
	=
	\mathbf u^{\ind_U L_K}
	F_U^{L_K}(\mathbf u^{\bB_U}).
	\]
	On the other hand, the separation formula
	\cite[Corollary~6.3]{FZIV} gives
	\[
	z_K
	=
	\mathbf u^h
	F^{\mathrm{comb}}_{z_K,u}(\mathbf u^{\bB_U})
	\]
	for some $h\in\ZZ^{r+s}$, with the frozen
	factor included in $\mathbf u^h$.
	The matrix $\bB_U$ has full column rank, and both
	polynomials have constant term one.
	Lemma~\ref{lem:fsb-pointed-recovery} therefore implies
	\[
	F_U^{L_K}
	=
	F^{\mathrm{comb}}_{z_K,u}.
	\]
	Since $u_i=z_J$, we conclude that
	\[
	(z_J\Vert z_K)_f
	=
	\deg_{y_i}F^{\mathrm{comb}}_{z_K,u}
	=
	\deg_{y_i}F_U^{L_K}
	=0.
	\]
	
	Thus the distinct variables
	$\{z_J:J\in\cS\setminus\cI\}$ have pairwise vanishing
	$f$-compatibility degrees.
	By \cite[Theorem~4.18]{FuGyoda}, they belong to one
	source cluster.
	Admissibility gives
	\[
	|\cS\setminus\cI|=r,
	\]
	and every source cluster has exactly $r$ mutable variables.
	Hence these variables form the entire mutable cluster
	of a source seed.
	
	Let $U_*$ be the basic reachable Frobenius cluster-tilting
	object corresponding to this seed.
	Its non-projective summands are the objects $L_J$.
	Therefore
	\[
	qU_*
	\simeq
	\bigoplus_{J\in\cS\setminus\cI}qL_J
	\simeq
	\bigoplus_{J\in\cS\setminus\cI}qM_J.
	\]
	The direct sum above is consequently a basic
	cluster-tilting object of $\cC$, reachable from $qT$.
\end{proof}

\section{Tilting the boundary necklace and the main result}
\label{sec:fsb-tilting}

We first consider a connected loopless positroid of type $(k,n)$,
where $0<k<n$, and denote its trip permutation by $\pi$.
We work over $\mathbb C$ and retain the conventions of
Section~\ref{sec:fsb-character-input}.
All modules are left modules unless a right action is explicitly
specified.

Fix a reduced plabic graph $G$ with trip permutation $\pi$,
whose boundary vertices are labelled $1,\ldots,n$ in clockwise
order. Let $D$ be the associated Postnikov diagram obtained
from the trips of $G_0$. We use its boundary algebra
$B=eA_De$, and write
\[
Z=\kk\llbracket t\rrbracket,\qquad
K=\kk((t)),\qquad
C=C_{k,n}.
\]

Recall that $\mathcal M_\pi$ is the set of bases of the positroid
indexed by $\pi$, namely
\[
\mathcal M_\pi
=
\left\{
I\in\binom{[n]}{k}:
\Delta_I\neq 0\text{ in }R_\pi
\right\}.
\]
Thus $I\in\mathcal M_\pi$ means that $\Delta_I$ does not vanish
identically on $\widehat\Pi_\pi^\circ$.

Now let $\rho\in S_n$ be admissible for $\pi$ in the sense of
Definition~\ref{def:fsb-admissible}, and put
\[
\mu=\rho^{-1}\pi\rho,
\qquad
\iota=\pi\rho.
\]
Choose a reduced plabic graph $G$ with trip permutation $\mu$.
The relabelled graph $G^\rho$ then has trip permutation $\pi$.

Let
\[
\mathbf I=(I_1,\ldots,I_n)
\]
be the boundary Grassmannlike necklace of $G^\rho$, obtained
by reading its boundary target labels in the cyclic order
inherited from $G$. We denote the sets of boundary labels
and all face labels by
\[
\cI=\{I_a:a\in[n]\},
\qquad
\cS=
\{I_F^{\mathrm{tgt}}(G^\rho):
F\text{ is a face of }G^\rho\}.
\]
The sequence $\mathbf I$ records the cyclic order and retains
repeated labels. In the sets $\cI\subseteq\cS$, each label
is counted only once.

\begin{definition}[Boundary module]
	\label{def:fsb-boundary-module}
	For $I\in\cI$, let $M_I\in\CM(B)$ be the rank-one $B$-lattice
	whose restriction to $C=C_{k,n}$ is the circle module labelled by
	$I$, as in Section~\ref{sec:fsb-rank-one}. We call
	\begin{equation}
		\label{eq:fsb-boundary-module}
		P_{\cI}:=\bigoplus_{I\in\cI}M_I
	\end{equation}
	the \emph{boundary module associated with $\mathbf I$}. Each distinct
	boundary label contributes exactly one summand.
\end{definition}

Admissibility gives $\cI\subseteq\mathcal M_\pi$, so the rank-one
restriction criterion
\cite[Propositions~5.6 and~5.8]{Pressland}
ensures that each $M_I$, with $I\in\cI$, admits a $B$-lattice
lift unique up to isomorphism. Thus $P_{\cI}$ is well defined
up to isomorphism.
Here ``boundary'' refers to the boundary faces of $G^\rho$
indexing its summands.
We will establish the tilting properties of $P_{\cI}$ in
Theorem~\ref{prop:fsb-boundary-tilting}.

Put
\[
m=|\cS|,\qquad f=|\cI|,\qquad r=m-f.
\]
In the connected case, the ordinary boundary coordinates form
a basis of the coefficient group with $n$ elements.
By the Unit Necklace Theorem
(Theorem~\ref{thm:fsb-unit-necklace}), the coordinates indexed
by $\cI$ form another basis of the same group, so $f=n$.
The general count is established in
Lemma~\ref{lem:fsb-boundary-count}.
We order every face-labelled extended cluster with the
$r$ labels in $\cS\setminus\cI$ first and the $f$ labels
in $\cI$ last.

The preceding section identifies the face modules in the
standard categorical model.
Theorem~\ref{thm:fsb-rank-one-normalization} gives
\[
\Phi^{\mathrm s}(M_J)=\Delta_J
\qquad(J\in\cS),
\]
while Proposition~\ref{prop:fsb-common-face-cluster} shows that
the stable images of the mutable face modules form a basic
reachable cluster-tilting object.

We will construct a Frobenius categorification whose projective-injective
objects are indexed by the new boundary necklace.
We first prove that $P_{\cI}$ is a tilting $B$-module and then
use its endomorphism order
\[
B_{\cI}=\End_B(P_{\cI})^{\mathrm{op}}
\]
to construct the new categorification model. Its indecomposable
projective-injective objects will be
\[
\Hom_B(P_{\cI},M_I),
\qquad I\in\cI.
\]
Once this model and the comparison functor are constructed,
Proposition~\ref{prop:fsb-relative-recognition} will provide
the comparison of full indices and extended exchange matrices.

\subsection{Orders, lattices, and tilting modules}

\begin{definition}[Orders and lattices]
	\label{def:necklace-orders}
	A \emph{$Z$-order} is a unital $Z$-algebra $\Lambda$, with central
	$Z$-action, which is finite and free as a $Z$-module and for which
	$K\otimes_Z\Lambda$ is a finite-dimensional semisimple $K$-algebra.
	We write $\fgmod\Lambda$ for the category of finitely generated left
	$\Lambda$-modules and
	\[
	\CM(\Lambda)=\{M\in\fgmod\Lambda:M\text{ is free as a }Z\text{-module}\}.
	\]
	Its objects are called \emph{$\Lambda$-lattices}. Its conflations are
	short exact sequences of $\Lambda$-modules all of whose terms are
	lattices. Such sequences split over $Z$, but need not split over
	$\Lambda$. The dual order-module is
	\[
	\Lambda^\vee=\Hom_Z(\Lambda,Z),\qquad
	(a\varphi b)(c)=\varphi(bca).
	\]
\end{definition}

\begin{definition}[Gorenstein-projective modules]
	\label{def:necklace-gorenstein}
	A two-sided Noetherian ring $\Lambda$ is \emph{Iwanaga--Gorenstein} if
	\[
	\operatorname{id}_{\Lambda}\Lambda<\infty,
	\qquad
	\operatorname{id}_{\Lambda^{\mathrm{op}}}\Lambda<\infty.
	\]
	A finitely generated module is \emph{Gorenstein-projective} if it is a
	cycle in an exact complex of finitely generated projective modules
	which remains exact after applying $\Hom_\Lambda(-,\Lambda)$.
	We denote their full subcategory by $\GP\Lambda$. For an
	Iwanaga--Gorenstein ring,
	\[
	\GP\Lambda
	=\{M\in\fgmod\Lambda:\Ext_\Lambda^a(M,\Lambda)=0\text{ for all }a>0\}.
	\]
	With the inherited short exact sequences, $\GP\Lambda$ is Frobenius
	and its projective-injective objects are $\proj\Lambda$.
	For a $Z$-order, its Gorenstein-projective modules are lattices, since
	they embed into finitely generated projectives. The
	\emph{Gorenstein-projective dimension} $\operatorname{Gpd}_\Lambda M$
	is the least length of a resolution of $M$ by Gorenstein-projective
	modules, or $\infty$ if no such resolution exists.
\end{definition}

For an object $M$ of an additive category, $\add M$ consists of direct
summands of finite direct sums of copies of $M$. Inside
$\Db(\fgmod\Lambda)$, the notation $\thick(M)$ denotes the smallest
full triangulated subcategory containing $M$ and closed under direct
summands. In particular,
\[
\per\Lambda=\thick(\Lambda)\simeq K^b(\proj\Lambda),\qquad
\Dsg(\Lambda)=\Db(\fgmod\Lambda)/\per\Lambda.
\]
The first is the category of \emph{perfect complexes}, and the second
is the \emph{singularity category}. For an Iwanaga--Gorenstein ring,
Buchweitz's equivalence identifies
$\underline{\GP\Lambda}$ with $\Dsg(\Lambda)$ \cite{Buchweitz}.

\begin{definition}[Tilting module]
	\label{def:necklace-tilting}
	A finitely generated $\Lambda$-module $V$ is a \emph{tilting module}
	if
	\[
	\operatorname{pd}_\Lambda V<\infty,\qquad
	\Ext_\Lambda^a(V,V)=0\ (a>0),\qquad
	\thick(V)=\per\Lambda.
	\]
	It is a \emph{$1$-tilting module} if, in addition,
	$\operatorname{pd}_\Lambda V\leq1$. The last condition means generation
	in the derived category; it is not the assertion
	$\add V=\proj\Lambda$.
\end{definition}

Recall that for the boundary order $B$, restriction is an exact fully faithful
functor
\[
\operatorname{res}:\CM(B)\longrightarrow\CM(C).
\]
The rank-one circle module $M_I$ belongs to its essential image
exactly when $\Delta_I\ne0$ in $R_\pi$. We identify a $B$-lattice with
its restriction, but retain the algebra as a subscript on $\Hom$ and
$\Ext$ whenever it matters. We also use
\[
\Omega_B\CM(B)\subseteq\GP B,\qquad
\add\!\left(\bigoplus_{I\in\mathcal I^\pi}M_I\right)=\add B^\vee;
\]
see \cite[Propositions~5.6, 5.8, 5.25 and Corollary~5.22]{Pressland}.
Here $\Omega_BM$ is the kernel of a surjection from a finitely
generated projective module to $M$.

\subsection{Boundary tilting}

For $u,v\in[n]$,
$[u,v)$ denotes the clockwise cyclic interval starting at $u$ and
ending just before $v$. For a set $S$ disjoint from $\{u,v\}$, we
abbreviate $S\cup\{u,v\}$ to $Suv$.

Let $\mathbf J=(J_1,\ldots,J_n)$ be a Grassmannlike necklace
with removal permutation $\rho$, insertion permutation $\iota$,
and trip permutation $\pi=\iota\rho^{-1}$. Thus
\[
J_{a+1}
=
\bigl(J_a\setminus\{\rho(a)\}\bigr)\cup\{\iota(a)\}.
\]
All indices are read modulo $n$. We remain in the connected
case, where $\pi$ has no fixed points.

\begin{definition}[Toggle]
	\label{def:necklace-toggle}
	A \emph{toggle at position $a$} interchanges the consecutive
	removal-insertion steps at positions $a-1$ and $a$.
	It is allowed when
	\[
	\rho(a-1)\neq\iota(a),
	\qquad
	\rho(a)\neq\iota(a-1).
	\]
	The resulting necklace $\mathbf J'$ is defined by
	\[
	J'_a
	=
	\bigl(J_{a-1}\setminus\{\rho(a)\}\bigr)\cup\{\iota(a)\},
	\qquad
	J'_b=J_b\quad(b\neq a).
	\]
\end{definition}

Let $s_{a-1}$ exchange the cyclic positions $a-1$ and $a$.
The new removal and insertion permutations are
\[
\rho'=\rho s_{a-1},
\qquad
\iota'=\iota s_{a-1}.
\]
Consequently, the trip permutation is unchanged:
$\iota'(\rho')^{-1}=\pi$.

\begin{definition}[Aligned toggle]
	\label{def:necklace-aligned-toggle}
	Place the boundary labels $1,\ldots,n$ clockwise on a circle.
	Two disjoint directed chords
	$\alpha\to\beta$ and $\gamma\to\delta$, with four distinct
	endpoints, are \emph{aligned} if their clockwise cyclic
	order is
	\[
	(\alpha,\gamma,\delta,\beta)
	\qquad\text{or}\qquad
	(\alpha,\beta,\delta,\gamma).
	\]
	An allowed toggle at $a$ is \emph{aligned} if the chords
	\[
	\rho(a-1)\longrightarrow\iota(a-1),
	\qquad
	\rho(a)\longrightarrow\iota(a)
	\]
	are aligned.
\end{definition}

An \emph{aligned-toggle path} is a finite sequence
\[
\mathbf J^{(0)},\mathbf J^{(1)},\ldots,\mathbf J^{(N)}
\]
whose consecutive necklaces differ by an aligned toggle.
Since toggles preserve the trip permutation, all these necklaces
have the same trip permutation $\pi$.
The path is called a \emph{unit-necklace path} if
\[
\Delta_{J_b^{(\ell)}}\in\mathscr F_\pi
\qquad
(0\leq\ell\leq N,\ b\in[n]).
\]
The convention for loop chords, needed when $\pi$ has fixed
points, is given in Section~\ref{sec:fsb-disconnected}.

The condition $\iota\leq_\circ\pi$ supplies an aligned-toggle path
from the ordinary forward necklace to $\mathbf I$; all its terms
are unit necklaces \cite[Lemma~4.13 and Theorem~4.14]{FSB}.
The intermediate necklaces are not required to be weakly separated.


\begin{example}[An aligned toggle of unit necklaces]
	\label{ex:fsb-aligned-toggle}
	We write $ij$ for the subset $\{i,j\}$.
	Consider the permutation
	\[
	\pi=
	\begin{pmatrix}
		1&2&3&4&5\\
		4&3&5&1&2
	\end{pmatrix}.
	\]
	Its forward Grassmann necklace is
	\[
	\mathbf J=(12,24,34,45,15),
	\qquad
	\rho=\mathrm{id},
	\qquad
	\iota=\pi.
	\]
	Indeed, its consecutive removal-insertion steps are
	\[
	12
	\xrightarrow{\,1\mapsto4\,}24
	\xrightarrow{\,2\mapsto3\,}34
	\xrightarrow{\,3\mapsto5\,}45
	\xrightarrow{\,4\mapsto1\,}15
	\xrightarrow{\,5\mapsto2\,}12.
	\]
	
	We toggle at position $a=2$.
	The conditions for this operation are satisfied:
	\[
	\rho(1)=1\neq3=\iota(2),
	\qquad
	\rho(2)=2\neq4=\iota(1).
	\]
	The two chords involved are
	\[
	1\longrightarrow4,
	\qquad
	2\longrightarrow3.
	\]
	They are disjoint, and their endpoints occur in the clockwise
	order $1,2,3,4$. This is the order
	$(\alpha,\gamma,\delta,\beta)$ with
	$(\alpha,\beta,\gamma,\delta)=(1,4,2,3)$.
	Hence the toggle is aligned.
	
	
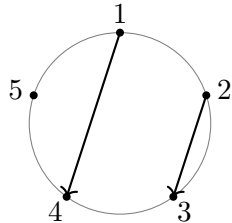
\begin{center}
		\begin{tikzpicture}
			\draw[gray] (0,0) circle (1.2);
			\foreach \j/\theta in
			{1/90,2/18,3/-54,4/-126,5/162}
			{
				\coordinate (p\j) at (\theta:1.2);
				\fill (p\j) circle (1.5pt);
				\node at (\theta:1.45) {$\j$};
			}
			\draw[->,thick] (p1) -- (p4);
			\draw[->,thick] (p2) -- (p3);
		\end{tikzpicture}
	\end{center}
	
	Interchanging the two steps replaces
	\[
	12\xrightarrow{\,1\mapsto4\,}24
	\xrightarrow{\,2\mapsto3\,}34
	\]
	by
	\[
	12\xrightarrow{\,2\mapsto3\,}13
	\xrightarrow{\,1\mapsto4\,}34.
	\]
	Thus only the second necklace term changes:
	\[
	J'_2=(J_1\setminus\{2\})\cup\{3\}=13,
	\qquad
	\mathbf J'=(12,13,34,45,15).
	\]
	The new removal and insertion permutations are
	\[
	\rho'=(1\,2),
	\qquad
	\iota'=
	\begin{pmatrix}
		1&2&3&4&5\\
		3&4&5&1&2
	\end{pmatrix},
	\]
	and $\iota'(\rho')^{-1}=\pi$.
	
	We now check the unit-necklace condition.
	The positroid associated with $\pi$ has bases
	\[
	\mathcal M_\pi
	=
	\binom{[5]}{2}\setminus\{23\}.
	\]
	Hence $\Delta_{23}=0$ in $R_\pi$, and the Pl\"ucker relation
	specializes to
	\[
	\Delta_{13}\Delta_{24}
	=
	\Delta_{12}\Delta_{34}
	+
	\Delta_{14}\Delta_{23}
	=
	\Delta_{12}\Delta_{34}.
	\]
	Since $\mathbf J$ is the forward Grassmann necklace,
	\[
	\mathscr F_\pi
	=
	\left\langle
	\Delta_{12},\Delta_{24},\Delta_{34},
	\Delta_{45},\Delta_{15}
	\right\rangle_{\ZZ}.
	\]
	Therefore
	\[
	\Delta_{13}
	=
	\frac{\Delta_{12}\Delta_{34}}{\Delta_{24}}
	\in\mathscr F_\pi.
	\]
	All other entries of $\mathbf J'$ already occur in
	$\mathbf J$. Thus both necklaces are unit necklaces,
	and $(\mathbf J,\mathbf J')$ is an aligned-toggle
	unit-necklace path of length one.
\end{example}

The following lemma realizes an aligned toggle by a short exact
sequence of rank-one lattices.

\begin{lemma}[Short exact sequence associated with an aligned toggle]
	\label{lem:fsb-toggle-sequence}
	Let $\mathbf J=(J_1,\ldots,J_n)$ and
	$\mathbf J'=(J'_1,\ldots,J'_n)$ be Grassmannlike necklaces
	of type $(k,n)$ with a fixed-point-free trip permutation.
	Suppose that $\mathbf J'$ is obtained from $\mathbf J$
	by a nontrivial aligned toggle at position $a$.
	
	After interchanging $\mathbf J$ and $\mathbf J'$ if necessary,
	there exist four distinct boundary labels $u,v,w,x$, occurring
	in this clockwise cyclic order, and a subset
	\[
	S\subseteq[n]\setminus\{u,v,w,x\},
	\qquad |S|=k-2,
	\]
	such that
	\[
	J_{a-1}=Suv,\qquad
	J_a=Svx,\qquad
	J_{a+1}=Swx,\qquad
	J'_a=Suw.
	\]
	Here $Spq$ denotes $S\cup\{p,q\}$, and the remaining necklace
	entries are unchanged.
	
	With this choice of orientation, there is a short exact
	sequence in $\CM(C)$
	\begin{equation}
		\label{eq:fsb-toggle-sequence}
		0\longrightarrow M_{Suw}
		\xrightarrow{\alpha}
		M_{Suv}\oplus M_{Swx}
		\xrightarrow{\beta}
		M_{Svx}
		\longrightarrow0.
	\end{equation}
	Thus the new label determines the kernel, the two neighboring
	labels determine the middle term, and the old label determines
	the quotient.
	
	If all four rank-one $C$-lattices admit $B$-lattice lifts,
	then \eqref{eq:fsb-toggle-sequence} is also a short exact
	sequence in $\CM(B)$, after identifying these lifts with
	their restrictions to $C$.
\end{lemma}
\begin{proof}
	Let $\rho$ and $\iota$ be the removal and insertion permutations
	of $\mathbf J$, respectively, so that
	\[
	J_{b+1}
	=
	\bigl(J_b\setminus\{\rho(b)\}\bigr)\cup\{\iota(b)\}.
	\]
	We first determine the four labels involved in the toggle.
	Write
	\[
	r_1=\rho(a-1),\qquad s_1=\iota(a-1),\qquad
	r_2=\rho(a),\qquad s_2=\iota(a).
	\]
	The fixed-point-free assumption and the toggle conditions imply
	that these four labels are distinct. By the defining recurrence of $\mathbf J$, we have
	\[
	J_a
	=
	\bigl(J_{a-1}\setminus\{r_1\}\bigr)\cup\{s_1\},
	\qquad
	J_{a+1}
	=
	\bigl(J_a\setminus\{r_2\}\bigr)\cup\{s_2\}.
	\] Hence
	\[
	\{r_1,r_2\}\subseteq J_{a-1},
	\qquad
	\{s_1,s_2\}\cap J_{a-1}=\varnothing.
	\]
	Setting
$	S=J_{a-1}\setminus\{r_1,r_2\},$
	we obtain
	\begin{equation*}
		\begin{aligned}
			J_{a-1}&=Sr_1r_2,
			& J_a&=Ss_1r_2,\\
			J_{a+1}&=Ss_1s_2,
			& J'_a&=Sr_1s_2.
		\end{aligned}
	\end{equation*}
	In particular, $|S|=k-2$ and $S$ is disjoint from the four
	endpoints.
	
	Since the toggle is aligned, the clockwise cyclic order of
	these endpoints is either
	\[
	(r_1,r_2,s_2,s_1)
	\qquad\text{or}\qquad
	(r_1,s_1,s_2,r_2).
	\]
	In the first case, take
	$(u,v,w,x)=(r_1,r_2,s_2,s_1).$
	In the second case, interchange $\mathbf J$ and $\mathbf J'$
	and take
	\[
	(u,v,w,x)=(r_2,r_1,s_1,s_2).
	\]
	In both cases, $u,v,w,x$ occur in clockwise cyclic order and
	\[
	(J_{a-1},J_a,J_{a+1},J'_a)
	=(Suv,Svx,Swx,Suw).
	\]
	
	We now construct the short exact sequence in $\CM(C)$.
	For boundary labels $p,q$, let $[p,q)$ denote the clockwise
	cyclic interval starting at $p$ and ending immediately before
	$q$. For each \(i\in[n]\), put
	\[
	\lambda_i=\mathbf1_{[v,w)}(i),
	\qquad
	\nu_i=\mathbf1_{[x,u)}(i),
	\qquad
	f_i=t^{\lambda_i},
	\qquad
	g_i=t^{\nu_i}.
	\]
	Here \(\mathbf1_A(i)=1\) if \(i\in A\), and \(0\) otherwise.
	Thus \(f_i=t\) on \([v,w)\) and \(f_i=1\) elsewhere;
	similarly, \(g_i=t\) on \([x,u)\) and \(g_i=1\) elsewhere.

	At the vertex $i+\tfrac12$, define
	\[
	\alpha_i=
	\begin{pmatrix}f_i\\g_i\end{pmatrix}
	:Z\longrightarrow Z\oplus Z,
	\qquad
	\beta_i=
	\begin{pmatrix}g_i&-f_i\end{pmatrix}
	:Z\oplus Z\longrightarrow Z.
	\]
	We claim that these vertex maps define $C$-module morphisms
	\[
	\alpha:M_{Suw}\longrightarrow M_{Suv}\oplus M_{Swx},
	\qquad
	\beta:M_{Suv}\oplus M_{Swx}\longrightarrow M_{Svx}.
	\]
	
	To check this, recall that a vertexwise multiplication map
	\[
	h_i=t^{d_i}:(M_I)_{i+\frac12}
	\longrightarrow(M_J)_{i+\frac12}
	\]
	commutes with $\mathsf x_i$ precisely when
	\[
	t^{d_i}t^{\mathbf1_I(i)}
	=
	t^{\mathbf1_J(i)}t^{d_{i-1}},
	\]
	or equivalently,
	\begin{equation}
		\label{eq:fsb-toggle-map-condition}
		d_i-d_{i-1}
		=\mathbf1_J(i)-\mathbf1_I(i).
	\end{equation}
	The same equality also gives commutativity with $\mathsf y_i$,
	since
	\[
	d_{i-1}+1-\mathbf1_I(i)
	=
	1-\mathbf1_J(i)+d_i.
	\]
	
	Let $\delta_{ip}$ be the Kronecker delta. The interval
	indicators satisfy
	\[
	\lambda_i-\lambda_{i-1}
	=\delta_{iv}-\delta_{iw},
	\qquad
	\nu_i-\nu_{i-1}
	=\delta_{ix}-\delta_{iu}.
	\]
	On the other hand, we have
	\begin{align*}
		\mathbf1_{Suv}(i)-\mathbf1_{Suw}(i)
		&=
		\mathbf1_{Svx}(i)-\mathbf1_{Swx}(i)
		=\delta_{iv}-\delta_{iw},\\
		\mathbf1_{Swx}(i)-\mathbf1_{Suw}(i)
		&=
		\mathbf1_{Svx}(i)-\mathbf1_{Suv}(i)
		=\delta_{ix}-\delta_{iu}.
	\end{align*}
	Thus \eqref{eq:fsb-toggle-map-condition} holds for all four
	components of $\alpha$ and $\beta$. Consequently, both maps
	are $C$-linear.
	
	It remains to verify exactness. At every vertex,
	\[
	\beta_i\alpha_i=g_if_i-f_ig_i=0.
	\]
	More explicitly, the maps on the four cyclic intervals are
	\[
	\begin{array}{c|c|c}
		i & \alpha_i & \beta_i\\ \hline
		[u,v)
		& \begin{pmatrix}1\\1\end{pmatrix}
		& \begin{pmatrix}1&-1\end{pmatrix}\\[4pt]
		[v,w)
		& \begin{pmatrix}t\\1\end{pmatrix}
		& \begin{pmatrix}1&-t\end{pmatrix}\\[4pt]
		[w,x)
		& \begin{pmatrix}1\\1\end{pmatrix}
		& \begin{pmatrix}1&-1\end{pmatrix}\\[4pt]
		[x,u)
		& \begin{pmatrix}1\\t\end{pmatrix}
		& \begin{pmatrix}t&-1\end{pmatrix}.
	\end{array}
	\]
	In particular, at least one of $f_i,g_i$ equals $1$.
	Therefore $\alpha_i$ is injective, $\beta_i$ is surjective,
	and
	\[
	\ker\beta_i
	=
	Z\begin{pmatrix}f_i\\g_i\end{pmatrix}
	=
	\operatorname{im}\alpha_i.
	\]
	Exactness can be checked at each vertex, so these maps give
	the required short exact sequence
	\[
	0\longrightarrow M_{Suw}
	\xrightarrow{\alpha}
	M_{Suv}\oplus M_{Swx}
	\xrightarrow{\beta}
	M_{Svx}
	\longrightarrow0.
	\]
	All its terms are $C$-lattices, so it is a conflation in
	$\CM(C)$.
	
	Finally, suppose that all four modules admit $B$-lattice
	lifts. By the full faithfulness of restriction
	\[
	\CM(B)\longrightarrow\CM(C),
	\]
	proved in \cite[Proposition~5.6]{Pressland}, the maps
	$\alpha$ and $\beta$ are $B$-linear after identifying the
	lifts with their restrictions. Restriction along $C\to B$
	does not change the underlying vector spaces or linear maps.
	Hence the same sequence is exact over $B$, and all its terms
	belong to $\CM(B)$.
	
	Along a unit-necklace path, each of the four labels occurs
	in one of the two necklaces. Its Pl\"ucker coordinate belongs
	to $\mathscr F_\pi$ and is therefore nonzero in $R_\pi$.
	By \cite[Proposition~5.8]{Pressland}, all four rank-one modules
	admit $B$-lattice lifts. Thus the preceding argument applies
	to every step of the path, with the orientation chosen above.
\end{proof}

\begin{example}
	\label{ex:necklace-toggle}
	Consider the permutation
	\[
	\pi=
	\begin{pmatrix}
		1&2&3&4&5&6\\
		4&6&5&2&1&3
	\end{pmatrix}.
	\]
	Its ordinary target Grassmann necklace is
	\[
	\mathbf J=(123,234,346,456,256,126),
	\]
	with removal permutation $\rho=\mathrm{id}$ and insertion
	permutation $\iota=\pi$; see
	\cite[Examples~4.6 and~4.11]{FSB}.
	We examine the toggle at position $a=3$.
	
	The two recurrences adjacent to $J_3$ are
	\[
	J_3=(J_2\setminus\{2\})\cup\{6\}=346,
	\qquad
	J_4=(J_3\setminus\{3\})\cup\{5\}=456.
	\]
	The toggle is allowed because $2\ne5$ and $3\ne6$.
	Moreover, the chords $2\to6$ and $3\to5$ are disjoint,
	and their endpoints occur in the clockwise order $2,3,5,6$.
	Thus the toggle is aligned. Interchanging the order of the
	two replacements gives
	\[
	J'_3=(J_2\setminus\{3\})\cup\{5\}=245,
	\qquad
	J_4=(J'_3\setminus\{2\})\cup\{6\}=456.
	\]
	The resulting necklace is therefore
	\[
	\mathbf J'=(123,234,245,456,256,126).
	\]
	
	\begin{center}
		\begin{tikzpicture}[
			>=Stealth,
			every node/.style={font=\small},
			lab/.style={fill=white,inner sep=2pt},
			entry/.style={
				draw,
				rounded corners=2pt,
				fill=white,
				minimum width=11mm,
				minimum height=7mm
			}
			]
			\draw[gray!65] (0,0) circle (1.35);
			\foreach \j/\theta in
			{1/90,2/30,3/-30,4/-90,5/-150,6/150}
			{
				\coordinate (p\j) at (\theta:1.35);
				\fill (p\j) circle (1.5pt);
				\node at (\theta:1.64) {$\j$};
			}
			\draw[->,thick,shorten <=2pt,shorten >=2pt]
			(p2) --
			node[lab,above=3pt] {$2\mapsto6$}
			(p6);
			\draw[->,thick,dashed,shorten <=2pt,shorten >=2pt]
			(p3) --
			node[lab,below=3pt] {$3\mapsto5$}
			(p5);
			\node[align=center] at (0,-2.13)
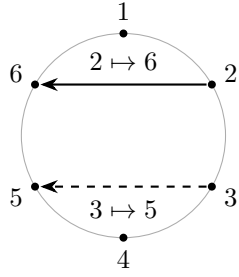
			{The aligned chords};
			
			\begin{scope}[xshift=3.45cm]
				\node[entry] (left)  at (0,0)    {$234$};
				\node[entry] (old)   at (2.3,1)  {$346$};
				\node[entry] (right) at (4.6,0)  {$456$};
				\node[entry] (new)   at (2.3,-1) {$245$};
				
				\draw[->,thick]
				(left) --
				node[lab,above,sloped] {$2\mapsto6$}
				(old);
				\draw[->,thick,dashed]
				(old) --
				node[lab,above,sloped] {$3\mapsto5$}
				(right);
				\draw[->,thick,dashed]
				(left) --
				node[lab,below,sloped] {$3\mapsto5$}
				(new);
				\draw[->,thick]
				(new) --
				node[lab,below,sloped] {$2\mapsto6$}
				(right);
				
				\draw[->,gray!75]
				(old) --
				node[right,font=\scriptsize] {toggle}
				(new);
				\node[align=center] at (2.3,-2.13)
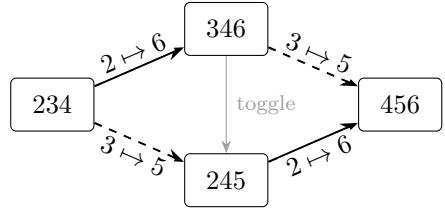
				{The two orders of replacement};
			\end{scope}
		\end{tikzpicture}
	\end{center}
	
	In the notation of the toggle sequence, we have
	\[
	S=\{4\},
	\qquad
	(u,v,w,x)=(2,3,5,6),
	\]
	and hence
	$(Suv,Svx,Swx,Suw)=(234,346,456,245).$

	The corresponding short exact sequence is
	\begin{equation}
		\label{eq:ex-necklace-toggle-sequence}
		0\longrightarrow M_{245}
		\xrightarrow{\alpha}
		M_{234}\oplus M_{456}
		\xrightarrow{\beta}
		M_{346}
		\longrightarrow0.
	\end{equation}
	
	We make its maps explicit. At the vertex $i+\tfrac12$, put
	\[
	\lambda_i=\mathbf1_{[3,5)}(i),
	\qquad
	\nu_i=\mathbf1_{[6,2)}(i),
	\]
	where the intervals are taken cyclically. Thus
	\[
	(\lambda_1,\ldots,\lambda_6)=(0,0,1,1,0,0),
	\qquad
	(\nu_1,\ldots,\nu_6)=(1,0,0,0,0,1).
	\]
	Define
	\[
	\alpha_i=
	\begin{pmatrix}
		t^{\lambda_i}\\
		t^{\nu_i}
	\end{pmatrix},
	\qquad
	\beta_i=
	\begin{pmatrix}
		t^{\nu_i}&-t^{\lambda_i}
	\end{pmatrix}.
	\]
	Equivalently, these maps are
	\[
	\begin{array}{c|c|c}
		i&\alpha_i&\beta_i\\ \hline
		1,6
		&\begin{pmatrix}1\\t\end{pmatrix}
		&\begin{pmatrix}t&-1\end{pmatrix}\\[5pt]
		2,5
		&\begin{pmatrix}1\\1\end{pmatrix}
		&\begin{pmatrix}1&-1\end{pmatrix}\\[5pt]
		3,4
		&\begin{pmatrix}t\\1\end{pmatrix}
		&\begin{pmatrix}1&-t\end{pmatrix}.
	\end{array}
	\]
	To verify $C$-linearity, observe that, with indices read
	modulo $6$,
	\begin{align*}
		\lambda_i-\lambda_{i-1}
		&=
		\mathbf1_{234}(i)-\mathbf1_{245}(i)
		=
		\mathbf1_{346}(i)-\mathbf1_{456}(i),\\
		\nu_i-\nu_{i-1}
		&=
		\mathbf1_{456}(i)-\mathbf1_{245}(i)
		=
		\mathbf1_{346}(i)-\mathbf1_{234}(i).
	\end{align*}
	These identities show that all four components commute
	with the arrows $\mathsf x_i$ and $\mathsf y_i$.
	At every vertex, we have
	\[
	\beta_i\alpha_i=0,
	\qquad
	\ker\beta_i
	=
	Z\begin{pmatrix}
		t^{\lambda_i}\\
		t^{\nu_i}
	\end{pmatrix}
	=
	\operatorname{im}\alpha_i.
	\]
	Since at least one of $t^{\lambda_i},t^{\nu_i}$ is $1$,
	the map $\alpha_i$ is injective and $\beta_i$ is surjective.
	This proves exactness of
	\eqref{eq:ex-necklace-toggle-sequence} in $\CM(C)$.
	
	We now compare this sequence with the Pl\"ucker coordinates.
	The positroid of $\pi$ is
	\[
	\mathcal M_\pi
	=
	\binom{[6]}{3}\setminus\{345,156\};
	\]
	see \cite[Example~4.11]{FSB}.
	In particular, all four labels in
	\eqref{eq:ex-necklace-toggle-sequence} belong to
	$\mathcal M_\pi$.
	The rank-one restriction criterion therefore provides their
	$B$-lattice lifts, and full faithfulness of restriction makes
	the displayed maps $B$-linear. Thus the sequence is also
	exact in $\CM(B)$.
	
	The three-term Pl\"ucker relation is
	\[
	\Delta_{245}\Delta_{346}
	=
	\Delta_{234}\Delta_{456}
	+
	\Delta_{246}\Delta_{345}.
	\]
	Since $345\notin\mathcal M_\pi$, we have $\Delta_{345}=0$
	in $R_\pi$. Consequently,
	\begin{equation}
		\label{eq:ex-necklace-toggle-monomial}
		\Delta_{245}
		=
		\frac{\Delta_{234}\Delta_{456}}{\Delta_{346}}
		\in\mathscr F_\pi.
	\end{equation}
	Here $\Delta_{234},\Delta_{346},\Delta_{456}$ are frozen
	coordinates of the ordinary target seed. Thus $\mathbf J'$
	is again a unit necklace. Formula
	\eqref{eq:ex-necklace-toggle-monomial} describes the coordinate
	change, while the explicit maps above realize the replacement
	of $M_{346}$ by $M_{245}$ through their common neighboring
	modules $M_{234}$ and $M_{456}$.
\end{example}

\begin{theorem}[Boundary tilting]
	\label{prop:fsb-boundary-tilting}
	The boundary module $P_{\cI}$ is a $1$-tilting $B$-module. More explicitly,
	\begin{equation}
		\label{eq:fsb-boundary-tilting}
		\operatorname{pd}_B P_{\cI}\leq1,\qquad
		\Ext_B^a(P_{\cI},P_{\cI})=0\ (a>0),\qquad \thick(P_{\cI})=\per B.
	\end{equation}
\end{theorem}

\begin{proof}
	We first show that
	\(\thick(B^\vee)=\per B\).
	Recall that \(\CM(B)\) has enough projectives and injectives,
	with projective objects \(\add B\) and injective objects
	\(\add B^\vee\), respectively.
	
	Put \(d=\operatorname{id}_B B<\infty\) and \(K_0=B\).
	Choose conflations
	\[
	0\longrightarrow K_r\longrightarrow J^r
	\longrightarrow K_{r+1}\longrightarrow0,
	\qquad
	J^r\in\add B^\vee,
	\qquad 0\leq r<d.
	\]
	For every lattice \(X\) and every \(m>0\), adjunction gives
	\[
	\Ext_B^m(X,B^\vee)
	\cong
	\Ext_Z^m(X,Z)
	=0.
	\]
	Indeed, a projective \(B\)-resolution of \(X\) is also
	a projective \(Z\)-resolution, and \(X\) is free over \(Z\).
	Consequently,
	\(\Ext_B^m(X,J^r)=0\), and the long exact Ext sequences yield
	\[
	\Ext_B^m(X,K_{r+1})
	\cong
	\Ext_B^{m+1}(X,K_r).
	\]
	Iteration gives
	\[
	\Ext_B^1(X,K_d)
	\cong
	\Ext_B^{d+1}(X,B)
	=0.
	\]
	Thus \(K_d\) is injective in \(\CM(B)\), so
	\(K_d\in\add B^\vee\).
	Setting \(J^d=K_d\) and splicing the conflations, we obtain
	a finite exact sequence
	\[
	0\longrightarrow B\longrightarrow J^0
	\longrightarrow\cdots\longrightarrow J^d
	\longrightarrow0,
	\qquad J^r\in\add B^\vee.
	\]
	When \(d=0\), this means that \(B\) itself is injective
	in \(\CM(B)\).
	In either case,
	\[
	B\in\thick(B^\vee).
	\]
	
	Apply the same argument to \(B^{\mathrm{op}}\).
	Since
	\[
	d'=\operatorname{id}_{B^{\mathrm{op}}}B<\infty,
	\]
	the right \(B\)-module \(B\) admits a finite coresolution
	by objects of \(\add(B^\vee_B)\).
	The duality \((-)^\vee=\Hom_Z(-,Z)\) is exact on lattices
	and satisfies
	\[
	(B^\vee)^\vee\simeq B.
	\]
	Dualizing this coresolution therefore gives a finite
	projective resolution of the left \(B\)-module \(B^\vee\):
	\[
	0\longrightarrow Q_{d'}\longrightarrow\cdots
	\longrightarrow Q_0\longrightarrow B^\vee
	\longrightarrow0,
	\qquad Q_r\in\add B.
	\]
	Hence \(B^\vee\in\per B=\thick(B)\).
	Combining the two inclusions proves
	\begin{equation}
		\label{eq:fsb-injective-generator}
		\thick(B^\vee)=\thick(B)=\per B.
	\end{equation}
	
	We now transfer this generation property to \(P_{\cI}\).
	By \cite[Lemma~4.13 and Theorem~4.14]{FSB}, there is an
	aligned-toggle path of unit necklaces
	\[
	\mathbf I^{(0)},\mathbf I^{(1)},\ldots,
	\mathbf I^{(N)}=\mathbf I
	\]
	starting at the ordinary forward necklace.
	For each \(a\), let \(\cI^{(a)}\) be the set of distinct
	terms of \(\mathbf I^{(a)}\), and put
	\[
	P^{(a)}
	=
	\bigoplus_{I\in\cI^{(a)}}M_I.
	\]
	All these rank-one \(B\)-lattices are defined because
	every necklace along the path is a unit necklace.
	For the initial necklace,
	\[
	\add P^{(0)}=\add B^\vee,
	\]
	so \eqref{eq:fsb-injective-generator} gives
	\[
	\thick(P^{(0)})=\per B.
	\]
	
	Consider a nontrivial toggle from
	\(\mathbf I^{(a)}\) to \(\mathbf I^{(a+1)}\), replacing
	\(J_{\mathrm{old}}\) by \(J_{\mathrm{new}}\).
	By Lemma~\ref{lem:fsb-toggle-sequence}, the two modules
	\(M_{J_{\mathrm{old}}}\) and \(M_{J_{\mathrm{new}}}\)
	are the end terms of a short exact sequence whose middle
	term \(E_a\) is the direct sum of the two neighbouring
	boundary modules.
	These neighbouring labels are unchanged, so
	\[
	E_a\in\add P^{(a)}\cap\add P^{(a+1)}.
	\]
	The associated triangle then gives
	\[
	M_{J_{\mathrm{new}}}\in\thick(P^{(a)}),
	\qquad
	M_{J_{\mathrm{old}}}\in\thick(P^{(a+1)}).
	\]
	Since every other necklace term is unchanged, we conclude that
	\[
	\thick(P^{(a+1)})=\thick(P^{(a)}).
	\]
	Trivial toggles leave the boundary sum unchanged.
	Induction along the path therefore yields
	\[
	\thick(P_{\cI})
	=
	\thick(P^{(N)})
	=
	\thick(P^{(0)})
	=
	\per B.
	\]
	
	In particular, \(P_{\cI}\) is perfect and has finite
	projective dimension.
	Choose a projective presentation
	\[
	0\longrightarrow\Omega_BP_{\cI}
	\longrightarrow Q\longrightarrow P_{\cI}
	\longrightarrow0,
	\qquad Q\in\proj B.
	\]
	Since \(P_{\cI}\) is a lattice, the boundary-order
	syzygy property gives
	\[
	\Omega_BP_{\cI}\in\GP B.
	\]
	This syzygy also has finite projective dimension.
	A Gorenstein-projective module of finite projective
	dimension is projective: its finite projective resolution
	splits successively, starting with the last short exact
	sequence, because
	\(\Ext_B^1(G,Q')=0\) for every Gorenstein-projective
	module \(G\) and every projective module \(Q'\).
	Thus
	\[
	\Omega_BP_{\cI}\in\proj B,
	\qquad
	\operatorname{pd}_BP_{\cI}\leq1.
	\]
	
	Finally, we prove self-extension vanishing.
	For \(X,Y\in\CM(B)\), restriction induces an injection
	\begin{equation}
		\label{eq:fsb-ext-injection-again}
		\Ext_B^1(X,Y)
		\lhook\joinrel\longrightarrow
		\Ext_C^1(X,Y).
	\end{equation}
	Indeed, an extension of two lattices has a lattice
	as its middle term. If its restriction to \(C\) splits,
	the resulting \(C\)-linear splitting is \(B\)-linear
	by full faithfulness of restriction.
	
	Admissibility and
	Theorem~\ref{thm:fsb-cluster-criterion} imply that
	the final boundary collection \(\cI\) is weakly separated.
	The rank-one extension criterion \cite{JKS} gives
	\[
	\Ext_C^1(M_I,M_J)=0
	\qquad(I,J\in\cI).
	\]
	Using \eqref{eq:fsb-ext-injection-again} and taking
	finite direct sums, we obtain
	\[
	\Ext_B^1(P_{\cI},P_{\cI})=0.
	\]
	The projective-dimension bound then gives
	\[
	\Ext_B^a(P_{\cI},P_{\cI})=0
	\qquad(a\geq2).
	\]
	Therefore \(P_{\cI}\) satisfies all three conditions
	in \eqref{eq:fsb-boundary-tilting}, and is a
	\(1\)-tilting \(B\)-module.
\end{proof}

\subsection{The tilted order and its face modules}

\begin{definition}[Tilted boundary order]
	\label{def:necklace-tilted-order}
	The \emph{tilted boundary order} associated with $\cI$ is
	\[
	B_{\cI}=\End_B(P_{\cI})^{\mathrm{op}}.
	\]
	We regard $P_{\cI}$ as a $(B,B_{\cI})$-bimodule by
	$p\cdot a^{\mathrm{op}}=a(p)$ for $a\in\End_B(P_{\cI})$.
	For a left $B$-module $M$, the left $B_{\cI}$-action on
	$\Hom_B(P_{\cI},M)$ is $a^{\mathrm{op}}\cdot h=h\circ a$.
	For $J\in\cS$, put
	\[
	N_J=\Hom_B(P_{\cI},M_J),\qquad
	T_{\cI}=\bigoplus_{J\in\cS}N_J.
	\]

\end{definition}

\begin{lemma}
	\label{lem:fsb-derived-order}
	The algebra $B_{\cI}$ is a two-sided Noetherian
	Iwanaga--Gorenstein $Z$-order. There are inverse equivalences
	\begin{equation}
		\label{eq:fsb-derived-equivalence}
		\begin{aligned}
			F_{\cI}=\mathbf R\Hom_B(P_{\cI},-)&:
			\Db(\fgmod B)\xrightarrow{\sim}\Db(\fgmod B_{\cI}),\\
			L_{\cI}=P_{\cI}\otimes^{\mathbf L}_{B_{\cI}}-&:
			\Db(\fgmod B_{\cI})\xrightarrow{\sim}\Db(\fgmod B).
		\end{aligned}
	\end{equation}
	They restrict to equivalences of perfect categories and induce
	inverse equivalences of singularity categories.
\end{lemma}

\begin{proof}
	Put
	\[
	P=P_{\cI},
	\qquad
	\Lambda=B_{\cI}=\End_B(P)^{\mathrm{op}}.
	\]
	We regard $P$ as a $(B,\Lambda)$-bimodule with the actions
	specified above.
	
	We first verify that $\Lambda$ is a two-sided Noetherian
	$Z$-order. Since $P$ is a $B$-lattice, it is finite free over
	$Z$, and
	$\End_B(P)\subseteq\End_Z(P).$
	
	The ring $Z$ is a discrete valuation ring, so every submodule
	of a finite free $Z$-module is again finite free. Consequently,
	$\Lambda$ is finite free over $Z$, with central $Z$-action.
	
	Write
	\[
	B_K=K\otimes_Z B,
	\qquad
	P_K=K\otimes_Z P.
	\]
	Since $P$ is finitely presented over $B$, localization gives
	\[
	K\otimes_Z\Lambda
	\simeq
	\End_{B_K}(P_K)^{\mathrm{op}}.
	\]
	The generic fibre $B_K$ of the reference boundary order is
	semisimple. Hence $P_K$ is a finite-dimensional semisimple
	$B_K$-module, and its endomorphism algebra is semisimple.
	Thus $\Lambda$ is a $Z$-order. Moreover, every left or right
	ideal of $\Lambda$ is a $Z$-submodule of the finite
	$Z$-module $\Lambda$. The ascending chain condition over
	$Z$ therefore implies that $\Lambda$ is Noetherian on
	both sides.
	
	By Proposition~\ref{prop:fsb-boundary-tilting}, we have
	\[
	\operatorname{pd}_B P\leq1,
	\qquad
	\Ext_B^j(P,P)=0\quad(j>0),
	\qquad
	\thick(P)=\per B.
	\]
	Since $P$ is concentrated in degree zero, we also have
	\[
	\Hom_{\Db(\fgmod B)}(P,P[j])=0
	\qquad(j<0).
	\]
	Thus $P$ is a tilting object with endomorphism algebra
	$\Lambda^{\mathrm{op}}$.
	
	Set
	$Q=\mathbf R\Hom_B(P,B),$
	viewed as a complex of $(\Lambda,B)$-bimodules. By
	Derived Morita theory \cite{Rickard}, we have isomorphisms
	\begin{equation}
		\label{eq:fsb-tilting-bimodule-identities}
		P\otimes_\Lambda^{\mathbf L}Q\simeq B,
		\qquad
		Q\otimes_B^{\mathbf L}P\simeq\Lambda
	\end{equation}
	in the corresponding derived categories of bimodules.
	Moreover, $P$ and $Q$ are perfect over each of their
	two acting rings separately. Since $P$ is perfect as a
	left $B$-module, there is a natural isomorphism
	\[
	\mathbf R\Hom_B(P,-)
	\simeq Q\otimes_B^{\mathbf L}-.
	\]
	Therefore the functors
	\[
	F=Q\otimes_B^{\mathbf L}-,
	\qquad
	L=P\otimes_\Lambda^{\mathbf L}-
	\]
	are mutually inverse. Indeed,
	\begin{align*}
		LF(X)
		&\simeq
		\bigl(P\otimes_\Lambda^{\mathbf L}Q\bigr)
		\otimes_B^{\mathbf L}X
		\simeq X,\\
		FL(Y)
		&\simeq
		\bigl(Q\otimes_B^{\mathbf L}P\bigr)
		\otimes_\Lambda^{\mathbf L}Y
		\simeq Y.
	\end{align*}
	Their finite Tor-amplitudes and finite generation properties
	imply that they preserve bounded complexes with finitely
	generated cohomology. This proves
	\eqref{eq:fsb-derived-equivalence}.
	
	In particular, we have
	$F(P)\simeq\Lambda,
	$ and $
	L(\Lambda)\simeq P.$
	It follows that
	\[
	F(\per B)
	=
	F(\thick(P))
	=
	\thick(\Lambda)
	=
	\per\Lambda.
	\]
	Hence the equivalences restrict to the perfect categories.
	
	We next prove that $\Lambda$ is Iwanaga--Gorenstein.
	We use the finite self-injective dimensions of $B$ on both
	sides and the finite amplitudes of the tilting functors.
	
	First consider left modules. Since $P$ is perfect over $B$
	and $\operatorname{id}_B B<\infty$, it has a bounded
	injective resolution
	\[
	P\xrightarrow{\sim}I^\bullet,
	\qquad
	I^j=0\quad(j>N)
	\]
	for some integer $N$.
	Since $P$ is also perfect as a right $\Lambda$-module,
	there is an integer $a$, independent of
	$X\in\fgmod\Lambda$, such that
	\[
	H^j(LX)=0\qquad(j<a).
	\]
	Choose a representative $A_X^\bullet$ of $LX$ with
	$A_X^j=0$ for $j<a$.
	
	For $q>N-a$, the shifted injective complex satisfies
	\[
	(I^\bullet[q])^j=I^{j+q}=0
	\qquad(j\geq a).
	\]
	Thus every chain map
	$A_X^\bullet\to I^\bullet[q]$ is zero.
	A bounded complex of injectives computes morphisms into it
	in the derived category, so
	\[
	\Hom_{\Db(\fgmod B)}(LX,P[q])=0
	\qquad(q>N-a).
	\]
	Using $L(\Lambda)\simeq P$, we obtain
	\[
	\begin{aligned}
		\Ext_\Lambda^q(X,\Lambda)
		&\simeq
		\Hom_{\Db(\fgmod\Lambda)}(X,\Lambda[q])\\
		&\simeq
		\Hom_{\Db(\fgmod B)}(LX,P[q])\\
		&=0
		\qquad(q>N-a).
	\end{aligned}
	\]
	The bound is uniform in $X$. In particular, it applies to
	$\Lambda/J$ for every left ideal $J\subseteq\Lambda$.
	By the injective-dimension form of Baer's criterion, one has
	\[
	\operatorname{id}_\Lambda\Lambda<\infty.
	\]
	
	For right modules, the same bimodule identities give inverse
	equivalences
	\[
	-\otimes_\Lambda^{\mathbf L}Q:
	\Db(\fgmod\Lambda^{\mathrm{op}})
	\xrightarrow{\sim}
	\Db(\fgmod B^{\mathrm{op}}),
	-\otimes_B^{\mathbf L}P.
	\]
	The first functor sends the regular right $\Lambda$-module
	to $Q$. Since $Q$ is perfect as a right $B$-complex and
	$\operatorname{id}_{B^{\mathrm{op}}}B<\infty$, it admits
	a bounded injective representative
	\[
	Q\simeq J^\bullet,
	\qquad
	J^j=0\quad(j>N')
	\]
	for some integer $N'$.
	
	Since $Q$ is perfect as a left $\Lambda$-complex, there is
	an integer $a'$, independent of the finitely generated
	right $\Lambda$-module $Y$, such that
	\[
	H^j\bigl(Y\otimes_\Lambda^{\mathbf L}Q\bigr)=0
	\qquad(j<a').
	\]
	Repeating the preceding degree argument gives
	\[
	\begin{aligned}
		\Ext_{\Lambda^{\mathrm{op}}}^q(Y,\Lambda)
		&\simeq
		\Hom_{\Db(\fgmod B^{\mathrm{op}})}
		\bigl(Y\otimes_\Lambda^{\mathbf L}Q,Q[q]\bigr)\\
		&=0
		\qquad(q>N'-a').
	\end{aligned}
	\]
	By Baer's criterion \cite[\S3A]{Lam}, we have
	\[
	\operatorname{id}_{\Lambda^{\mathrm{op}}}\Lambda
	<\infty.
	\]
	Together with two-sided Noetherianity, these two bounds
	prove that $\Lambda$ is Iwanaga--Gorenstein.
	
	Finally, since $F$ and $L$ identify the perfect subcategories,
	they descend to mutually inverse equivalences of Verdier
	quotients:
	\[
	\begin{aligned}
		\overline F:\Dsg(B)
		&\xrightarrow{\sim}\Dsg(\Lambda),\\
		\overline L:\Dsg(\Lambda)
		&\xrightarrow{\sim}\Dsg(B).
	\end{aligned}
	\]
	Substituting $\Lambda=B_{\cI}$ completes the proof.
\end{proof}

\begin{definition}[$P_{\cI}$-orthogonal lattices]
	\label{def:necklace-orthogonal-lattices}
	Let
	\[
	\mathcal L_{\cI}=\{M\in\CM(B):
	\Ext_B^1(P_{\cI},M)=0=\Ext_B^1(M,P_{\cI})\}.
	\]
	This is an extension-closed full subcategory, with its inherited
	lattice exact structure. It is not being identified with all of
	$\CM(B)$ or with all of $\GP B_{\cI}$.
\end{definition}

\begin{lemma}
	\label{lem:necklace-orthogonal-images}
	For $M\in\mathcal L_{\cI}$, we have
	\[
	F_{\cI}(M)\simeq\Hom_B(P_{\cI},M)\in\GP B_{\cI}.
	\]
	The functor $\Hom_B(P_{\cI},-):\mathcal L_{\cI}\to\GP B_{\cI}$ is exact and
	fully faithful. In particular, for every $J\in\cS$,
	\begin{equation}
		\label{eq:fsb-face-two-sided-ext}
		\Ext_B^1(P_{\cI},M_J)=0=\Ext_B^1(M_J,P_{\cI}).
	\end{equation}
\end{lemma}

\begin{proof}
	If $G\in\GP B$, $L$ has finite projective dimension $d$, and $a>0$,
	dimension shifting in the second variable gives
	\begin{equation}
		\label{eq:fsb-gp-perfect-ext}
		\Ext_B^a(G,L)
		\simeq\Ext_B^{a+d}(G,\Omega_B^dL)=0.
	\end{equation}
	Taking $G=\Omega_BM$ and $L=P_{\cI}$ shows that
	$\Ext_B^a(M,P_{\cI})=0$ for $a\geq2$. Its degree-one group vanishes by
	definition of $\mathcal L_{\cI}$. Similarly,
	$\operatorname{pd}_B P_{\cI}\leq1$ and $\Ext_B^1(P_{\cI},M)=0$ imply
	$\Ext_B^{>0}(P_{\cI},M)=0$. Thus $F_{\cI}(M)$ is concentrated in degree
	zero, and
	\[
	\Ext_{B_{\cI}}^a(\Hom_B(P_{\cI},M),B_{\cI})
	\simeq\Ext_B^a(M,P_{\cI})=0\qquad(a>0).
	\]
	The Iwanaga--Gorenstein criterion gives Gorenstein-projectivity.
	Full faithfulness follows by taking degree-zero morphisms under the
	derived equivalence. Exactness follows by applying $\Hom_B(P_{\cI},-)$
	to a conflation whose left endpoint belongs to $\mathcal L_{\cI}$.
	Finally, weak separation of $\cS$ and
	\eqref{eq:fsb-ext-injection-again} give
	\eqref{eq:fsb-face-two-sided-ext}.
\end{proof}

\begin{theorem}
	\label{prop:fsb-frobenius-realization}
	The exact category
$	\cE_{\cI}=\GP B_{\cI}$
	is Krull--Schmidt and Frobenius, with projective-injective subcategory
	\[
	\cP_{\cI}=\proj B_{\cI}=\add\!\left(\bigoplus_{I\in\cI}N_I\right).
	\]
	Its stable category is Hom-finite and $2$-Calabi--Yau. The object
	$T_{\cI}$ is basic cluster-tilting, and $N_J$ is projective if and
	only if $J\in\cI$. Moreover, $F_{\cI}M_J\simeq N_J$ for every
	$J\in\cS$.
\end{theorem}

\begin{proof}
	The Gorenstein-projective assertions follow from the preceding lemma.
	Since $B_{\cI}$ is finite over the complete local ring $Z$, its
	finitely generated modules have semiperfect endomorphism rings.
	The category $\GP B_{\cI}$ is closed under summands, hence is
	Krull--Schmidt. We have equivalences
	\[
	\underline{\cE}_{\cI}\simeq\Dsg(B_{\cI})
	\xrightarrow[\sim]{\overline L_{\cI}}\Dsg(B),
	\qquad \underline N_J\longmapsto qM_J.
	\]
	The last category is a Hom-finite $2$-Calabi--Yau category.
	
	With the action defined above, $a^{\mathrm{op}}\mapsto a$ identifies
	the left regular $B_{\cI}$-module with $\Hom_B(P_{\cI},P_{\cI})$.
	Thus
	\[
	{}_{B_{\cI}}B_{\cI}
	\simeq\Hom_B(P_{\cI},P_{\cI})
	\simeq\bigoplus_{I\in\cI}N_I.
	\]
	Distinct labels give
	nonisomorphic rank-one modules and hence nonisomorphic $N_J$.
	If $N_J$ is projective, applying $L_{\cI}$ gives
	$M_J\in\add P_{\cI}$, which forces $J\in\cI$.
	
	By Proposition~\ref{prop:fsb-common-face-cluster}, the nonprojective
	part of $T_{\cI}$ is cluster-tilting in the stable category.
	Since $T_{\cI}$ contains every indecomposable projective-injective,
	we have
	\[
	\add T_{\cI}
	=\{X\in\cE_{\cI}:\Ext^1(T_{\cI},X)=0\}
	=\{X\in\cE_{\cI}:\Ext^1(X,T_{\cI})=0\}.
	\]
	For clarity, if a stable object lies in $\add\underline T_{\cI}$,
	Krull--Schmidt decomposition identifies its nonprojective summands
	with summands of $T_{\cI}$; its remaining summands are projective.
	Finally, Hom modules are finite over $Z$. Choosing finite sets of
	$Z$-generators gives right and left $\add T_{\cI}$-approximations.
	Thus $\add T_{\cI}$ is functorially finite, as required.
\end{proof}

\subsection{The relabeled endomorphism algebra}

\begin{definition}[relabeled endomorphism algebra]
	\label{def:necklace-face-algebra}
	The \emph{relabeled endomorphism algebra} associated with
	the face-relabeled cluster-tilting object $T_{\cI}$ is
	\[
	A_{\cI}
	=
	\End_{B_{\cI}}(T_{\cI})^{\mathrm{op}}.
	\]
\end{definition}

More generally, let $V=\bigoplus_j V_j$ be a basic
cluster-tilting object in a module Frobenius category.
We write
\[
A_V=\End(V)^{\mathrm{op}},
\qquad
H_V=\Hom(V,-),
\qquad
Q_j=H_V(V_j).
\]
The functor $H_V$ takes values in left $A_V$-modules,
and $Q_j$ is the indecomposable projective corresponding
to $V_j$.

We write $K_0(\proj A_V)$ for the split Grothendieck group
of finitely generated projective left $A_V$-modules.
If an $A_V$-module $M$ admits a finite resolution
\[
0\longrightarrow Q_d\longrightarrow\cdots
\longrightarrow Q_0\longrightarrow M\longrightarrow0
\]
by finitely generated projectives, its class is defined by
\[
[M]=\sum_{i=0}^{d}(-1)^i[Q_i]
\in K_0(\proj A_V).
\]
This class is independent of the chosen resolution.
Whenever a simple module admits such a resolution,
we interpret its class in $K_0(\proj A_V)$ in this sense.

\begin{lemma}[Finite global dimension]
	\label{lem:fsb-face-finite-gldim}
	Let $\Lambda$ be an Iwanaga--Gorenstein $Z$-order and let
	$V\in\GP\Lambda$ be cluster-tilting. Put
	$A=\End_\Lambda(V)^{\mathrm{op}}$ and $H=\Hom_\Lambda(V,-)$.
	If every finitely generated $\Lambda$-module has
	Gorenstein-projective dimension at most $d$, then
	\[
	\operatorname{l.gl.dim}A\leq\max\{3,d+2\}.
	\]
	In particular, $A_{\cI}$ has finite left and right global dimension.
\end{lemma}

\begin{proof}
	The regular module $\Lambda$ belongs to $\add V$. If
	$K\in\fgmod\Lambda$, choose $Z$-generators
	$f_1,\ldots,f_N$ of $\Hom_\Lambda(V,K)$. The map
	$(f_1,\ldots,f_N):V^N\to K$ is a right $\add V$-approximation and
	is surjective: a surjection from an object of $\add\Lambda$ factors
	through it. Iterating yields
	\begin{equation}
		\label{eq:fsb-approximation-tower}
		0\longrightarrow K^{a+1}\longrightarrow V^a
		\xrightarrow{p_a}K^a\longrightarrow0,
		\qquad K^0=K,\quad V^a\in\add V.
	\end{equation}
	The approximation property makes $H(p_a)$ surjective. Rigidity of
	$V$ and the long exact sequence give
	$\Ext_\Lambda^1(V,K^{a+1})=0$. Dimension shifting gives
	$\operatorname{Gpd}_\Lambda K^a\leq\max\{d-a,0\}$.
	For $s=\max\{1,d\}$, the module $K^s$ is Gorenstein-projective
	and orthogonal to $V$, hence lies in $\add V$. Applying $H$ to
	the tower gives $\operatorname{pd}_AH(K)\leq s$.
	
	Given $X\in\fgmod A$, lift a projective presentation through the
	equivalence $H:\add V\simeq\proj A$:
	\[
	H(V_1)\xrightarrow{H(u)}H(V_0)\twoheadrightarrow X.
	\]
	For $K=\ker u$, left exactness gives
	\[
	0\longrightarrow H(K)\longrightarrow H(V_1)\longrightarrow
	H(V_0)\longrightarrow X\longrightarrow0.
	\]
	Thus $\operatorname{pd}_AX\leq s+2$. The algebra $A$ is Noetherian,
	since it is finite over $Z$. In particular this uniform bound applies
	to every cyclic module $A/J$, so it bounds the left global dimension.
	
	For the right-hand statement, use the exact duality
	\[
	(-)^*=\Hom_\Lambda(-,\Lambda):
	(\GP\Lambda)^{\mathrm{op}}\xrightarrow{\sim}\GP\Lambda^{\mathrm{op}}.
	\]
	It sends $V$ to a cluster-tilting object $V^*$ and gives
	\[
	\End_{\Lambda^{\mathrm{op}}}(V^*)^{\mathrm{op}}
	\simeq A^{\mathrm{op}}.
	\]
	Apply the left-hand argument over $\Lambda^{\mathrm{op}}$.
	The required finite bounds on Gorenstein-projective dimensions exist
	on both sides because $\Lambda$ is Iwanaga--Gorenstein.
\end{proof}

\subsection{The relabeled ice quiver}
\label{subsec:fsb-relabeled-ice-quiver}

\begin{definition}[Gabriel quiver and frozen vertices]
	\label{def:necklace-ice-quiver}
	Let $A$ be a split basic semiperfect $\kk$-algebra, with a
	complete set of pairwise orthogonal primitive idempotents
	$e_1,\ldots,e_m$, and put
	\[
	\mathfrak r_A=\operatorname{rad}A.
	\]
	Thus $A/\mathfrak r_A\simeq\kk^m$, and idempotents lift
	modulo $\mathfrak r_A$.
	Assume that $\mathfrak r_A/\mathfrak r_A^2$ is
	finite-dimensional over $\kk$, as holds for the endomorphism
	orders considered here.
	
	For vertices $a,b$, define the arrow space
	\[
	\operatorname{Arr}_A(a,b)
	=
	\frac{e_b\mathfrak r_A e_a}
	{e_b\mathfrak r_A^2 e_a}.
	\]
	The Gabriel quiver $Q_A$ has one vertex for each $e_a$ and
	$\dim_\kk\operatorname{Arr}_A(a,b)$
	arrows from $a$ to $b$.
	
	A choice of frozen vertices gives $Q_A$ its ice-quiver
	structure. We write $Q_A^{\mathrm{fr}}$ for the full
	subquiver on these vertices. Thus
	$Q_A^{\mathrm{fr}}$ records every arrow whose source and
	target are both frozen.
\end{definition}
	For $A_{\cI}$, put
	\[
	Q_{\cI}=Q_{A_{\cI}},
	\qquad
	(Q_{\cI})_0=\cS,
	\qquad
	(Q_{\cI}^{\mathrm{fr}})_0=\cI.
	\]
	All arrows of $Q_{\cI}$ are retained, including loops and
	oriented $2$-cycles when present.
	
	The associated extended exchange matrix is
	\[
	\bB_{A_{\cI}}
	=
	(b_{ba})_{
		b\in\cS,\,
		a\in\cS\setminus\cI},
	\qquad
	b_{ba}
	=
	\dim_\kk\operatorname{Arr}_{A_{\cI}}(a,b)
	-
	\dim_\kk\operatorname{Arr}_{A_{\cI}}(b,a).
	\]
	In particular, the matrix records the signed arrow counts
	involving mutable vertices, while
	$Q_{\cI}^{\mathrm{fr}}$ retains the additional
	frozen--frozen arrow data.

Choose a maximal weakly separated collection
\[
\widehat{\cS}\subseteq\binom{[n]}k
\qquad\text{with}\qquad
\cS\subseteq\widehat{\cS},
\]
and put
\[
\widehat T=\bigoplus_{J\in\widehat{\cS}}M_J,
\qquad
\widehat A=\End_C(\widehat T)^{\mathrm{op}},
\qquad
e=\sum_{J\in\cS}e_J.
\]
Write $\widehat Q$ for the Gabriel quiver of $\widehat A$.

We use the Grassmannian endomorphism-quiver description of
\cite{BKM}, with the same opposite-endomorphism convention as for
$A_{\cI}$. With our clockwise planar convention, the arrows run
clockwise around white cells and counterclockwise around black
cells. Equivalently, a dual arrow crossing an internal plabic
edge has its white endpoint on the right.

\begin{definition}
	\label{def:necklace-inside-domain}
	Choose points $p_1,\ldots,p_n$ in sufficiently general
	strictly convex clockwise position, and put
	\[
	p(J)=\sum_{j\in J}p_j.
	\]
	The plabic tiling of $\widehat{\cS}$ has vertices
	$p(J)$ for $J\in\widehat{\cS}$.
	Its white cells are the nondegenerate polygons formed by
	labels $H\cup\{a\}$ with a common $(k-1)$-subset $H$.
	Its black cells are the nondegenerate polygons formed by
	labels $L\setminus\{a\}$ with a common $(k+1)$-subset $L$.
	The sides of these polygons are the tiling edges.
\end{definition}
	The boundary necklace $\mathbf I=(I_1,\ldots,I_n)$
	determines the closed polygonal curve
	\[
	\zeta(\mathbf I)
	=
	\bigcup_{a=1}^n[p(I_a),p(I_{a+1})],
	\qquad I_{n+1}=I_1.
	\]
	The inside domain $D^{\mathrm{in}}_{\mathbf I}$ consists
	of the $k$-subsets $J$ that are weakly separated from every
	$I_a$ and whose points $p(J)$ lie on or inside this curve.
	When repeated necklace terms occur, the enclosed region
	is the union of the closed regions bounded by its simple
	closed components.
	
	We write
	\[
	V_{\mathrm{in}}=\cS\setminus\cI,
	\qquad
	V_{\mathrm{out}}=\widehat{\cS}\setminus\cS.
	\]
	Thus $V_{\mathrm{in}}$ consists of the retained mutable
	vertices, and $V_{\mathrm{out}}$ consists of the vertices
	omitted from $e\widehat A e$.

In the connected nondegenerate case, the necklace terms are
distinct and $\zeta(\mathbf I)$ is simple. For $n=2$, the curve
may be a segment traversed twice; the inside domain is then
interpreted using the labels on that segment. There are no
mutable vertices in this case, but the complete frozen quiver
is still retained.

The inside-domain correspondence gives
\begin{equation}
	\label{eq:fsb-inside-intersection}
	\widehat{\cS}\cap D^{\mathrm{in}}_{\mathbf I}
	=
	\cS;
\end{equation}
see \cite[Theorem~6.3]{FG} and
\cite[Remark~4.20 and Lemma~7.3]{FSB}.
Indeed, $\cS$ is maximal weakly separated in the inside domain.
The left-hand side is a weakly separated collection containing
$\cS$, so maximality gives the equality.

\begin{example}[An inside domain in $\operatorname{Gr}(2,5)$]
	\label{ex:fsb-inside-domain-gr25}
	We write $ij$ for the subset $\{i,j\}$.
	Consider the permutation $\pi=24513$, whose forward
	Grassmann necklace is
	\[
	\mathbf I=(13,23,34,45,15).
	\]
	Let
	$\cI=\{13,23,34,45,15\}$
	be the set of its necklace terms, and choose
	\[
	\cS=\cI\cup\{14\}
	=\{13,14,15,23,34,45\}.
	\]
	
	The positroid associated with $\pi$ is
	\[
	\mathcal M_\pi
	=
	\binom{[5]}2\setminus\{12\}.
	\]
	The collection $\cS$ is maximal weakly separated in
	$\mathcal M_\pi$ and contains $\cI$.
	Indeed, its members are pairwise weakly separated,
	while the remaining labels $24,25,35$ cannot be added:
	$24$ and $25$ are not weakly separated from $13$,
	and $35$ is not weakly separated from $14$.
	
	By \cite[Theorem~2.9]{FSB}, we may therefore choose
	a reduced plabic graph $G$ with trip permutation $\pi$
	whose target face labels are precisely $\cS$.
	Thus $\cS$ is the set of all target face labels,
	$\cI$ is the set of boundary labels, and $14$ is the
	unique internal label.
	This is the case $\rho=\mathrm{id}$ of the relabeled
	construction.
	
	To obtain an ambient maximal weakly separated
	collection, we add the label $12$:
	\[
	\widehat{\cS}
	=
	\cS\cup\{12\}
	=
	\{12,13,14,15,23,34,45\}.
	\]
	This collection consists of the five sides of a
	pentagon together with the two noncrossing diagonals
	$13$ and $14$, so it is maximal weakly separated
	in $\binom{[5]}2$.
	
	In the plabic tiling of $\widehat{\cS}$, the labels
	$12,13,14,15$ form a white cell because they contain
	the common subset $\{1\}$.
	The labels $13,14,34$ form a black cell because they
	are contained in the common subset $\{1,3,4\}$.
	The necklace curve follows the points in the order
	\[
	p(13),\ p(23),\ p(34),\ p(45),\ p(15),\ p(13).
	\]
	
	The following picture shows the ambient tiling.
	Black cells are shaded gray, and the orange curve
	is $\zeta(\mathbf I)$.
	Each displayed label $J$ denotes the point $p(J)$.
	The hollow point labelled $35$ marks an additional
	candidate label; it is not a vertex of the chosen
	tiling.
	
	
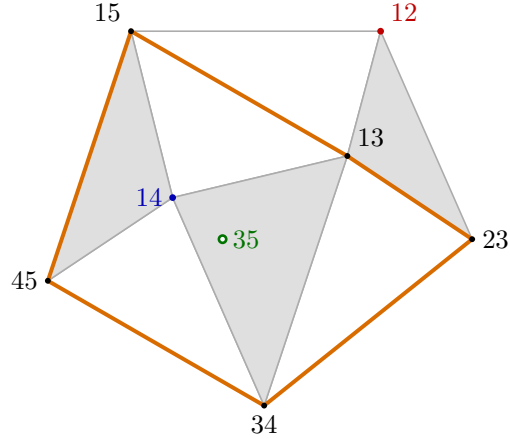
\begin{center}
		\begin{tikzpicture}[
			scale=0.55,
			every node/.style={font=\small},
			tiling edge/.style={
				draw=gray!65,
				line width=0.6pt
			},
			necklace curve/.style={
				draw=orange!85!black,
				line width=1.5pt
			}
			]
			\coordinate (p1) at (0,3);
			\coordinate (p2) at (3,1);
			\coordinate (p3) at (2.2,-2);
			\coordinate (p4) at (-2,-3);
			\coordinate (p5) at (-3,1);
			
			\coordinate (q12) at ($(p1)+(p2)$);
			\coordinate (q13) at ($(p1)+(p3)$);
			\coordinate (q14) at ($(p1)+(p4)$);
			\coordinate (q15) at ($(p1)+(p5)$);
			\coordinate (q23) at ($(p2)+(p3)$);
			\coordinate (q34) at ($(p3)+(p4)$);
			\coordinate (q45) at ($(p4)+(p5)$);
			\coordinate (q35) at ($(p3)+(p5)$);
			
			\filldraw[fill=white,tiling edge]
			(q12)--(q13)--(q14)--(q15)--cycle;
			\filldraw[fill=white,tiling edge]
			(q13)--(q23)--(q34)--cycle;
			\filldraw[fill=white,tiling edge]
			(q14)--(q34)--(q45)--cycle;
			
			\filldraw[fill=gray!25,tiling edge]
			(q12)--(q13)--(q23)--cycle;
			\filldraw[fill=gray!25,tiling edge]
			(q13)--(q14)--(q34)--cycle;
			\filldraw[fill=gray!25,tiling edge]
			(q14)--(q15)--(q45)--cycle;
			
			\draw[necklace curve]
			(q13)--(q23)--(q34)--(q45)--(q15)--cycle;
			
			\fill (q13) circle (2pt);
			\fill (q23) circle (2pt);
			\fill (q34) circle (2pt);
			\fill (q45) circle (2pt);
			\fill (q15) circle (2pt);
			
			\node[above right] at (q13) {$13$};
			\node[right]       at (q23) {$23$};
			\node[below]       at (q34) {$34$};
			\node[left]        at (q45) {$45$};
			\node[above left]  at (q15) {$15$};
			
			\fill[blue!70!black]
			(q14) circle (2.3pt);
			\node[left,text=blue!70!black]
			at (q14) {$14$};
			
			\fill[red!75!black]
			(q12) circle (2.3pt);
			\node[above right,text=red!75!black]
			at (q12) {$12$};
			
			\draw[
			draw=green!45!black,
			fill=white,
			line width=1pt
			]
			(q35) circle (2.5pt);
			\node[right,text=green!45!black]
			at (q35) {$35$};
		\end{tikzpicture}
	\end{center}
	
	The points $p(14)$ and $p(35)$ lie strictly inside
	the necklace curve, and both labels are weakly
	separated from every term of $\mathbf I$.
	The label $12$ is also weakly separated from every
	necklace term, but its point lies outside the curve.
	The remaining labels $24$ and $25$ are excluded
	because neither is weakly separated from $13$.
	Hence
	\[
	D^{\mathrm{in}}_{\mathbf I}
	=
	\cI\cup\{14,35\}.
	\]
	
	In particular, the inside domain and the selected
	face collection are different:
	\[
	D^{\mathrm{in}}_{\mathbf I}
	=
	\cI\cup\{14,35\},
	\qquad
	\cS=\cI\cup\{14\}.
	\]
	The label $35$ belongs to the inside domain but
	does not belong to $\widehat{\cS}$.
	Consequently,
	$$\widehat{\cS}\cap D^{\mathrm{in}}_{\mathbf I}
	=\cS,$$
	$$ V_{\mathrm{in}}
	=\cS\setminus\cI=\{14\},$$
	$$V_{\mathrm{out}}
	=\widehat{\cS}\setminus\cS=\{12\}.$$
	The two internal candidates $14$ and $35$ are not
	weakly separated from each other, so they cannot
	both occur in a weakly separated face collection.
\end{example}

\begin{proposition}
	\label{prop:fsb-complete-ice-quiver}
	There is an isomorphism
	\begin{equation}
		\label{eq:fsb-face-corner}
		A_{\cI}\simeq e\widehat A e
	\end{equation}
	identifying the primitive idempotents indexed by $\cS$.
	Put $\mathfrak r=\operatorname{rad}\widehat A$.
	
	For every $a,b\in\cS$, define
	\[
	\mathcal K_{ba}
	=
	\frac{e_b\mathfrak r^2e_a}
	{e_b\mathfrak r e\mathfrak r e_a}.
	\]
	Then the following statements hold.
	
	\begin{enumerate}
		\item
		The complete arrow space of $A_{\cI}$ is
		\begin{equation}
			\label{eq:fsb-complete-arrow-space}
			\operatorname{Arr}_{A_{\cI}}(a,b)
			\simeq
			\frac{e_b\mathfrak r e_a}
			{e_b\mathfrak r e\mathfrak r e_a}.
		\end{equation}
		Under this identification, there is a natural short
		exact sequence
		\begin{equation}
			\label{eq:fsb-corner-arrow-sequence}
			0\longrightarrow
			\mathcal K_{ba}
			\longrightarrow
			\operatorname{Arr}_{A_{\cI}}(a,b)
			\longrightarrow
			\operatorname{Arr}_{\widehat A}(a,b)
			\longrightarrow0.
		\end{equation}
		In particular,
		\begin{equation}
			\label{eq:fsb-complete-arrow-count}
			\#\{a\to b\text{ in }Q_{\cI}\}
			=
			\#\{a\to b\text{ in }\widehat Q\}
			+
			\dim_\kk\mathcal K_{ba}.
		\end{equation}
		
		\item
		If at least one of $a,b$ belongs to
		$V_{\mathrm{in}}$, then $\mathcal K_{ba}=0$.
		Moreover, $Q_{\cI}$ and $Q(G^\rho)$ agree at every
		arrow having a mutable endpoint, with the same
		orientation. Consequently,
		\[
		\bB_{A_{\cI}}=\bB_{G^\rho}.
		\]
		
		\item
		The full frozen subquiver $Q_{\cI}^{\mathrm{fr}}$
		has vertex set $\cI$ and, for every $a,b\in\cI$,
		\[
		\#\{a\to b\text{ in }Q_{\cI}^{\mathrm{fr}}\}
		=
		\#\{a\to b\text{ in }\widehat Q\}
		+
		\dim_\kk\mathcal K_{ba}.
		\]
		Thus any additional arrows arising from the corner
		have both endpoints frozen.
	\end{enumerate}
\end{proposition}

\begin{proof}
	By Lemma~\ref{lem:necklace-orthogonal-images} and the
	full faithfulness of restriction from $B$ to $C$, we have
	\[
	\begin{aligned}
		A_{\cI}
		&\simeq
		\End_B\!\left(
		\bigoplus_{J\in\cS}M_J
		\right)^{\mathrm{op}}\\
		&\simeq
		\End_C\!\left(
		\bigoplus_{J\in\cS}M_J
		\right)^{\mathrm{op}}\\
		&\simeq e\widehat A e.
	\end{aligned}
	\]
	This proves \eqref{eq:fsb-face-corner}.
	
	Notice that we have
	\[
	\operatorname{rad}(e\widehat A e)
	=
	e\mathfrak r e,
	\qquad
	\operatorname{rad}(e\widehat A e)^2
	=
	e\mathfrak r e\mathfrak r e.
	\]
	Consequently, for all retained vertices $a,b$, we have
	\[
	\operatorname{Arr}_{A_{\cI}}(a,b)
	\simeq
	\frac{e_b\mathfrak r e_a}
	{e_b\mathfrak r e\mathfrak r e_a}.
	\]
	The inclusions
$	e_b\mathfrak r e\mathfrak r e_a
	\subseteq
	e_b\mathfrak r^2e_a
	\subseteq
	e_b\mathfrak r e_a$
	give the short exact sequence
	\eqref{eq:fsb-corner-arrow-sequence}.
	Taking dimensions proves
	\eqref{eq:fsb-complete-arrow-count}.
	
	We next show that
	\begin{equation}
		\label{eq:fsb-inside-outside-isolation}
		\text{there are no arrows in $\widehat Q$ between
			$V_{\mathrm{in}}$ and $V_{\mathrm{out}}$.}
	\end{equation}
	The curve $\zeta(\mathbf I)$ is formed from tiling edges
	and chords lying inside individual white or black cells.
	A chord of this kind does not cross the interior of a
	tiling edge. An edge joining a vertex strictly inside
	the curve to one strictly outside it would have to
	cross the curve. Such an edge therefore cannot occur.
	
	The inside-domain correspondence identifies the tiling
	inside $\zeta(\mathbf I)$ with the relabeled plabic
	tiling; see \cite[Section~6]{FG} and
	\cite[Remark~4.20, Lemma~7.3 and Section~7.3]{FSB}.
	It preserves the white and black cells incident with
	internal vertices. Any new side introduced by cutting
	a cell joins two necklace vertices. Thus the ambient
	quiver and $Q(G^\rho)$ have the same arrows, with the
	same orientations, at every vertex in
	$V_{\mathrm{in}}$.
	
 We next claim that
	\begin{equation}
		\label{eq:fsb-corner-radical}
		e_b\mathfrak r^2e_a
		=
		e_b\mathfrak r e\mathfrak r e_a
		\qquad
		\text{if }a\in V_{\mathrm{in}}
		\text{ or }b\in V_{\mathrm{in}}.
	\end{equation}
	The inclusion from right to left is immediate.
	
	For the reverse inclusion, use the
	quiver presentation of $\widehat A$.
	If $a\in V_{\mathrm{in}}$, the first intermediate
	vertex of every path of length at least two starting
	at $a$ is retained, by
	\eqref{eq:fsb-inside-outside-isolation}.
	Cutting the path at that vertex expresses it as an
	element of
	$e_b\mathfrak r e\mathfrak r e_a$.
	If $b\in V_{\mathrm{in}}$, the same argument applies
	to the last intermediate vertex.
	
	This argument also applies to completed sums.
	Indeed, $\widehat A$ is finite over the complete
	discrete valuation ring $Z$. Its radical-adic and
	$t$-adic topologies agree, and the displayed
	$Z$-submodules are closed. Hence the path
	factorizations imply
	\eqref{eq:fsb-corner-radical}.
	
	It follows that $\mathcal K_{ba}=0$ whenever at least
	one endpoint is mutable. The short exact sequence
	\eqref{eq:fsb-corner-arrow-sequence} therefore identifies
	the arrow spaces of $A_{\cI}$ with those of $\widehat A$
	whenever at least one endpoint is mutable.
	Together with the inside-domain comparison and the
	chosen orientation convention, this proves
	\[
	\bB_{A_{\cI}}=\bB_{G^\rho}.
	\]
	
	Finally, for $a,b\in\cI$, the arrow-count formula
	\eqref{eq:fsb-complete-arrow-count} gives the stated
	description of $Q_{\cI}^{\mathrm{fr}}$.
\end{proof}

\begin{remark}[The additional frozen arrows]
	\label{rem:fsb-frozen-arrow-correction}
	The spaces $\mathcal K_{ba}$ record the arrows that
	can appear when the vertices in $V_{\mathrm{out}}$
	are omitted. More explicitly,
	\[
	\mathcal K_{ba}
	\simeq
	\frac{
		e_b\mathfrak r(1-e)\mathfrak r e_a
	}{
		e_b\mathfrak r(1-e)\mathfrak r e_a
		\cap
		e_b\mathfrak r e\mathfrak r e_a
	}.
	\]
	Thus these additional arrows are represented by paths
	through omitted vertices, modulo the relations of
	$\widehat A$ and factorizations through retained
	intermediate vertices.
	
	Choosing a basis of $\mathcal K_{ba}$ and lifts of an
	ambient arrow basis gives an arrow basis for the corner.
	These choices are not canonical. The complete arrow
	spaces themselves are intrinsic to $A_{\cI}$ and
	therefore independent of the auxiliary collection
	$\widehat{\cS}$.
	
	The proposition retains all frozen--frozen arrows.
	Identifying them individually with those of a specified
	complete dual quiver $Q(G^\rho)$ would additionally
	require proving
	\[
	\#\{a\to b\text{ in }Q(G^\rho)\}
	=
	\dim_\kk
	\frac{e_b\mathfrak r e_a}
	{e_b\mathfrak r e\mathfrak r e_a}
	\qquad(a,b\in\cI).
	\]
	This additional identification is not needed for the
	equality of extended exchange matrices.
\end{remark}

\begin{example}[The associated ice quivers]
	\label{ex:fsb-corner-gr25}
	Continue Example~\ref{ex:fsb-inside-domain-gr25}, with
	\[
	\cS=\cI\cup\{14\},
	\qquad
	\widehat{\cS}=\cS\cup\{12\}.
	\]
	Working over $C=C_{2,5}$, put
	\[
	\begin{aligned}
		T_{\cS}&=\bigoplus_{J\in\cS}M_J,
		&\widehat T&=T_{\cS}\oplus M_{12},\\
		\widehat A&=\End_C(\widehat T)^{\mathrm{op}},
		&e&=1-e_{12}.
	\end{aligned}
	\]
	Then
	\[
	A_{\cI}
	\simeq\End_C(T_{\cS})^{\mathrm{op}}
	\simeq e\widehat A e.
	\]
	
	The Gabriel quiver $\widehat Q$ of
	$\widehat A$ is displayed below.
	Blue boxes denote frozen vertices, and blue arrows
	have both endpoints frozen.
	In the ambient Grassmannian seed, the mutable
	vertices are $13$ and $14$.
	\[
	\widehat Q:\qquad
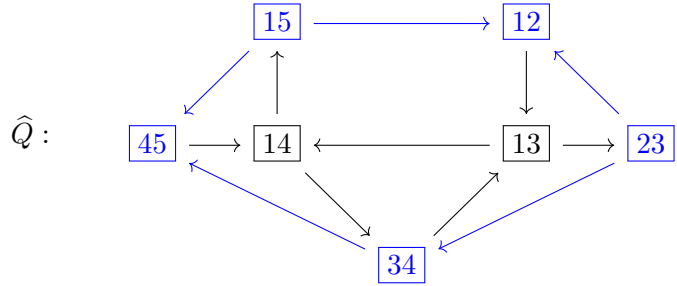
	\begin{tikzcd}[row sep=2.2em,column sep=1.8em]
		&
		\color{blue}\boxed{15}
		\arrow[rr,blue]
		\arrow[dl,blue]
		&&
		\color{blue}\boxed{12}
		\arrow[d,black]
		&\\
		\color{blue}\boxed{45}
		\arrow[r,black]
		&
		\boxed{14}
		\arrow[u,black]
		\arrow[dr,black]
		&&
		\boxed{13}
		\arrow[ll,black]
		\arrow[r,black]
		&
		\color{blue}\boxed{23}
		\arrow[ul,blue]
		\arrow[dll,blue]
		\\
		&&
		\color{blue}\boxed{34}
		\arrow[ur,black]
		\arrow[ull,blue]
		&&
	\end{tikzcd}
	\]
	
	We compute the additional arrow arising from the
	omission of $12$.
	For rank-one modules, write
	\[
	\Hom_C(M_I,M_J)=Zg_{J,I},
	\]
	where $g_{J,I}$ is the monomial generator whose
	vertex components are powers of $t$, with at least
	one component equal to $1$; see
	\cite[Lemma~7.4]{BKM}.
	
	The path $15\to12\to13$ in $\widehat Q$
	corresponds, under our opposite-endomorphism
	convention, to
	\[
	M_{13}\xrightarrow{g_{12,13}}M_{12}
	\xrightarrow{g_{15,12}}M_{15}.
	\]
	In the vertex order $j+\tfrac12$, $j=1,\ldots,5$,
	direct calculation gives
	\[
	\begin{aligned}
		g_{12,13}&=(1,t,1,1,1),\\
		g_{15,12}&=(t,1,1,1,t),\\
		f:=g_{15,13}&=(t,t,1,1,t).
	\end{aligned}
	\]
	Consequently,
	\[
	g_{15,12}g_{12,13}=f.
	\]
	For every other possible intermediate summand,
	the same calculation gives
	\[
	g_{15,K}g_{K,13}=tf
	\qquad
	(K\in\{14,23,34,45\}).
	\]
	Since $\End_C(M_I)=Z$, radical endomorphisms of
	the endpoints also produce multiples of $tf$.
	Therefore
	\[
	\operatorname{rad}_{\add T_{\cS}}^2(M_{13},M_{15})
	=tZf,
	\]
	and hence
	\begin{equation}
		\label{eq:fsb-gr25-corner-new-arrow}
		\operatorname{Irr}_{\add T_{\cS}}(M_{13},M_{15})
		=
		Zf/tZf
		\simeq\kk.
	\end{equation}
	Thus the class of $f$ gives a new arrow
	$\alpha:15\to13$ in the Gabriel quiver of $A_{\cI}$.
	
	Computing the remaining arrow spaces in the same way
	gives the following complete quiver:
	\[
	Q_{\cI}:\qquad
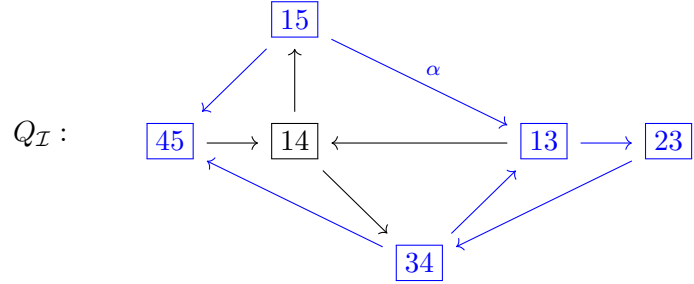
	\begin{tikzcd}[row sep=2.2em,column sep=1.8em]
		&
		\color{blue}\boxed{15}
		\arrow[dl,blue]
		\arrow[drr,blue,"\alpha"]
		&&&\\
		\color{blue}\boxed{45}
		\arrow[r,black]
		&
		\boxed{14}
		\arrow[u,black]
		\arrow[dr,black]
		&&
		\color{blue}\boxed{13}
		\arrow[ll,black]
		\arrow[r,blue]
		&
		\color{blue}\boxed{23}
		\arrow[dll,blue]
		\\
		&&
		\color{blue}\boxed{34}
		\arrow[ur,blue]
		\arrow[ull,blue]
		&&
	\end{tikzcd}
	\]
	Every arrow has multiplicity one, and there are no
	loops. The frozen vertices are precisely
	\[
	\cI=\{13,23,34,45,15\},
	\]
	so $14$ is the unique mutable vertex.
	In particular, $13$ is mutable in the ambient
	Grassmannian seed but frozen in $Q_{\cI}$.
	The only additional arrow is $\alpha:15\to13$,
	and all arrows incident with $14$ are unchanged.
	
	With the extended-cluster order
	\[
	X=(\Delta_{14},\Delta_{13},\Delta_{23},
	\Delta_{34},\Delta_{45},\Delta_{15}),
	\]
	our convention
	$b_{ij}=\#\{j\to i\}-\#\{i\to j\}$ gives
	\[
	\bB_{A_{\cI}}
	=
	\begin{pmatrix}
		0\\-1\\0\\1\\-1\\1
	\end{pmatrix},
	\qquad
	\widehat y
	=
	\frac{\Delta_{34}\Delta_{15}}
	{\Delta_{13}\Delta_{45}}.
	\]
	Mutation replaces $\Delta_{14}$ by $\Delta_{35}$:
	\[
	\Delta_{14}\Delta_{35}
	=
	\Delta_{13}\Delta_{45}
	+
	\Delta_{15}\Delta_{34}.
	\]
\end{example}

\subsection{Internal mutations and cluster characters}
Recall the full subcategory
\[
\mathcal L_{\cI}
=
\left\{
M\in\CM(B):
\Ext_B^1(P_{\cI},M)=0
=
\Ext_B^1(M,P_{\cI})
\right\}.
\]
By Lemma~\ref{lem:necklace-orthogonal-images}, the functor
\[
H=\Hom_B(P_{\cI},-):
\mathcal L_{\cI}\longrightarrow
\cE_{\cI}=\GP B_{\cI}
\]
is exact and fully faithful. Via restriction along $C\to B$,
we also regard $\mathcal L_{\cI}$ as a full exact subcategory
of the ambient Grassmannian Frobenius category $\CM(C)$.

Starting from $\widehat T$, we perform mutations in $\CM(C)$
only at vertices in $V_{\mathrm{in}}$, keeping the summands
indexed by $\cI$ and $V_{\mathrm{out}}$ fixed.
At each stage, write $M_j$ for the current $C$-lattice at
vertex $j$; these lattices need not remain rank one.
The retained direct sum is
\[
W=\bigoplus_{j\in V_{\mathrm{in}}\sqcup\cI}M_j.
\]
The following proposition shows that the retained summands
remain in $\mathcal L_{\cI}$ and that the ambient exchange
sequences lie in this subcategory. Applying $H$ then gives
the exchange conflations for $H(W)$ in $\cE_{\cI}$.

\begin{proposition}[Mutation compatibility]
	\label{prop:fsb-frobenius-mutations}
	Let $V$ be a basic cluster-tilting object reachable from
	$T_{\cI}$ in $\cE_{\cI}$, and let $k$ be a mutable vertex.
	The corresponding mutation in $\CM(C)$ induces
	the mutation $\mu_kV$ in $\cE_{\cI}$.
	
	Its two exchange conflations are obtained by applying
	$\Hom_B(P_{\cI},-)$ to the exchange sequences in $\CM(C)$,
	regarded as exact sequences of $B$-lattices.
	Moreover, the extended exchange matrices satisfy
	\[
	\bB_{\mu_kV}=\mu_k(\bB_V).
	\]

	Thus $(\cE_{\cI},T_{\cI})$ 
	realizes the mutation pattern
	of the target-labelled seed of $G^\rho$.
\end{proposition}
\begin{proof}
	Put
	\[
	H=\Hom_B(P_{\cI},-):
	\mathcal L_{\cI}\longrightarrow
	\cE_{\cI}=\GP B_{\cI},
	\qquad
	T_{\mathrm{out}}
	=
	\bigoplus_{J\in V_{\mathrm{out}}}M_J.
	\]
	By Lemma~\ref{lem:necklace-orthogonal-images}, $H$ is
	exact and fully faithful.
	
	We construct the mutations in $\CM(C)$, starting from
	\[
	\widehat T
	=
	\left(\bigoplus_{J\in\cS}M_J\right)
	\oplus T_{\mathrm{out}}.
	\]
	Only the summands at vertices in $V_{\mathrm{in}}$ are
	mutated. The summands indexed by $\cI$ and
	$V_{\mathrm{out}}$ remain fixed.
	After mutation, vertex labels denote positions in the
	mutation pattern, and the corresponding $C$-lattices
	need not remain rank one.
	If $V_{\mathrm{in}}=\varnothing$, there is nothing to prove.
	
	We first check that each initial summand $M_J$ with
	$J\in V_{\mathrm{in}}$ is non-projective in $\CM(C)$.
	Suppose otherwise. The restriction injections
	\[
	\Ext_B^1(M_J,N)
	\hookrightarrow
	\Ext_C^1(M_J,N)
	\qquad(N\in\CM(B))
	\]
	would give
	\[
	\Ext_B^1(M_J,N)=0
	\qquad(N\in\CM(B)).
	\]
	Choose a projective $B$-deflation
	\[
	0\longrightarrow K\longrightarrow Q
	\longrightarrow M_J\longrightarrow0.
	\]
	Its kernel is a $B$-lattice, so this sequence splits.
	Thus $M_J$ is $B$-projective and $qM_J=0$, contrary to
	Proposition~\ref{prop:fsb-common-face-cluster}.
	Hence mutation in $\CM(C)$ is available at every
	vertex in $V_{\mathrm{in}}$.
	
	For a cluster-tilting object $\widehat W$ in $\CM(C)$,
	write
	\[
	\widehat A_{\widehat W}
	=
	\End_C(\widehat W)^{\mathrm{op}},
	\]
	and let $\widehat Q_{\widehat W}$ and
	$\widehat{\bB}_{\widehat W}$ be its Gabriel quiver
	and extended exchange matrix, respectively.
	In this matrix, the frozen vertices correspond to the
	projective-injective summands in $\CM(C)$.
	The mutation results for $\CM(C)$ give exchange
	sequences, the absence of loops and oriented $2$-cycles,
	and the identity
	\[
	\widehat{\bB}_{\mu_k\widehat W}
	=
	\mu_k\bigl(\widehat{\bB}_{\widehat W}\bigr);
	\]
	see \cite[Proposition~4.2 and Section~6]{JKSQuantum}.
	
	By
	Proposition~\ref{prop:fsb-complete-ice-quiver}, $\widehat Q_{\widehat T}$ has no arrows
	between $V_{\mathrm{in}}$ and $V_{\mathrm{out}}$.
	This remains true after any sequence of mutations
	at vertices in $V_{\mathrm{in}}$.
	Indeed, suppose that the current matrix satisfies
	\[
	\widehat b_{oj}=0
	\qquad
	(o\in V_{\mathrm{out}},\ j\in V_{\mathrm{in}}).
	\]
	For $k\in V_{\mathrm{in}}$, the matrix mutation formula,
	together with $\widehat b_{ok}=0$, gives
	\[
	(\mu_k\widehat{\bB})_{oj}=0
	\qquad
	(o\in V_{\mathrm{out}},\ j\in V_{\mathrm{in}}).
	\]
	Since the Gabriel quiver after mutation has no oriented
	$2$-cycles, these zero entries imply the same absence
	of arrows.
	
	We now proceed by induction on the number of mutations.
	At a current stage, write
	\[
	W=\bigoplus_{j\in\cS}M_j,
	\qquad
	\widehat W=W\oplus T_{\mathrm{out}},
	\qquad
	V=H(W).
	\]
	Assume that $\widehat W$ is cluster-tilting in $\CM(C)$,
	all summands of $W$ belong to $\mathcal L_{\cI}$,
	and $V$ is cluster-tilting in $\cE_{\cI}$.
	These assertions hold initially by
	Lemma~\ref{lem:necklace-orthogonal-images} and
	Proposition~\ref{prop:fsb-frobenius-realization}.
	
	Fix $k\in V_{\mathrm{in}}$, and write
	\[
	W=M_k\oplus W_{\widehat k}.
	\]
	The exchange sequences at $M_k$ in $\CM(C)$ are
	\begin{equation}
		\label{eq:fsb-ambient-exchanges}
		\begin{gathered}
			0\longrightarrow M_k^*
			\longrightarrow E_+
			\xrightarrow{u}M_k
			\longrightarrow0,\\
			0\longrightarrow M_k
			\xrightarrow{v}E_-
			\longrightarrow M_k^*
			\longrightarrow0.
		\end{gathered}
	\end{equation}
	Their middle terms are determined by the arrows
	incident with $k$ in $\widehat Q_{\widehat W}$.
	The absence of arrows between $k$ and
	$V_{\mathrm{out}}$ therefore implies
	\[
	E_+,E_-\in\add W_{\widehat k}.
	\]
	Thus $E_+$, $E_-$, and $M_k$ are already $B$-lattices.
	
	We next equip $M_k^*$ with a $B$-module structure.
	Full faithfulness of restriction
	$\CM(B)\hookrightarrow\CM(C)$ makes $u$ and $v$
	$B$-linear. Set
	\[
	K=\ker_B(u),
	\qquad
	L=\operatorname{coker}_B(v).
	\]
	After restriction to $C$, both $K$ and $L$ are
	isomorphic to the $C$-lattice $M_k^*$.
	In particular, both are finite free over $Z$ and
	belong to $\CM(B)$.
	Their $C$-isomorphism lifts to a $B$-isomorphism by
	full faithfulness.
	We may therefore regard $M_k^*$ as a $B$-lattice
	and both sequences in
	\eqref{eq:fsb-ambient-exchanges} as short exact
	sequences of $B$-lattices.
	
	The object $P_{\cI}$ is a summand of
	$W_{\widehat k}$ and remains a summand of the
	cluster-tilting object
	\[
	\mu_k\widehat W
	=
	M_k^*\oplus W_{\widehat k}\oplus T_{\mathrm{out}}
	\qquad\text{in }\CM(C).
	\]
	Rigidity in $\CM(C)$ gives
	\[
	\Ext_C^1(P_{\cI},M_k^*)=0
	=
	\Ext_C^1(M_k^*,P_{\cI}).
	\]
	Using the restriction injections once more, we obtain
	\begin{equation}
		\label{eq:fsb-mutation-induction}
		\Ext_B^1(P_{\cI},M_k^*)=0
		=
		\Ext_B^1(M_k^*,P_{\cI}).
	\end{equation}
	Hence $M_k^*\in\mathcal L_{\cI}$.
	
	Put
	\[
	N_k=H(M_k),
	\qquad
	N_k^*=H(M_k^*),
	\qquad
	V_{\widehat k}=H(W_{\widehat k}).
	\]
	Exactness of $H$ gives conflations in $\cE_{\cI}$:
	\begin{equation}
		\label{eq:fsb-induced-exchange-conflations}
		\begin{gathered}
			0\longrightarrow N_k^*
			\longrightarrow H(E_+)
			\xrightarrow{H(u)}N_k
			\longrightarrow0,\\
			0\longrightarrow N_k
			\xrightarrow{H(v)}H(E_-)
			\longrightarrow N_k^*
			\longrightarrow0.
		\end{gathered}
	\end{equation}
	The approximation properties of
	\eqref{eq:fsb-ambient-exchanges} in $\CM(C)$
	restrict to $\add W_{\widehat k}$.
	Full faithfulness of restriction and of $H$ then
	shows that $H(u)$ and $H(v)$ are respectively
	minimal right and left
	$\add V_{\widehat k}$-approximations.
	The other endpoint maps have the corresponding
	approximation properties.
	The conflations are non-split, since a splitting
	would lift to a splitting of the corresponding
	exchange sequence in $\CM(C)$.
	
	Full faithfulness also shows that $N_k^*$ is
	indecomposable and is not isomorphic to any summand
	of $V$.
	Moreover, it is non-projective: every indecomposable
	projective of $\cE_{\cI}$ belongs to
	$\add H(P_{\cI})\subseteq\add V_{\widehat k}$.
	The exchange-complement theorem in the Hom-finite
	$2$-Calabi--Yau category $\underline{\cE}_{\cI}$
	identifies $N_k^*$ as the mutation complement of
	$N_k$; see \cite[Theorem~5.3]{IyamaYoshino}.
	Consequently,
	\[
	\mu_kV
	=
	V_{\widehat k}\oplus N_k^*
	=
	H(W_{\widehat k}\oplus M_k^*),
	\]
	and \eqref{eq:fsb-induced-exchange-conflations}
	are the exchange conflations in $\cE_{\cI}$.
	This proves the induction step.
	
	Finally, we compare the extended exchange matrices.
	At every stage, by full faithfulness, we have
	\[
	\begin{aligned}
		A_V
		:=
		\End_{B_{\cI}}(V)^{\mathrm{op}}
		&\simeq \End_B(W)^{\mathrm{op}}\\
		&\simeq \End_C(W)^{\mathrm{op}}\\
		&\simeq e_W\widehat A_{\widehat W}e_W,
	\end{aligned}
	\]
	where $e_W$ is the idempotent corresponding to $W$.
	The Gabriel quiver of $\widehat A_{\widehat W}$
	has no arrows between $V_{\mathrm{in}}$ and
	$V_{\mathrm{out}}$.
	The radical argument in
	Proposition~\ref{prop:fsb-complete-ice-quiver}
	therefore identifies the arrow spaces of $A_V$
	with those of $\widehat A_{\widehat W}$
	whenever at least one endpoint belongs to
	$V_{\mathrm{in}}$.
	It follows that
	\[
	\bB_V
	=
	\left(
	\widehat{\bB}_{\widehat W}
	\right)_{\cS\times V_{\mathrm{in}}}.
	\]
	The rows indexed by $\cI$ are retained and are
	regarded as frozen rows for $\cE_{\cI}$.
	
	For $k\in V_{\mathrm{in}}$, restriction to these rows
	and columns commutes with matrix mutation.
	Thus we have
	\[
	\begin{aligned}
		\bB_{\mu_kV}
		&=
		\left(
		\widehat{\bB}_{\mu_k\widehat W}
		\right)_{\cS\times V_{\mathrm{in}}}\\
		&=
		\left(
		\mu_k(\widehat{\bB}_{\widehat W})
		\right)_{\cS\times V_{\mathrm{in}}}\\
		&=
		\mu_k\!\left(
		\left(
		\widehat{\bB}_{\widehat W}
		\right)_{\cS\times V_{\mathrm{in}}}
		\right)\\
		&=\mu_k(\bB_V).
	\end{aligned}
	\]
	This equality includes all rows indexed by $\cI$.
	
	The initial identity
	\[
	\bB_{T_{\cI}}=\bB_{G^\rho}
	\]
	follows from
	Proposition~\ref{prop:fsb-complete-ice-quiver}.
	The construction applies along every finite sequence
	of mutations at vertices in $V_{\mathrm{in}}$,
	and therefore realizes the mutation pattern of
	the target-labeled seed of $G^\rho$.
\end{proof}

We now verify Assumption~\ref{ass:fsb-QP-realization}
for the stable category $\underline{\cE}_{\cI}$
with cluster-tilting object $\underline T_{\cI}$.
By Proposition~\ref{prop:fsb-common-face-cluster},
there is a mutation sequence
$\boldsymbol k=(k_1,\ldots,k_\ell)$ such that
\[
U
:=
\bigoplus_{J\in V_{\mathrm{in}}}qM_J
\simeq
\mu_{\boldsymbol k}(qT^{\mathrm s}),
\qquad
\overline L_{\cI}(\underline T_{\cI})\simeq U,
\]
where
\[
\overline L_{\cI}:
\underline{\cE}_{\cI}\xrightarrow{\sim}\cC
\]
is the stable equivalence induced by $L_{\cI}$.

Let $(Q_{\mathrm s},W_{\mathrm s})$ be the standard
source QP realization, and put
\[
(Q_U,W_U)
=
\mu_{\boldsymbol k}(Q_{\mathrm s},W_{\mathrm s}).
\]
Mutation compatibility of the source realization gives
a triangle equivalence
\[
\Theta_U:
\mathcal C(Q_U,W_U)\xrightarrow{\sim}\cC,
\qquad
\Theta_U(T_U^{\mathrm{can}})\simeq U,
\]
where $T_U^{\mathrm{can}}$ is the canonical
cluster-tilting object.
Moreover,
\[
\dim_\kk\operatorname{Jac}(Q_U,W_U)<\infty,
\qquad
(Q_U,W_U)\text{ is non-degenerate}.
\]

Choose a quasi-inverse of $\overline L_{\cI}$ and define
\[
\Theta_{\cI}
=
\overline L_{\cI}^{-1}\circ\Theta_U:
\mathcal C(Q_U,W_U)
\xrightarrow{\sim}
\underline{\cE}_{\cI}.
\]
Then
\[
\Theta_{\cI}(T_U^{\mathrm{can}})
\simeq
\overline L_{\cI}^{-1}(U)
\simeq
\underline T_{\cI}.
\]
For every $Y\in\underline{\cE}_{\cI}$, this equivalence
also identifies the character modules:
\[
\Hom_{\mathcal C(Q_U,W_U)}
\bigl(T_U^{\mathrm{can}},
\Sigma\Theta_{\cI}^{-1}Y\bigr)
\simeq
\Hom_{\underline{\cE}_{\cI}}
\bigl(\underline T_{\cI},\Sigma Y\bigr).
\]
These isomorphisms respect the right actions of the
corresponding endomorphism algebras.

Finally, for any further mutation sequence
$\boldsymbol j$, put
\[
(Q_{\boldsymbol j},W_{\boldsymbol j})
=
\mu_{\boldsymbol j}(Q_U,W_U).
\]
The transported realization and
Proposition~\ref{prop:fsb-frobenius-mutations} give
\[
\Theta_{\cI}
\bigl(\mu_{\boldsymbol j}T_U^{\mathrm{can}}\bigr)
\simeq
\mu_{\boldsymbol j}\underline T_{\cI}
\simeq
\underline{\mu_{\boldsymbol j}T_{\cI}}.
\]
Thus Assumption~\ref{ass:fsb-QP-realization} holds for
$(\underline{\cE}_{\cI},\underline T_{\cI})$
and throughout its reachable mutation class.

\bigskip

Let
$V=\bigoplus_{j=1}^{r+f}V_j$
be a basic cluster-tilting object reachable from $T_{\cI}$
in $\cE_{\cI}$, ordered so that $V_1,\ldots,V_r$ are
non-projective and $V_{r+1},\ldots,V_{r+f}$ are
projective-injective.
Put
\[
A=A_V=\End_{\cE_{\cI}}(V)^{\mathrm{op}},
\qquad
H=H_V=\Hom_{\cE_{\cI}}(V,-).
\]
For each $j$, let $e_j\in A$ be the idempotent corresponding
to $V_j$, and write
\[
Q_j=H(V_j)\simeq Ae_j,
\qquad
S_j=Q_j/\operatorname{rad}Q_j.
\]
Thus $Q_j$ and $S_j$ are the corresponding indecomposable
projective and simple left $A$-modules.

Set
$e_{\mathrm{fr}}=\sum_{j=r+1}^{r+f}e_j.$ Then we have
\[
\overline A
:=
A/Ae_{\mathrm{fr}}A
\simeq
\End_{\underline{\cE}_{\cI}}(\underline V)^{\mathrm{op}}
=
\Gamma_V^{\mathrm{op}}.
\]
This algebra is finite-dimensional. In particular, left
$\overline A$-modules are identified with the right
$\Gamma_V$-modules used in
Section~\ref{sec:fsb-character-input}.

For $X\in\cE_{\cI}$, put
\[
E_V(X)
:=
\Ext^1_{\cE_{\cI}}(V,X)
\simeq
\Hom_{\underline{\cE}_{\cI}}
(\underline V,\Sigma\underline X).
\]
Precomposition defines its left $A$-module structure:
\[
a^{\mathrm{op}}\cdot\xi=a^*\xi
\qquad
\bigl(a\in\End_{\cE_{\cI}}(V)\bigr).
\]
This action factors through $\overline A$. Indeed,
\[
e_jE_V(X)\simeq\Ext^1_{\cE_{\cI}}(V_j,X)=0
\qquad(j>r).
\]
Thus $E_V(X)$ is a finite-dimensional left
$\overline A$-module supported on the mutable vertices.

For $\mathbf v=(v_1,\ldots,v_r)\in\mathbb N^r$, define
\[
\Gr_{\mathbf v}(E_V(X))
=
\left\{
N\subseteq E_V(X):
\begin{array}{l}
	N\text{ is an }\overline A\text{-submodule},\\
	\dim_\kk e_jN=v_j\quad(1\leq j\leq r)
\end{array}
\right\}.
\]
Let $\chi$ denote the topological
Euler characteristic. The associated $F$-polynomial is
\[
F_V^X(y_1,\ldots,y_r)
=
\sum_{\mathbf v\in\mathbb N^r}
\chi\bigl(\Gr_{\mathbf v}(E_V(X))\bigr)
y^{\mathbf v},
\qquad
y^{\mathbf v}=\prod_{j=1}^r y_j^{v_j}.
\]

We have the following Proposition.

\begin{proposition}
	\label{prop:necklace-full-character}
	Identify
$	K_0^{\mathrm{sp}}(\add V)
	\simeq K_0(\proj A)
	\simeq \mathbb Z^{r+f}$
	by sending $[V_j]$ to $[Q_j]$ and then to the $j$-th
	standard basis vector. Let
	\[
	x=(x_1,\ldots,x_r,z_1,\ldots,z_f)
	\]
	be algebraically independent variables corresponding
	to the summands of $V$ in this order.
	
	For $X\in\cE_{\cI}$, choose a conflation
	\[
	0\longrightarrow V_1^X
	\longrightarrow V_0^X
	\longrightarrow X
	\longrightarrow0,
	\qquad
	V_0^X,V_1^X\in\add V,
	\]
	and put
$	\ind_VX=[V_0^X]-[V_1^X].$
\begin{itemize}
	\item 	The formula
	\[
	\begin{aligned}
		\CC_V(X)
		&=
		x^{\ind_VX}
		\sum_{\mathbf v\in\mathbb N^r}
		\chi\bigl(\Gr_{\mathbf v}(E_V(X))\bigr)
		x^{\bB_V\mathbf v}\\
		&=
		x^{\ind_VX}F_V^X(x^{\bB_V})
	\end{aligned}
	\]
	defines a Frobenius cluster character with values in
	$\kk[x_1^{\pm1},\ldots,x_r^{\pm1},
	z_1^{\pm1},\ldots,z_f^{\pm1}],$
	where
	$	x^{\bB_V}
	=
	(x^{\bB_Ve_1},\ldots,x^{\bB_Ve_r}).$
	
	Its values on the initial summands are
	\[
	\CC_V(V_j)=
	\begin{cases}
		x_j,&1\leq j\leq r,\\
		z_{j-r},&r<j\leq r+f.
	\end{cases}
	\]
	\item The construction of
	Section~\ref{sec:fsb-character-input} gives a localized
	extension $\CC_{V,\mathrm{loc}}$ to $\Db(\cE_{\cI})$.
	Explicitly, for $Y\in\Db(\cE_{\cI})$, if
	\[
	\kappa(Y)
	=
	[X]+\sum_{p=1}^f a_p[V_{r+p}],
	\qquad
	X\in\cE_{\cI},\quad a_p\in\mathbb Z,
	\]
	then
	\[
	\CC_{V,\mathrm{loc}}(Y)
	=
	\CC_V(X)\prod_{p=1}^f z_p^{a_p}.
	\]
	This expression is independent of the chosen
	representative $X$.
\end{itemize}

\end{proposition}

\subsection{The connected case}
\label{sec:fsb-connected-proof}
Recall that
\[
\cS
=
\left\{
I_F^{\mathrm{tgt}}(G^\rho):
F\text{ is a face of }G^\rho
\right\}
\]
is the set of distinct target face labels.
The subset $\cI$ consists of the boundary labels, and
$V_{\mathrm{in}}=\cS\setminus\cI$ indexes the mutable
summands.
Admissibility gives $\cS\subseteq\mathcal M_\pi$.
By the restriction criterion
\eqref{eq:fsb-boundary-restriction}, each $J\in\cS$
therefore determines a rank-one $B$-lattice $M_J$.
Its corresponding summand in $T_{\cI}$ is
\[
N_J=\Hom_B(P_{\cI},M_J)\in\cE_{\cI}.
\]

To compare these objects with the standard source
categorification, we use the mutually quasi-inverse
equivalences of Lemma~\ref{lem:fsb-derived-order}:
\[
\begin{aligned}
	F_{\cI}
	&=\mathbf R\Hom_B(P_{\cI},-):
	\Db(\fgmod B)
	\xrightarrow{\sim}
	\Db(\fgmod B_{\cI}),\\
	L_{\cI}
	&=P_{\cI}\otimes_{B_{\cI}}^{\mathbf L}-:
	\Db(\fgmod B_{\cI})
	\xrightarrow{\sim}
	\Db(\fgmod B).
\end{aligned}
\]
We transport them to the derived categories of
\[
\cE_{\cI}=\GP B_{\cI},
\qquad
\cE_{\mathrm s}=\GP B.
\]
Both algebras are Iwanaga--Gorenstein, so their
Gorenstein-projective subcategories are resolving
and every finitely generated module admits a finite
Gorenstein-projective resolution.
The inclusions consequently induce triangle equivalences
\[
\begin{aligned}
	j_{\cI}&:\Db(\cE_{\cI})
	\xrightarrow{\sim}\Db(\fgmod B_{\cI}),\\
	j_{\mathrm s}&:\Db(\cE_{\mathrm s})
	\xrightarrow{\sim}\Db(\fgmod B).
\end{aligned}
\]
Choosing quasi-inverses, we obtain the mutually
quasi-inverse triangle equivalences
\begin{equation}
	\label{eq:fsb-main-functor}
	\begin{aligned}
		\mathsf F_{\cI}
		&=
		j_{\mathrm s}^{-1}\circ L_{\cI}\circ j_{\cI}:
		\Db(\cE_{\cI})
		\xrightarrow{\sim}
		\Db(\cE_{\mathrm s}),\\
		\mathsf F_{\cI}^{-1}
		&=
		j_{\cI}^{-1}\circ F_{\cI}\circ j_{\mathrm s}:
		\Db(\cE_{\mathrm s})
		\xrightarrow{\sim}
		\Db(\cE_{\cI}).
	\end{aligned}
\end{equation}
Thus the comparison from $\cE_{\cI}$ to the standard
source category is induced by derived tensor product.

For $J\in\cS$,
Lemma~\ref{lem:necklace-orthogonal-images} gives
\[
F_{\cI}(M_J)\simeq N_J.
\]
Together with the counit
$L_{\cI}F_{\cI}\simeq\mathrm{Id}$, this yields
\begin{equation}
	\label{eq:fsb-main-face-image}
	\mathsf F_{\cI}(N_J)
	\simeq
	j_{\mathrm s}^{-1}\bigl(L_{\cI}F_{\cI}(M_J)\bigr)
	\simeq
	j_{\mathrm s}^{-1}(M_J)
	\qquad(J\in\cS).
\end{equation}
Here $M_J$ is viewed as a complex concentrated in
degree zero in $\Db(\fgmod B)$, and
$j_{\mathrm s}^{-1}(M_J)$ is its realization in
$\Db(\cE_{\mathrm s})$.

\begin{theorem}[Connected case]
	\label{prop:necklace-connected-comparison}
	Let $G^\rho$ be an admissibly relabeled reduced plabic
	graph for a connected loopless positroid $\pi$.
	Then the identity on $R_\pi$ is a quasi-cluster
	isomorphism from the target-labeled cluster structure
	of $G^\rho$ to the standard source cluster structure.
	Consequently, it also gives a quasi-cluster isomorphism
	to the standard target cluster structure.
\end{theorem}

\begin{proof}
	By the tilting property of $P_{\cI}$ and
	Lemma~\ref{lem:fsb-derived-order}, $\mathsf F_{\cI}$ induces a triangle
	equivalence
	\[
	\overline{\mathsf F}_{\cI}:
	\underline{\cE}_{\cI}
	\xrightarrow{\sim}
	\underline{\cE}_{\mathrm s}.
	\]
	
	By Proposition~\ref{prop:fsb-common-face-cluster}
	and \eqref{eq:fsb-main-face-image}, we may choose
	a basic source cluster-tilting object
	\[
	U
	=
	\bigoplus_{J\in V_{\mathrm{in}}}U_J
	\oplus
	\bigoplus_{p=1}^fP'_p
	\]
	and a mutation sequence $\boldsymbol k$ such that
	\[
	U\simeq\mu_{\boldsymbol k}T^{\mathrm s},
	\qquad
	\underline U_J
	\simeq
	\overline{\mathsf F}_{\cI}(\underline N_J)
	\quad(J\in V_{\mathrm{in}}).
	\]
	Here $P'_1,\ldots,P'_f$ are all the indecomposable
	projective-injective objects of $\cE_{\mathrm s}$.
	
	Let $\kappa_{\mathrm s}$ be the relative-class map for
	$\cE_{\mathrm s}=\GP B$, defined by
	\[
	\kappa_{\mathrm s}(Y)
	=\phi_{\mathrm s}^{-1}([Y]_{\mathrm{rel}})
	\in K_0^{\mathrm{sp}}(\cE_{\mathrm s}),
	\qquad Y\in\Db(\cE_{\mathrm s}),
	\]
	where $\phi_{\mathrm s}$ is the isomorphism of
	Proposition~\ref{prop:relative-split-grothendieck}
	applied to $\cE_{\mathrm s}$. In particular,
	\[
	\kappa_{\mathrm s}(H)=[H]
	\quad(H\in\cE_{\mathrm s}),\qquad
	\kappa_{\mathrm s}(P^\bullet)
	=\sum_i(-1)^i[P^i]
	\quad(P^\bullet\in K^b(\proj B)).
	\]
	By Lemma~\ref{lem:fsb-projective-correction}, write
	\[
	\begin{aligned}
		\kappa_{\mathrm s}(\mathsf F_{\cI}N_J)
		&=
		[U_J]+\sum_{p=1}^f a_{pJ}[P'_p]
		&& (J\in V_{\mathrm{in}}),\\
		\kappa_{\mathrm s}(\mathsf F_{\cI}N_I)
		&=
		\sum_{p=1}^f c_{pI}[P'_p]
		&& (I\in\cI).
	\end{aligned}
	\]
	Set
	\[
	A=(a_{pJ}),
	\qquad
	C_0=(c_{pI}),
	\qquad
	G=
	\begin{pmatrix}
		I_r&0\\
		A&C_0
	\end{pmatrix}.
	\]
	The matrix $C_0$ represents the isomorphism
	\[
	K_0\bigl(K^b(\proj B_{\cI})\bigr)
	\xrightarrow[\sim]{\,K_0(\mathsf F_{\cI})\,}
	K_0\bigl(K^b(\proj B)\bigr)
	\]
	in the bases $\{[N_I]:I\in\cI\}$ and
	$\{[P'_p]:1\leq p\leq f\}$.
	Hence
	\[
	C_0\in\operatorname{GL}_f(\mathbb Z),
	\qquad
	G\in\operatorname{GL}_{r+f}(\mathbb Z).
	\]
	
	By Propositions~\ref{prop:fsb-complete-ice-quiver}
	and~\ref{prop:fsb-frobenius-mutations}, the
	initial extended exchange matrix of $T_{\cI}$ is
	$\widetilde B_{G^\rho}$. The by
	Proposition~\ref{prop:fsb-relative-recognition}, we have
	\begin{equation}
		\label{eq:fsb-main-matrix}
		\begin{gathered}
			G\widetilde B_{G^\rho}=\widetilde B_U,\\
			I_U(\mathsf F_{\cI}X)
			=
			G I_{T_{\cI}}(X)
			\qquad(X\in\Db(\cE_{\cI})).
		\end{gathered}
	\end{equation}
	Write $x=(x_J)_{J\in\cS}$ and $u$ for the ordered
	extended clusters at $T_{\cI}$ and $U$, respectively.
	Since $G$ is unimodular, the Laurent substitution
	extends to a field isomorphism
	\[
	f_G:\mathbb F_{T_{\cI}}
	\xrightarrow{\sim}\mathbb F_U,
	\qquad
	x^v\longmapsto u^{Gv}.
	\]
	The same proposition gives
	\begin{equation}
		\label{eq:fsb-connected-character-naturality}
		f_G\bigl(\CC_{T_{\cI},\mathrm{loc}}(X)\bigr)
		=
		\CC_{U,\mathrm{loc}}(\mathsf F_{\cI}X)
		\qquad(X\in\Db(\cE_{\cI})).
	\end{equation}

	By admissibility and \cite[Theorem~4.21]{FSB},
	the seed $\Sigma_{G^\rho}^{\mathrm{tgt}}$ defines
	a cluster structure on $R_\pi$.
	Together with the preceding identification of the
	initial seed of $T_{\cI}$, this yields an isomorphism
	\[
	\eta_{\cI}:
	R_{\cI,\mathrm{loc}}\xrightarrow{\sim}R_\pi,
	\qquad
	x_J\longmapsto\Delta_J
	\quad(J\in\cS).
	\]
	Let
	$\eta_{\mathrm s}:
	R_{\mathrm s,\mathrm{loc}}\xrightarrow{\sim}R_\pi$
	be the standard source cluster-structure isomorphism.
	We use the same notation for their extensions to
	fraction fields.
	
	Let
	\[
	\mu_{\boldsymbol k}^x:
	\mathbb F_U\xrightarrow{\sim}\mathbb F_{T^{\mathrm s}}
	\]
	be the change of coordinates along the chosen mutation
	sequence, and put
	\[
	\theta
	=
	\eta_{\mathrm s}\circ\mu_{\boldsymbol k}^x\circ f_G:
	\mathbb F_{T_{\cI}}
	\xrightarrow{\sim}\operatorname{Frac}(R_\pi).
	\]
	For every $J\in\cS$, by
	\eqref{eq:fsb-connected-character-naturality},
	and Lemma~\ref{lem:fsb-seed-covariance}, we have
	\begin{equation}
		\label{eq:fsb-main-geometric-check}
		\begin{aligned}
			\theta(x_J)
			&=
			\eta_{\mathrm s}\!\left(
			\mu_{\boldsymbol k}^x
			\bigl(\CC_{U,\mathrm{loc}}
			(\mathsf F_{\cI}N_J)\bigr)
			\right)\\
			&=
			\Phi^{\mathrm s}
			\bigl(j_{\mathrm s}\mathsf F_{\cI}N_J\bigr)\\
			&=
			\Phi^{\mathrm s}(M_J)\\
			&=
			\Delta_J
			=
			\eta_{\cI}(x_J).
		\end{aligned}
	\end{equation}
	The third equality uses
	\eqref{eq:fsb-main-face-image}, and the fourth uses
	Theorem~\ref{thm:fsb-rank-one-normalization}.
	This calculation includes all boundary labels.
	
	Since the coordinates $x_J$, $J\in\cS$, generate
	$\mathbb F_{T_{\cI}}$, we obtain
	\[
	\theta=\eta_{\cI}.
	\]
	It follows that the ring isomorphism
	$\psi
	:=
	\eta_{\mathrm s}^{-1}\circ\eta_{\cI}:
	R_{\cI,\mathrm{loc}}
	\xrightarrow{\sim}
	R_{\mathrm s,\mathrm{loc}}$
	is the restriction of
	\[
	\mu_{\boldsymbol k}^x\circ f_G.
	\]
	Thus $\psi$ satisfies the initial-coordinate condition
	of Proposition~\ref{prop:fsb-relative-recognition}.
	Together with
	\[
	U\simeq\mu_{\boldsymbol k}T^{\mathrm s},
	\qquad
	C_0\in\operatorname{GL}_f(\mathbb Z),
	\qquad
	G\widetilde B_{G^\rho}=\widetilde B_U,
	\]
	it shows that $\psi$ is a quasi-cluster
	isomorphism.
	We now identify both abstract cluster algebras with
	$R_\pi$ via
	\[
	\eta_{\cI}:R_{\cI,\mathrm{loc}}\xrightarrow{\sim}R_\pi,
	\qquad
	\eta_{\mathrm s}:R_{\mathrm s,\mathrm{loc}}
	\xrightarrow{\sim}R_\pi.
	\]
	Under these identifications, the quasi-cluster
	isomorphism $\psi$ becomes the map
	\[
	R_\pi
	\xrightarrow{\eta_{\cI}^{-1}}
	R_{\cI,\mathrm{loc}}
	\xrightarrow{\psi}
	R_{\mathrm s,\mathrm{loc}}
	\xrightarrow{\eta_{\mathrm s}}
	R_\pi.
	\]
	Since $\psi=\eta_{\mathrm s}^{-1}\circ\eta_{\cI}$,
	for every $h\in R_\pi$ we have
	\[
	\eta_{\mathrm s}\bigl(
	\psi(\eta_{\cI}^{-1}(h))\bigr)
	=
	\eta_{\mathrm s}\bigl(
	\eta_{\mathrm s}^{-1}(
	\eta_{\cI}(\eta_{\cI}^{-1}(h)))\bigr)
	=h.
	\]
	Thus
	\[
	\eta_{\mathrm s}\circ\psi\circ\eta_{\cI}^{-1}
	=\mathrm{id}_{R_\pi}.
	\]
	Consequently, the identity on $R_\pi$ is a
	quasi-cluster isomorphism from the relabeled
	target cluster structure to the standard
	source cluster structure.
	
	Finally, the identity on $R_\pi$ is also a
	quasi-cluster isomorphism between the standard source
	and target structures by
	\cite[Theorem~6.16]{Pressland}.
	Composing the two comparisons proves the assertion
	for the standard target structure.
\end{proof}

\subsection{The general case}
\label{sec:fsb-disconnected}

We extend the connected result to arbitrary loopless positroids
by combining the comparisons for their connected components
at the level of coordinate rings.

For the combinatorial toggle paths, we use the definition of
aligned chords in \cite[Definition~4.5]{FSB}, including chords
corresponding to fixed points. Explicitly, for distinct
$a,c,d\in[n]$, the loop chord $a\to a$ and the chord $c\to d$
are aligned if and only if
\[
a<_a c<_a d,
\]
where $<_a$ denotes the cyclic order beginning at $a$.
Two distinct loop chords are not aligned.

This convention extends the combinatorial notion of an aligned
toggle to the disconnected setting. The short exact sequence
\eqref{eq:fsb-toggle-sequence} is used only for nontrivial aligned
toggles in the connected case, where the four endpoints are distinct.

\begin{definition}[Components and cyclic restrictions]
	\label{def:necklace-components}
	A matroid is \emph{connected} if it cannot be written
	as a direct sum on two nonempty disjoint ground sets.
	Write the unique decomposition of $\mathcal M_\pi$
	into connected components as
	\[
	\mathcal M_\pi
	=\bigoplus_{j=1}^c\mathcal M_j,
	\qquad
	\mathcal M_j=\mathcal M_\pi|_{S_j},
	\qquad
	[n]=S_1\sqcup\cdots\sqcup S_c,
	\]
	and put
	\[
	n_j=|S_j|,
	\qquad
	k_j=\operatorname{rank}\mathcal M_j.
	\]
	Each $S_j$ carries the cyclic order inherited from $[n]$.
\end{definition}
	By \cite[Theorem~7.6 and Corollary~7.9]{ARW},
	the component partition is noncrossing: no two distinct
	blocks contain alternating pairs of cyclically ordered
	elements. Moreover, it is the finest noncrossing partition
	whose blocks are individually preserved by $\pi$.
	We denote it by
	\[
	\operatorname{NCInv}(\pi)=\{S_1,\ldots,S_c\}.
	\]

	For a nonempty subset $S\subseteq[n]$, a
	\emph{cyclic enumeration} is a cyclic-order-preserving
	bijection
	\[
	\beta:[|S|]=\{1,2,\ldots,|S|\}\xrightarrow{\sim}S.
	\]
	If $\sigma(S)=S$, its \emph{standardized restriction}
	with respect to $\beta$ is
	\[
	\beta^{-1}\circ(\sigma|_S)\circ\beta,
	\]
	with fixed-point decorations transported along $\beta$
	when $\sigma$ is decorated.
	Changing the first element of the enumeration conjugates
	this restriction by a cyclic shift. Thus the standardized
	positroid changes only by cyclic relabeling; the positroid
	on $S$ is unchanged.

Let $\mathbf J=(J_1,\ldots,J_n)$ have removal permutation
$\rho$ and insertion permutation $\pi\rho$.
For a nonempty $\pi$-invariant subset $S\subseteq[n]$,
put $m=|S|$ and list the positions
\[
\{a\in[n]:\rho(a)\in S\}
=\{a_1,\ldots,a_m\}
\]
in cyclic order. Define the \emph{projection to $S$} by
\[
\operatorname{pr}_S\mathbf J=(K_1,\ldots,K_m),
\qquad K_\ell=J_{a_\ell}\cap S.
\]
Thus $K_\ell$ is the intersection with $S$ immediately
before the removal at position $a_\ell$.

Since $\pi(S)=S$, every omitted step leaves the intersection
with $S$ unchanged. Consequently,
\[
K_{\ell+1}
=
\bigl(K_\ell\setminus\{\rho(a_\ell)\}\bigr)
\cup\{\pi\rho(a_\ell)\},
\qquad K_{m+1}=K_1.
\]
The projected list retains all repetitions and is defined
up to cyclic rotation.

If $\mathbf J$ is a unit necklace for $\pi$ and $S=S_j$
is a component, then each $J_{a_\ell}$ is a basis of
$\mathcal M_\pi$. Hence $K_\ell$ is a basis of
$\mathcal M_\pi|_{S_j}$, so
\[
|K_\ell|=k_j
\qquad(1\leq\ell\leq n_j).
\]

\begin{example}[Projection to connected components]
	\label{ex:necklace-component-projection}
	We abbreviate subsets by concatenating their elements.
	Consider
	\[
	\pi=(1\,3)(2\,4)(5\,6).
	\]
	The corresponding positroid has bases
	\[
	\mathcal M_\pi
	=
	\left\{
	A\cup\{b\}:
	A\in\binom{\{1,2,3,4\}}{2},\
	b\in\{5,6\}
	\right\}.
	\]
	Its connected components have ground sets and ranks
	\[
	S_1=\{1,2,3,4\},\qquad k_1=2,
	\qquad
	S_2=\{5,6\},\qquad k_2=1.
	\]
	The ordinary forward Grassmann necklace is
	\[
	\mathbf I^\pi=(125,235,345,145,125,126).
	\]
	
	Now consider the Grassmannlike necklace
	\[
	\mathbf J=(125,235,125,126,125,145)
	\]
	with removal and insertion data
	\[
	\begin{array}{c|cccccc}
		a             &1&2&3&4&5&6\\ \hline
		\rho(a)       &1&3&5&6&2&4\\
		\pi\rho(a)    &3&1&6&5&4&2\\
		J_a           &125&235&125&126&125&145
	\end{array}
	\]
	Indeed, these satisfy
	\[
	J_{a+1}
	=
	\bigl(J_a\setminus\{\rho(a)\}\bigr)
	\cup\{\pi\rho(a)\},
	\qquad J_7=J_1.
	\]
	Every term of $\mathbf J$ occurs in $\mathbf I^\pi$.
	Hence each $\Delta_{J_a}$ is an ordinary target frozen
	variable, so $\mathbf J$ is a unit necklace for $\pi$.
	
	For $S_1$, the retained positions are $1,2,5,6$,
	since these are precisely the positions with
	$\rho(a)\in S_1$. Therefore
	\[
	\begin{aligned}
		\operatorname{pr}_{S_1}\mathbf J
		&=(J_1\cap S_1,J_2\cap S_1,
		J_5\cap S_1,J_6\cap S_1)\\
		&=(12,23,12,14).
	\end{aligned}
	\]
	Its successive removals and insertions are
	\[
	12\xrightarrow{1\mapsto3}23
	\xrightarrow{3\mapsto1}12
	\xrightarrow{2\mapsto4}14
	\xrightarrow{4\mapsto2}12.
	\]
	In particular, the two occurrences of $12$ are retained:
	they precede different removal steps.
	
	For $S_2$, the retained positions are $3,4$, giving
	\[
	\operatorname{pr}_{S_2}\mathbf J
	=(J_3\cap S_2,J_4\cap S_2)
	=(\{5\},\{6\}),
	\]
	with transitions
	\[
	\{5\}\xrightarrow{5\mapsto6}\{6\}
	\xrightarrow{6\mapsto5}\{5\}.
	\]
	Thus the two projections have $|S_1|=4$ and
	$|S_2|=2$ positions, and their terms have sizes
	$k_1=2$ and $k_2=1$, respectively.
\end{example}

\begin{lemma}[Boundary count]
	\label{lem:fsb-boundary-count}
	Let \(\mathcal M_\pi\) be a loopless positroid on
	the ground set \([N]=\{1,\ldots,N\}\), where \(N\geq1\).
	Let \(d\) be its number of connected components, and let
	\[
	\mathbf I^\pi=(I_1^\pi,\ldots,I_N^\pi)
	\]
	be its forward Grassmann necklace. Then
	\[
	\bigl|\{I_a^\pi:a\in[N]\}\bigr|=N-d+1.
	\]
	The Pl\"ucker coordinates corresponding to these distinct
	labels freely generate the frozen monomial group
	\(\mathscr F_\pi\). In particular,
	\[
	\mathscr F_\pi\simeq\mathbb Z^{N-d+1}.
	\]
\end{lemma}

\begin{proof}
	For a positroid \(\mathcal M\), denote by \(b(\mathcal M)\)
	the number of distinct terms in its forward Grassmann
	necklace. We prove the counting formula by induction
	on the number of connected components.
	
	First suppose that \(\mathcal M_\pi\) is connected.
	Let \(e_1,\ldots,e_N\) be the standard basis vectors of
	\(\mathbb Z^N\), and write
	\[
	\mathbf 1_J=\sum_{j\in J}e_j
	\qquad(J\subseteq[N]).
	\]
	The necklace recurrence gives
	\[
	\mathbf 1_{I_{a+1}^\pi}
	-\mathbf 1_{I_a^\pi}
	=e_{\pi(a)}-e_a.
	\]
	Suppose that \(I_a^\pi=I_b^\pi\) for some
	\(1\leq a<b\leq N\), and put
	\[
	E=\{a,a+1,\ldots,b-1\}.
	\]
	Summing the recurrence yields
	\[
	\begin{aligned}
		0
		&=\mathbf 1_{I_b^\pi}-\mathbf 1_{I_a^\pi}\\
		&=\sum_{i\in E}(e_{\pi(i)}-e_i)\\
		&=\mathbf 1_{\pi(E)}-\mathbf 1_E.
	\end{aligned}
	\]
	Hence \(\pi(E)=E\).
	The sets \(E\) and \([N]\setminus E\) are nonempty
	cyclic intervals, so they form a nontrivial
	\(\pi\)-invariant noncrossing partition.
	By \cite[Corollary~7.9]{ARW}, the component partition
	refines this partition, contradicting connectedness.
	Thus all necklace terms are distinct, and
	\[
	b(\mathcal M_\pi)=N.
	\]
	
	Now suppose that \(d>1\).
	By \cite[Theorem~7.6]{ARW}, the component ground sets
	form a noncrossing partition.
	Every noncrossing partition has a block that is a
	cyclic interval. After a cyclic relabeling, we may
	therefore choose a component with ground set
	\[
	S=\{1,\ldots,m\},
	\qquad
	T=\{m+1,\ldots,N\}.
	\]
	Write
	\[
	\mathcal M_\pi
	=\mathcal M_S\oplus\mathcal M_T,
	\qquad
	\mathcal M_S=\mathcal M_\pi|_S,
	\qquad
	\mathcal M_T=\mathcal M_\pi|_T.
	\]
	Here \(\mathcal M_S\) is connected, while
	\(\mathcal M_T\) has \(d-1\) connected components.
	
	Let
	\[
	(J_1,\ldots,J_m),
	\qquad
	(K_{m+1},\ldots,K_N)
	\]
	be the forward Grassmann necklaces of these two
	restrictions, with their inherited cyclic orders.
	A lexicographically minimal basis of a direct sum
	is the union of the corresponding minimal bases of
	its summands. Consequently,
	\[
	I_a^\pi=
	\begin{cases}
		J_a\cup K_{m+1},&1\leq a\leq m,\\
		J_1\cup K_a,&m+1\leq a\leq N.
	\end{cases}
	\]
	Put
	\[
	\begin{aligned}
		\mathcal B_S
		&=\{J_a\cup K_{m+1}:1\leq a\leq m\},\\
		\mathcal B_T
		&=\{J_1\cup K_a:m+1\leq a\leq N\}.
	\end{aligned}
	\]
	Since \(S\) and \(T\) are disjoint, we have
	\[
	|\mathcal B_S|=b(\mathcal M_S),
	\qquad
	|\mathcal B_T|=b(\mathcal M_T).
	\]
	Moreover,
$	\mathcal B_S\cap\mathcal B_T
	=\{J_1\cup K_{m+1}\}.$ Indeed, equality between labels from the two families
	forces equality of their intersections with both
	\(S\) and \(T\).
	
	The connected case and the induction hypothesis now give
	\[
	\begin{aligned}
		b(\mathcal M_\pi)
		&=|\mathcal B_S\cup\mathcal B_T|\\
		&=b(\mathcal M_S)+b(\mathcal M_T)-1\\
		&=m+\bigl((N-m)-(d-1)+1\bigr)-1\\
		&=N-d+1.
	\end{aligned}
	\]
	The same argument applies when the chosen component
	consists of a single element.
	
	Finally, let
	$\mathcal B_\pi=\{I_a^\pi:a\in[N]\}.$
	The coordinates \(\Delta_J\), for \(J\in\mathcal B_\pi\),
	are the frozen variables of an ordinary target seed.
	They are therefore algebraically independent;
	see \cite[Definition~2.13 and Theorem~4.21]{FSB}.
	If
	\[
	\prod_{J\in\mathcal B_\pi}\Delta_J^{a_J}=1,
	\qquad a_J\in\mathbb Z,
	\]
	then clearing negative exponents gives
	$\prod_{\substack{J\in\mathcal B_\pi\\a_J>0}}
	\Delta_J^{a_J}
	=
	\prod_{\substack{J\in\mathcal B_\pi\\a_J<0}}
	\Delta_J^{-a_J}.$
	
	By algebraic independence, we have \(a_J=0\) for every
	\(J\in\mathcal B_\pi\). Thus these coordinates freely
	generate \(\mathscr F_\pi\), and
	\[
	\mathscr F_\pi
	\simeq\mathbb Z^{|\mathcal B_\pi|}
	=\mathbb Z^{N-d+1}.
	\]
\end{proof}

For the component positroid on \(S_j\), let
\(\mathcal B_j\) be the set of distinct labels in its
forward Grassmann necklace, and put
$R_j=\kk[\widehat\Pi_j^\circ].$

The usual Pl\"ucker grading extends to this ring by
assigning degree \(1\) to every Pl\"ucker coordinate
and degree \(-1\) to each inverted boundary coordinate.
We write
$R_j^0=(R_j)_0$
for its degree-zero subring. Define the local coefficient group and its degree-zero
subgroup by
\[
\begin{aligned}
	\mathbb P_j
	&=
	\left\{
	\prod_{K\in\mathcal B_j}\Delta_K^{a_K}:
	a_K\in\mathbb Z
	\right\}
	\subseteq R_j^\times,\\
	\mathbb P_j^0
	&=\ker\bigl(\deg:\mathbb P_j\longrightarrow\mathbb Z\bigr).
\end{aligned}
\]
In particular, we have
$\deg\left(\prod_{K\in\mathcal B_j}\Delta_K^{a_K}\right)
=\sum_{K\in\mathcal B_j}a_K.$

Choose an ordinary global target frozen coordinate
\(p=\Delta_L\), and set
\[
p_j=\Delta_{L\cap S_j}
\qquad(1\leq j\leq c).
\]
The restriction \(L\cap S_j\) is a forward necklace
label of the component positroid. Hence
\[
p_j\in\mathbb P_j\subseteq R_j^\times,
\qquad \deg p_j=1.
\]

\begin{definition}[Diagonal graded product]
	\label{def:necklace-segre}
	Let \(R_1,\ldots,R_c\) be the graded
	\(\kk\)-algebras defined above.
	Their \emph{diagonal graded product} is
	\[
	R_1\mathbin{\#}\cdots\mathbin{\#}R_c
	=
	\bigoplus_{d\in\mathbb Z}
	\left(
	(R_1)_d\otimes_\kk\cdots\otimes_\kk(R_c)_d
	\right)
	\subseteq
	R_1\otimes_\kk\cdots\otimes_\kk R_c,
	\]
	with multiplication induced by the tensor product algebra.
	Its degree-\(d\) component consists of sums of tensors
	\(f_1\otimes\cdots\otimes f_c\) with
	\(\deg f_j=d\) for every \(j\).
	Thus the grading records the common degree of the factors.
	In the present setting, this is the
	\emph{localized homogeneous Segre product}.
\end{definition}

\begin{lemma}
	\label{lem:fsb-component-coefficients}
	Let
	$\mathcal M_\pi
	=\mathcal M_1\oplus\cdots\oplus\mathcal M_c$
	be the connected-component decomposition of a loopless
	positroid on \([n]=\{1,\ldots,n\}\).
	Write \(S_j\) for the ground set of \(\mathcal M_j\), and put
	\[
	n_j=|S_j|,
	\qquad
	k_j=\operatorname{rank}\mathcal M_j,
	\qquad
	R_j=\kk[\widehat\Pi_j^\circ].
	\]
	Use the linear orders on the sets \(S_j\) inherited from
	\(1<\cdots<n\). Choose an ordinary global target frozen coordinate
	\(p=\Delta_L\), and put
	\[
	p_j=\Delta_{L\cap S_j}.
	\]
	Each \(p_j\) is an ordinary local frozen coordinate
	and is an invertible element of degree one.
	Then the following statements hold.
	
	\begin{enumerate}
		\item
		There exist fixed signs
		\(\varepsilon_1,\ldots,\varepsilon_n\in\{1,-1\}\),
		depending only on the ordered component decomposition
		and the ranks \(k_j\), with the following property.
		\begin{itemize}
			\item For matrices \(M_j\) representing points of the component
			positroid cones, let \(\widetilde M_j\) be obtained from
			\(M_j\) by inserting zero columns outside \(S_j\), and set
			\[
			M=
			\begin{pmatrix}
				\widetilde M_1\\
				\vdots\\
				\widetilde M_c
			\end{pmatrix},
			\qquad
			D=\operatorname{diag}(\varepsilon_1,\ldots,\varepsilon_n).
			\]
			Then
			\[
			\Delta_J(MD)
			=
			\prod_{j=1}^c\Delta_{J\cap S_j}(M_j)
			\qquad(J\in\mathcal M_\pi).
			\]
			Thus the fixed column changes \(M\mapsto MD\)
			remove the determinant signs in the direct-sum formula
			simultaneously for all \(J\).
			\item  The resulting pullback induces a graded isomorphism
			\[
			\partial:
			R_\pi\xrightarrow{\sim}
			R_1\mathbin{\#}\cdots\mathbin{\#}R_c
			\subseteq
			R_1\otimes_\kk\cdots\otimes_\kk R_c,
			\]
			given by
			\[
			\partial(\Delta_J)
			=
			\bigotimes_{j=1}^c\Delta_{J\cap S_j}
			\qquad(J\in\mathcal M_\pi).
			\]
			In particular,
			\[
			\partial(p)=p_1\otimes\cdots\otimes p_c.
			\]
		\end{itemize} 
		
		\item
		For \(f\in R_j^0=(R_j)_0\), the tensor
		\(1\otimes\cdots\otimes f\otimes\cdots\otimes1\)
		has degree zero in every factor and therefore belongs
		to the diagonal graded product.
		Hence there is an embedding
		\[
		\iota_j:R_j^0\lhook\joinrel\longrightarrow R_\pi,
		\qquad
		\iota_j(f)
		=
		\partial^{-1}
		(1\otimes\cdots\otimes f\otimes\cdots\otimes1).
		\]
		These embeddings identify the global degree-zero subring
		and the full homogeneous ring as
		\[
		(R_\pi)_0
		\simeq
		\bigotimes_{j=1}^cR_j^0,
		\]
		\begin{equation}
			\label{eq:component-ring-normalisation}
			R_\pi
			\simeq
			\left(\bigotimes_{j=1}^cR_j^0\right)[p^{\pm1}].
		\end{equation}
		Here a tensor \(f_1\otimes\cdots\otimes f_c\)
		is sent to \(\prod_j\iota_j(f_j)\), and the Laurent
		generator is sent to the chosen coordinate \(p\).
		
		\item
		Let \(\mathbb P_j\) be the local frozen monomial group,
		and let
		\[
		\mathbb P_j^0
		=
		\ker\bigl(\deg:\mathbb P_j\longrightarrow\mathbb Z\bigr).
		\]
		Then \(\iota_j(\mathbb P_j^0)\subseteq\mathbb P_\pi\),
		where \(\mathbb P_\pi=\mathscr F_\pi\).
		Moreover, multiplication induces an isomorphism
		\[
		\begin{aligned}
			\left(\bigoplus_{j=1}^c\mathbb P_j^0\right)
			\oplus\mathbb Z[p]
			&\xrightarrow{\sim}\mathbb P_\pi,\\
			((f_j)_{j=1}^c,d[p])
			&\longmapsto
			p^d\prod_{j=1}^c\iota_j(f_j).
		\end{aligned}
		\]
		Here \(\mathbb Z[p]\) is the infinite cyclic exponent
		lattice generated by \(p\).
		Equivalently, every global frozen Laurent monomial has
		a unique expression
		\[
		p^d\prod_{j=1}^c\iota_j(f_j),
		\qquad d\in\mathbb Z,\quad f_j\in\mathbb P_j^0.
		\]
		Consequently,
		\begin{equation}
			\label{eq:fsb-component-frozen-lattice}
			\mathbb P_\pi
			\simeq
			\left(\bigoplus_{j=1}^c\mathbb P_j^0\right)
			\oplus\mathbb Z[p],
			\qquad
			\operatorname{rank}\mathbb P_\pi
			=\sum_{j=1}^c(n_j-1)+1
			=n-c+1.
		\end{equation}
	\end{enumerate}
\end{lemma}

\begin{proof}
	We first construct the fixed column signs.
	For \(J\in\mathcal M_\pi\), put
	\[
	J_j=J\cap S_j.
	\]
	Since \(\mathcal M_\pi=\bigoplus_j\mathcal M_j\),
	we have \(|J_j|=k_j\).
	For \(a<b\), define
	\[
	\nu_{ab}(J)
	=
	\bigl|\{(u,v)\in J_a\times J_b:u>v\}\bigr|,
	\qquad
	\sigma(J)=(-1)^{\sum_{a<b}\nu_{ab}(J)}.
	\]
	Reordering the columns of \(M_J\) by their component
	blocks gives
	\[
	\Delta_J(M)
	=
	\sigma(J)\prod_{j=1}^c\Delta_{J_j}(M_j).
	\]
	
	The component partition is noncrossing.
	Consequently, after restricting the linear order to
	\(S_a\cup S_b\), the two blocks occur in at most three
	consecutive groups. Otherwise, four alternating groups
	would give four cyclically alternating elements from
	the two blocks.
	
	Thus, for each \(a<b\), we may choose one of the forms
	\[
	\begin{aligned}
		&S_a=S_a^-\sqcup S_a^+,
		&&S_a^-<S_b<S_a^+,\\
		\text{or}\qquad
		&S_b=S_b^-\sqcup S_b^+,
		&&S_b^-<S_a<S_b^+.
	\end{aligned}
	\]
	Here \(A<B\) means that every element of \(A\) precedes
	every element of \(B\), and the end pieces may be empty.
	In the first case,
	\[
	\nu_{ab}(J)=k_b|J\cap S_a^+|,
	\]
	while in the second case,
	\[
	\nu_{ab}(J)=k_a|J\cap S_b^-|.
	\]
	Define
	\[
	c_{ab,i}=
	\begin{cases}
		k_b\mathbf 1_{S_a^+}(i),&\text{in the first case},\\
		k_a\mathbf 1_{S_b^-}(i),&\text{in the second case},
	\end{cases}
	\]
	and put
	\[
	c_i=\sum_{a<b}c_{ab,i},
	\qquad
	\varepsilon_i=(-1)^{c_i},
	\qquad
	D=\operatorname{diag}(\varepsilon_1,\ldots,\varepsilon_n).
	\]
	These choices depend only on the ordered component
	decomposition and its ranks. Moreover,
	\[
	\sigma(J)
	=
	(-1)^{\sum_{i\in J}c_i}
	=
	\prod_{i\in J}\varepsilon_i.
	\]
	Therefore
	\[
	\begin{aligned}
		\Delta_J(MD)
		&=
		\left(\prod_{i\in J}\varepsilon_i\right)\Delta_J(M)\\
		&=
		\sigma(J)^2
		\prod_{j=1}^c\Delta_{J_j}(M_j)\\
		&=
		\prod_{j=1}^c\Delta_{J_j}(M_j).
	\end{aligned}
	\]
	This proves the asserted simultaneous sign correction.
	
	We next establish the ring comparison, including
	the localization.
	Let \(A_\pi\) and \(A_j\) denote the homogeneous
	coordinate rings of the closed positroid cones,
	before inverting their boundary coordinates.
	The direct-sum description and its compatibility with
	the Segre embedding give
	\[
	A_\pi
	\simeq
	\bigoplus_{d\geq0}
	\bigotimes_{j=1}^c(A_j)_d;
	\]
	see \cite[Propositions~4.2 and~4.4]{Pressland}.
	Column sign changes preserve the vanishing and
	nonvanishing of Pl\"ucker coordinates.
	With the choice of \(D\) above, this isomorphism sends
	\[
	\Delta_J\longmapsto
	\bigotimes_{j=1}^c\Delta_{J\cap S_j}.
	\]
	For a singleton component, \(A_j=\kk[p_j]\), and
	the same assertion follows directly.
	
	Every global boundary coordinate maps to a tensor
	product of local boundary coordinates.
	These are units in the rings \(R_j\).
	Consequently, localization gives an injective map
	\[
	\partial:
	R_\pi\lhook\joinrel\longrightarrow
	R_1\mathbin{\#}\cdots\mathbin{\#}R_c.
	\]
	We show that this map is surjective.
	
	For \(f\in R_j\), write
	\[
	\epsilon_j(f)
	=
	1\otimes\cdots\otimes f\otimes\cdots\otimes1
	\in\bigotimes_{h=1}^cR_h.
	\]
	If \(f\) has degree zero, this tensor belongs to the
	diagonal graded product.
	
	Let \(Q_b=\Delta_{I_b^\pi}\) be the ordinary global
	necklace coordinates, with \(Q_{n+1}=Q_1\).
	List the local necklace coordinates in cyclic order as
	\[
	q_{j,1},\ldots,q_{j,n_j},
	\qquad q_{j,n_j+1}=q_{j,1}.
	\]
	At a global removal \(b\in S_j\), only the \(S_j\)-part
	of the necklace changes. If this is the \(a\)-th local
	step, then
	\[
	\partial\left(\frac{Q_{b+1}}{Q_b}\right)
	=
	\epsilon_j\left(\frac{q_{j,a+1}}{q_{j,a}}\right).
	\]
	Both the displayed ratio and its inverse belong to
	\(\partial(R_\pi)\).
	Since \(p_j\) is one of the local necklace coordinates,
	multiplying consecutive ratios gives
	\[
	\epsilon_j(q_{j,a}/p_j)^{\pm1}
	\in\partial(R_\pi)
	\qquad(1\leq a\leq n_j).
	\]
	
	For any basis \(K\in\mathcal M_j\), put
	\[
	J(K)
	=
	K\cup\bigcup_{h\neq j}(L\cap S_h).
	\]
	The direct-sum decomposition implies
	\(J(K)\in\mathcal M_\pi\), and hence
	\[
	\partial\left(\frac{\Delta_{J(K)}}p\right)
	=
	\epsilon_j\left(\frac{\Delta_K}{p_j}\right).
	\]
	Now the degree-zero subring is generated by
	\[
	R_j^0
	=
	\kk\left[
	\frac{\Delta_K}{p_j},
	\frac{p_j}{q_{j,a}}:
	K\in\mathcal M_j,\ 1\leq a\leq n_j
	\right].
	\]
	Indeed, the numerator and denominator of a homogeneous
	monomial of degree zero have the same total degree,
	so inserting the corresponding powers of \(p_j\)
	expresses it in these generators.
	We have therefore proved
	\[
	\epsilon_j(R_j^0)\subseteq\partial(R_\pi)
	\qquad(1\leq j\leq c).
	\]
	
	Since \(p_j\) is a degree-one unit, we have
	\[
	(R_j)_d=R_j^0p_j^d
	\qquad(d\in\mathbb Z).
	\]
	Writing \(p_\otimes=p_1\otimes\cdots\otimes p_c\),
	we obtain
	\[
	\begin{aligned}
		R_1\mathbin{\#}\cdots\mathbin{\#}R_c
		&=
		\bigoplus_{d\in\mathbb Z}
		\left(\bigotimes_{j=1}^cR_j^0\right)p_\otimes^d\\
		&\simeq
		\left(\bigotimes_{j=1}^cR_j^0\right)
		[p_\otimes^{\pm1}].
	\end{aligned}
	\]
	The image of \(\partial\) contains every
	\(\epsilon_j(R_j^0)\) and contains
	\(p_\otimes^{\pm1}=\partial(p^{\pm1})\).
	It therefore equals the diagonal graded product.
	
	In particular,
	$\iota_j=\partial^{-1}\circ\epsilon_j|_{R_j^0}$
	is well defined. Taking degree zero gives
	$(R_\pi)_0\simeq\bigotimes_{j=1}^cR_j^0,$
	and the preceding Laurent decomposition proves
	\eqref{eq:component-ring-normalisation}.
	
	It remains to identify the frozen monomial group.
	By Lemma~\ref{lem:fsb-boundary-count}, the local
	coordinates \(q_{j,1},\ldots,q_{j,n_j}\) freely
	generate \(\mathbb P_j\). Since they all have degree one,
	\[
	\mathbb P_j^0
	=
	\left\langle
	\frac{q_{j,a+1}}{q_{j,a}}:
	1\leq a\leq n_j
	\right\rangle_{\mathbb Z}.
	\]
	The comparison of consecutive necklace ratios above gives
	$\iota_j(\mathbb P_j^0)\subseteq\mathbb P_\pi.$
	Conversely, each global frozen coordinate satisfies
	\[
	Q_b
	=
	p\prod_{j=1}^c
	\iota_j\left(
	\frac{\Delta_{I_b^\pi\cap S_j}}{p_j}
	\right),
	\]
	and every quotient on the right belongs to
	\(\mathbb P_j^0\).
	Thus the homomorphism
	\[
	((f_j)_{j=1}^c,d[p])
	\longmapsto
	p^d\prod_{j=1}^c\iota_j(f_j)
	\]
	is surjective onto \(\mathbb P_\pi\).
	
	For injectivity, suppose that
	\[
	p^d\prod_{j=1}^c\iota_j(f_j)=1,
	\qquad f_j\in\mathbb P_j^0.
	\]
	Taking homogeneous degrees gives \(d=0\).
	Applying \(\partial\) then gives
$	f_1\otimes\cdots\otimes f_c=1.$
	The local frozen coordinates are algebraically independent.
	Tensoring their Laurent polynomial subalgebras over
	\(\kk\) gives an injection
	\[
	\bigotimes_{j=1}^c
	\kk[q_{j,1}^{\pm1},\ldots,q_{j,n_j}^{\pm1}]
	\lhook\joinrel\longrightarrow
	\bigotimes_{j=1}^cR_j.
	\]
	The source is a Laurent polynomial algebra in all the
	local frozen coordinates.
	The displayed monomial equality therefore forces
	every exponent to vanish, so \(f_j=1\) for every \(j\).
	This proves the asserted uniqueness and the group
	isomorphism in \eqref{eq:fsb-component-frozen-lattice}.
	
	Finally, \(\deg:\mathbb P_j\to\mathbb Z\) is surjective
	because \(\deg p_j=1\). Hence
$	\operatorname{rank}\mathbb P_j^0=n_j-1,$
	and consequently
	\[
	\operatorname{rank}\mathbb P_\pi
	=
	\sum_{j=1}^c(n_j-1)+1
	=
	n-c+1.
	\]
\end{proof}

\begin{lemma}[Components and local admissibility]
	\label{lem:fsb-local-components}
	Let \(G^\rho\) be admissible, and set
	\[
	\mu=\rho^{-1}\pi\rho,
	\qquad
	A_j=\rho^{-1}(S_j)
	\quad(1\leq j\leq c).
	\]
	Then \(A_1,\ldots,A_c\) are precisely the ground sets
	of the connected components of \(\mathcal M_\mu\).
	
	Equip \(A_j\) and \(S_j\) with their inherited cyclic orders.
	There exist cyclic order-preserving bijections
	\[
	\alpha_j:[n_j]\xrightarrow{\sim}A_j,
	\qquad
	\beta_j:[n_j]\xrightarrow{\sim}S_j
	\]
	such that the induced permutations
	\[
	\mu_j=\alpha_j^{-1}\mu\alpha_j,
	\qquad
	\pi_j=\beta_j^{-1}\pi\beta_j,
	\qquad
	\rho_j=\beta_j^{-1}\rho\alpha_j
	\]
	satisfy
	$\pi_j=\rho_j\mu_j\rho_j^{-1},$
	and \(\rho_j\) is admissible for \(\pi_j\) whenever
	\(n_j>1\).
	If \(n_j=1\), all three permutations are the identity,
	and the corresponding component is a singleton coloop.
	The choice of \(\alpha_j\) and \(\beta_j\) amounts only
	to choosing cyclic origins.
\end{lemma}

\begin{proof}
	Let \(e_1,\ldots,e_n\) be the standard basis of \(\mathbb Q^n\),
	and write \(\mathbf 1_B=\sum_{i\in B}e_i\) for \(B\subseteq[n]\).
	For a matroid \(\mathcal M\) on \([n]\), put
	\[
	V(\mathcal M)
	=
	\operatorname{span}_{\mathbb Q}
	\{\mathbf 1_B-\mathbf 1_{B'}:B,B'\in\mathcal M\}.
	\]
	Let
	\[
	[n]=C_1\sqcup\cdots\sqcup C_d,
	\qquad
	\mathcal M=\bigoplus_{h=1}^d\mathcal M|_{C_h},
	\]
	be the connected-component decomposition of \(\mathcal M\).
	Then we have
	\[
	V(\mathcal M)
	=
	\left\{
	v=(v_1,\ldots,v_n)\in\mathbb Q^n:
	\sum_{a\in C_h}v_a=0
	\text{ for }1\leq h\leq d
	\right\}.
	\]
	In particular, for \(a\neq b\),
	\[
	a,b\text{ belong to the same component}
	\quad\Longleftrightarrow\quad
	e_a-e_b\in V(\mathcal M).
	\]
	
	Let \(\mathcal C\) be the set of distinct labels of a
	Pl\"ucker seed defining a cluster structure on \(R_{\mathcal M}\).
	For \(B\in\mathcal M\), the Laurent phenomenon gives
	\[
	\Delta_B
	=
	\sum_{\mathbf a}c_{\mathbf a}
	\prod_{K\in\mathcal C}\Delta_K^{a_K}.
	\]
	Using
	\[
	\operatorname{wt}(\Delta_K)=\mathbf 1_K,
	\qquad
	\deg\Delta_K=1,
	\]
	and algebraic independence of the seed coordinates, we obtain
	\[
	c_{\mathbf a}\neq0
	\quad\Longrightarrow\quad
	\sum_Ka_K\mathbf 1_K=\mathbf 1_B,
	\qquad
	\sum_Ka_K=1.
	\]
	For any fixed \(K_0\in\mathcal C\), this yields
	$\mathbf 1_B-\mathbf 1_{K_0}
	=
	\sum_Ka_K(\mathbf 1_K-\mathbf 1_{K_0}),$
	and hence
	\begin{equation}
		\label{eq:fsb-local-seed-weight-span}
		V(\mathcal M)
		=
		\operatorname{span}_{\mathbb Q}
		\{\mathbf 1_K-\mathbf 1_{K'}:K,K'\in\mathcal C\}.
	\end{equation}

	Let
	$\mathcal C
	=
	\{I_F^{\mathrm{tgt}}:F\in\operatorname{Faces}(G)\}
	\subseteq\binom{[n]}{k}$
	be the set of distinct target labels of all faces of \(G\),
	including boundary faces.
	The corresponding label set for \(G^\rho\) is
	\[
	\rho(\mathcal C)=\{\rho(K):K\in\mathcal C\},
	\qquad
	\rho(K)=\{\rho(a):a\in K\}.
	\]
	By admissibility and \cite[Theorem~4.21]{FSB}, we have
	\[
	\mathcal A(\Sigma_G^T)=R_\mu,
	\qquad
	\mathcal A(\Sigma_{G^\rho}^T)=R_\pi.
	\]
	Extend \(\rho\) linearly to \(\mathbb Q^n\) by
	\(\rho(e_a)=e_{\rho(a)}\).
	Applying \eqref{eq:fsb-local-seed-weight-span} to these two seeds gives
	\[
	\begin{aligned}
		V(\mathcal M_\pi)
		&=
		\operatorname{span}_{\mathbb Q}
		\{\mathbf 1_{\rho(K)}-\mathbf 1_{\rho(K')}:
		K,K'\in\mathcal C\}\\
		&=
		\rho\!\left(
		\operatorname{span}_{\mathbb Q}
		\{\mathbf 1_K-\mathbf 1_{K'}:K,K'\in\mathcal C\}
		\right)\\
		&=\rho\bigl(V(\mathcal M_\mu)\bigr).
	\end{aligned}
	\]
	Then we have
	\[
	\begin{aligned}
		a\sim_{\mathcal M_\mu}b
		&\Longleftrightarrow e_a-e_b\in V(\mathcal M_\mu)\\
		&\Longleftrightarrow
		e_{\rho(a)}-e_{\rho(b)}\in V(\mathcal M_\pi)\\
		&\Longleftrightarrow
		\rho(a)\sim_{\mathcal M_\pi}\rho(b).
	\end{aligned}
	\]
	Thus the component ground sets of \(\mathcal M_\mu\) are
	\[
	A_j=\rho^{-1}(S_j),
	\qquad 1\leq j\leq c.
	\]
	
	We next prove local admissibility.
	Let \(f,u\in\operatorname{Bound}(k,n)\) be the affine lifts
	of \(\pi,\pi\rho\), and put
	$r=f^{-1}u.$
	Then
	\[
	\overline r=\rho,
	\qquad
	\operatorname{av}(r)=0,
	\qquad
	u\leq_R f.
	\]
	By \cite[Lemma~4.17]{FSB}, the bounded affine lift of
	\(\mu\) is
	\[
	m=r^{-1}u=r^{-1}fr.
	\]
	Admissibility and the affine dimension formula imply
	\[
	\dim\widehat\Pi_\mu^\circ
	=
	\dim\widehat\Pi_\pi^\circ
	\quad\Longrightarrow\quad
	\ell(m)=\ell(f).
	\]
	Hence
	\begin{equation}
		\label{eq:fsb-local-global-factorizations}
		f=ur^{-1},
		\qquad
		m=r^{-1}u,
		\qquad
		\ell(f)=\ell(u)+\ell(r)=\ell(m).
	\end{equation}
	
	Fix \(j\) with \(n_j>1\). Write
	$\widetilde A_j=A_j+n\mathbb Z,
	\widetilde S_j=S_j+n\mathbb Z,$
	and choose increasing bijections
	\[
	a_j:\mathbb Z\xrightarrow{\sim}\widetilde A_j,
	\qquad
	b_j:\mathbb Z\xrightarrow{\sim}\widetilde S_j
	\]
	satisfying
	\[
	a_j(t+n_j)=a_j(t)+n,
	\qquad
	b_j(t+n_j)=b_j(t)+n.
	\]
	Since
	\[
	\begin{aligned}
		f(\widetilde S_j)&=\widetilde S_j,
		&
		m(\widetilde A_j)&=\widetilde A_j,\\
		r(\widetilde A_j)&=\widetilde S_j,
		&
		u(\widetilde A_j)&=\widetilde S_j,
	\end{aligned}
	\]
	we may define affine permutations of period \(n_j\) by
	\[
	\begin{aligned}
		f_j&=b_j^{-1}fb_j,
		&
		u_j&=b_j^{-1}ua_j,\\
		r_j&=b_j^{-1}ra_j,
		&
		m_j&=a_j^{-1}ma_j.
	\end{aligned}
	\]
	They satisfy
	\begin{equation}
		\label{eq:fsb-local-affine-identities}
		f_j=u_jr_j^{-1},
		\qquad
		m_j=r_j^{-1}u_j=r_j^{-1}f_jr_j.
	\end{equation}
	
	For an affine permutation \(w\) of period \(N\), define
$	\operatorname{av}(w)
	=
	\frac1N\sum_{t=1}^N(w(t)-t)\in\mathbb Z.$
For affine permutations of the same period, we have
	\[
	\operatorname{av}(v\circ w)
	=
	\operatorname{av}(v)+\operatorname{av}(w).
	\]
	
	We first choose the cyclic origin of \(A_j\).
	For \(h\in\mathbb Z\), replace \(a_j\) by
	\[
	a_j'(t)=a_j(t+h).
	\]
	The local affine permutations defined using \(a_j'\) are
	\[
	\begin{aligned}
		r_j'(t)
		&=b_j^{-1}\bigl(r(a_j'(t))\bigr)
		=r_j(t+h),\\
		u_j'(t)
		&=b_j^{-1}\bigl(u(a_j'(t))\bigr)
		=u_j(t+h),\\
		m_j'(t)
		&=(a_j')^{-1}\bigl(m(a_j'(t))\bigr)
		=m_j(t+h)-h.
	\end{aligned}
	\]
	The permutation \(f_j\) remains unchanged.
	Since \(r_j(t)-t\) is \(n_j\)-periodic,
	\[
	\begin{aligned}
		\operatorname{av}(r_j')
		&=
		\frac1{n_j}\sum_{t=1}^{n_j}
		\bigl(r_j(t+h)-t\bigr)\\
		&=
		\frac1{n_j}\sum_{t=1}^{n_j}
		\bigl(r_j(t+h)-(t+h)\bigr)+h\\
		&=\operatorname{av}(r_j)+h.
	\end{aligned}
	\]
	We choose \(h=-\operatorname{av}(r_j)\) and henceforth
	use the new enumerations and local permutations, omitting
	the primes. Thus
	\begin{equation}
		\label{eq:fsb-local-average-zero}
		\operatorname{av}(r_j)=0.
	\end{equation}
	Only the cyclic origin of \(A_j\) has changed; its inherited
	cyclic order is preserved.
	
	We next verify local length-additivity.
	For an affine permutation \(w\), define
	\[
	\operatorname{Inv}(w)
	=
	\{(s,t)\in\mathbb Z^2:s<t,\ w(s)>w(t)\}.
	\]
	For affine permutations \(x,y\) of the same period, by
	\cite[Lemma~2.26]{FSB}, we have
	\[
	\ell(x\circ y)=\ell(x)+\ell(y)
	\quad\Longleftrightarrow\quad
	\operatorname{Inv}(x)
	\cap\operatorname{Inv}(y^{-1})=\varnothing.
	\]
	Applying this criterion to
	\eqref{eq:fsb-local-global-factorizations}, we obtain
	\[
	\begin{aligned}
		\operatorname{Inv}(u)\cap\operatorname{Inv}(r)
		&=\varnothing,\\
		\operatorname{Inv}(u^{-1})\cap\operatorname{Inv}(r^{-1})
		&=\varnothing.
	\end{aligned}
	\]
	
	For any integers \(s<t\), the increasing enumerations satisfy
	\[
	a_j(s)<a_j(t),
	\qquad
	b_j(s)<b_j(t).
	\]
	The definitions of the local permutations therefore give
	\[
	\begin{aligned}
		u_j(s)>u_j(t)
		&\Longleftrightarrow
		u(a_j(s))>u(a_j(t)),\\
		r_j(s)>r_j(t)
		&\Longleftrightarrow
		r(a_j(s))>r(a_j(t)),\\
		u_j^{-1}(s)>u_j^{-1}(t)
		&\Longleftrightarrow
		u^{-1}(b_j(s))>u^{-1}(b_j(t)),\\
		r_j^{-1}(s)>r_j^{-1}(t)
		&\Longleftrightarrow
		r^{-1}(b_j(s))>r^{-1}(b_j(t)).
	\end{aligned}
	\]
	Thus a common inversion of \(u_j,r_j\) would give a common
	inversion of \(u,r\); the same argument applies to their
	inverses. Hence
	\[
	\begin{aligned}
		\operatorname{Inv}(u_j)\cap\operatorname{Inv}(r_j)
		&=\varnothing,\\
		\operatorname{Inv}(u_j^{-1})
		\cap\operatorname{Inv}(r_j^{-1})
		&=\varnothing.
	\end{aligned}
	\]
	Since
$	f_j=u_j\circ r_j^{-1}$ and $
	m_j=r_j^{-1}\circ u_j,$
	the length criterion yields
	\begin{equation}
		\label{eq:fsb-local-length-additivity}
		\ell(f_j)
		=
		\ell(u_j)+\ell(r_j)
		=
		\ell(m_j).
	\end{equation}
	
	Put \(k_j=\operatorname{rank}(\mathcal M_\pi|_{S_j})\).
	The restriction of \(f\) to this component gives
	$f_j\in\operatorname{Bound}(k_j,n_j).$
	Using \(u_j=f_j\circ r_j\) and
	\eqref{eq:fsb-local-average-zero}, we obtain
	\[
	\operatorname{av}(u_j)
	=
	\operatorname{av}(f_j)+\operatorname{av}(r_j)
	=k_j.
	\]
	Together with \eqref{eq:fsb-local-length-additivity}, this gives
	$u_j\leq_R f_j.$
	Since \(\operatorname{Bound}(k_j,n_j)\) is a lower ideal
	for right weak order, then
	\[
	u_j\in\operatorname{Bound}(k_j,n_j).
	\]
	
	For \(m_j\), boundedness follows directly from
	$a_j(t)<m(a_j(t))
	\leq a_j(t)+n
	=a_j(t+n_j).$
	Applying the increasing map \(a_j^{-1}\), we get
	\[
	t<m_j(t)\leq t+n_j.
	\]
	Moreover, one has
	\[
	\begin{aligned}
		\operatorname{av}(m_j)
		&=
		\operatorname{av}(r_j^{-1}\circ f_j\circ r_j)\\
		&=
		-\operatorname{av}(r_j)
		+\operatorname{av}(f_j)
		+\operatorname{av}(r_j)\\
		&=k_j.
	\end{aligned}
	\]
	Hence
	$m_j\in\operatorname{Bound}(k_j,n_j).$
	For \(1\leq t\leq n_j\), let \(\alpha_j(t)\) and
	\(\beta_j(t)\) be the representatives in \([n]\) of
	\(a_j(t)\) and \(b_j(t)\) modulo \(n\), respectively.
	These define cyclic order-preserving bijections
	\[
	\alpha_j:[n_j]\xrightarrow{\sim}A_j,
	\qquad
	\beta_j:[n_j]\xrightarrow{\sim}S_j.
	\]
	By the definitions of \(\pi_j,\mu_j,\rho_j\), we have
	\[
	\begin{aligned}
		f_j(t)&\equiv\pi_j(t)\pmod{n_j},\\
		m_j(t)&\equiv\mu_j(t)\pmod{n_j},\\
		r_j(t)&\equiv\rho_j(t)\pmod{n_j},\\
		u_j(t)&\equiv\pi_j(\rho_j(t))\pmod{n_j}.
	\end{aligned}
	\]
	Thus \(f_j,m_j,u_j\) are the bounded affine lifts of
	\(\pi_j,\mu_j,\pi_j\rho_j\), respectively.
	The identities
	\[
	f_j=r_j\circ m_j\circ r_j^{-1},
	\qquad
	u_j\leq_R f_j
	\]
	therefore imply
	\[
	\pi_j=\rho_j\mu_j\rho_j^{-1},
	\qquad
	\pi_j\rho_j\leq_\circ\pi_j.
	\]
	Finally, \eqref{eq:fsb-local-length-additivity} gives
	\[
	\begin{aligned}
		\dim\widehat\Pi_{\mu_j}^{\circ}
		&=k_j(n_j-k_j)+1-\ell(m_j)\\
		&=k_j(n_j-k_j)+1-\ell(f_j)\\
		&=\dim\widehat\Pi_{\pi_j}^{\circ}.
	\end{aligned}
	\]
	Hence \(\rho_j\) is admissible for \(\pi_j\).
	
	If \(n_j=1\), then
	$\pi_j=\mu_j=\rho_j=\mathrm{id}.$
	Since the positroid is loopless, its singleton component has
	rank one. Its unique element is therefore a coloop,
	that is, an element belonging to every basis.
\end{proof}

\begin{definition}[Homogeneous amalgamation of component seeds]
	\label{def:necklace-amalgamation}
	Let \(\Sigma_j\) be the seed associated with the component
	diagram \(D_j\), for \(1\leq j\leq c\).
	Assume that the diagrams are embedded in a common disc
	compatibly with the noncrossing partition
	\([n]=S_1\sqcup\cdots\sqcup S_c\).
	Set
	\[
	D=\bigcup_{j=1}^cD_j,
	\qquad
	Q(D)=
	\left(\bigsqcup_{j=1}^cQ(D_j)\right)\big/\sim,
	\]
	where \(\sim\) identifies frozen vertices corresponding
	to the same boundary region of \(D\).

	For each global face \(F\), denote by \(F_j\) the face
	of \(D_j\) containing \(F\), and by \(x_{j,F_j}\in R_j\)
	its coordinate.
	The \emph{homogeneous amalgamation} of
	\(\Sigma_1,\ldots,\Sigma_c\) is the seed
	\[
	\Sigma_D
	=
	\bigl((X_F)_{F\in\operatorname{Faces}(D)},Q(D)\bigr),
	\qquad
	X_F
	=
	\partial^{-1}
	\bigl(x_{1,F_1}\otimes\cdots\otimes x_{c,F_c}\bigr),
	\]
	where \(\partial^{-1}\) denotes the inverse of the graded isomorphism
	\[
	\partial:R_\pi\xrightarrow{\sim}
	R_1\mathbin{\#}\cdots\mathbin{\#}R_c
	\]
	from Lemma~\ref{lem:fsb-component-coefficients}.
\end{definition}

	This expression is well defined because
$	\deg x_{j,F_j}=1 (1\leq j\leq c).$
	For distinct indices \(\ell,j\), the connected diagram \(D_j\)
	is disjoint from \(D_\ell\) and meets the boundary of the disc.
	It is therefore contained in a unique boundary face
	\(F_{\ell j}\) of \(D_\ell\).
	Let \(x_{\ell,F_{\ell j}}\) be the coordinate of the frozen
	vertex of the component seed \(\Sigma_\ell\) corresponding
	to this face, and define
	\[
	q_{\ell j}
	=
	x_{\ell,F_{\ell j}}
	\in\mathbb P_\ell,
	\qquad
	\deg q_{\ell j}=1.
	\]
	Thus \(q_{\ell j}\) is a frozen coordinate of component \(\ell\),
	determined by the boundary region containing component \(j\). Define
	\[
	u_j
	=
	\prod_{\ell\neq j}
	\iota_\ell\left(\frac{q_{\ell j}}{p_\ell}\right).
	\]
	By \eqref{eq:fsb-component-frozen-lattice},
	\[
	u_j\in\mathbb P_\pi,
	\qquad
	\deg u_j=0.
	\]
	
	For a homogeneous local cluster variable \(x_v\in R_j\)
	of degree \(d_v\), its homogeneous lift \(X_v\in R_\pi\)
	is defined by
	\[
	\partial(X_v)
	=
	q_{1j}^{d_v}\otimes\cdots\otimes
	q_{j-1,j}^{d_v}\otimes x_v\otimes
	q_{j+1,j}^{d_v}\otimes\cdots\otimes q_{cj}^{d_v}.
	\]
	Equivalently,
	\[
	X_v
	=
	(pu_j)^{d_v}
	\iota_j\left(\frac{x_v}{p_j^{d_v}}\right).
	\]
	For initial face variables, we have \(d_v=1\), and this formula
	agrees with the definition of \(X_F\).
	The elements \(q_{\ell j}\) and \(u_j\) are invariant
	under mutations in component \(j\).

\begin{example}[Homogeneous amalgamation of two rank-one seeds]
	\label{ex:necklace-amalgamation}
	Let
	\[
	S_1=\{1,2,3,4\},\qquad S_2=\{5,6,7,8\},\qquad
	\mathcal M_\pi=
	\left\{I\cup J:I\in\binom{S_1}{2},\ J\in\binom{S_2}{2}\right\}.
	\]
	The component positroids are uniform of rank two, and
	\[
	\pi=(1\,3)(2\,4)(5\,7)(6\,8).
	\]
	Write \(a_{ij}=\Delta^{(1)}_{\{i,j\}}\) and
	\(b_{rs}=\Delta^{(2)}_{\{r,s\}}\) for the local Pl\"ucker
	coordinates. The diagonal graded isomorphism satisfies
	\[
	\partial(\Delta_{ijrs})=a_{ij}\otimes b_{rs}
	\qquad(i<j\text{ in }S_1,\ r<s\text{ in }S_2).
	\]
	Choose the component seeds with mutable variables \(a_{13}\) and
	\(b_{57}\), and respective frozen variables
	\[
	(a_{12},a_{23},a_{34},a_{14}),\qquad
	(b_{56},b_{67},b_{78},b_{58}).
	\]
	Their Postnikov diagram and reduced plabic graph are shown in
	Figure~\ref{fig:amalgamation-diagrams}. All face labels are target labels;
	the label \(I\) denotes the coordinate \(\Delta_I\).
	
	\begingroup
	\definecolor{amalgFrozen}{RGB}{42,81,128}
	\definecolor{amalgMutable}{RGB}{0,112,103}
	\definecolor{amalgShared}{RGB}{155,78,8}
	\definecolor{amalgSOne}{RGB}{34,82,145}
	\definecolor{amalgSTwo}{RGB}{185,93,15}
	\definecolor{amalgSThree}{RGB}{0,125,105}
	\definecolor{amalgSFour}{RGB}{133,66,145}
	\tikzset{
		amalg boundary/.style={draw=black!45,line width=.5pt},
		amalg label/.style={font=\small,inner sep=1.3pt,fill=white},
		amalg face/.style={amalg label,text=amalgFrozen},
		amalg mutable face/.style={amalg label,text=amalgMutable},
		amalg shared face/.style={amalg label,text=amalgShared},
		amalg black/.style={circle,draw=black,fill=black,inner sep=2.2pt},
		amalg white/.style={circle,draw=black,fill=white,inner sep=2.2pt},
		amalg strand/.style={line width=.65pt,rounded corners=4pt,
			postaction={decorate},decoration={markings,
				mark=at position .48 with {\arrow{Stealth[length=3.5pt,width=3pt]}},
				mark=at position .84 with {\arrow{Stealth[length=3.5pt,width=3pt]}}}},
		amalg mutable/.style={circle,draw=amalgMutable,fill=amalgMutable!7,
			line width=.7pt,minimum size=7mm,inner sep=2pt,font=\small},
		amalg frozen/.style={rectangle,draw=amalgFrozen,fill=amalgFrozen!5,
			minimum width=8mm,minimum height=6mm,inner sep=2pt,font=\small},
		amalg shared/.style={amalg frozen,draw=amalgShared,fill=amalgShared!9},
		amalg arrow/.style={-{Stealth[length=4.5pt,width=3.5pt]},line width=.65pt},
		amalg frozen arrow/.style={amalg arrow,draw=black!45,dashed}
	}
	
	\newcommand{\amalgGeometry}{%
		\draw[amalg boundary] (0,0) ellipse[x radius=4.8,y radius=2.4];
		\foreach \i/\ang in {1/255,2/210,3/150,4/105,5/75,6/30,7/-30,8/-75}{%
			\coordinate (b\i) at ({4.8*cos(\ang)},{2.4*sin(\ang)});
			\fill (b\i) circle[radius=1.3pt];
			\node[font=\small] at ({5.13*cos(\ang)},{2.73*sin(\ang)}) {\(\i\)};
		}
		\coordinate (v1) at (-1.2,-.7); \coordinate (v2) at (-3,-.7);
		\coordinate (v3) at (-3,.7);   \coordinate (v4) at (-1.2,.7);
		\coordinate (v5) at (1.2,.7);  \coordinate (v6) at (3,.7);
		\coordinate (v7) at (3,-.7);   \coordinate (v8) at (1.2,-.7);
		\foreach \i/\j in {1/2,2/3,3/4,1/4,5/6,6/7,7/8,5/8}{%
			\coordinate (e\i\j) at ($(v\i)!.5!(v\j)$);
		}
	}
	\newcommand{\amalgFaceLabels}{%
		\node[amalg mutable face] at (-2.1,0) {\(1356\)};
		\node[amalg mutable face] at (2.1,0) {\(1257\)};
		\node[amalg shared face] at (0,0) {\(1256\)};
		\node[amalg face] at (-2.45,-1.48) {\(2356\)};
		\node[amalg face] at (-3.98,0) {\(3456\)};
		\node[amalg face] at (-2.45,1.48) {\(1456\)};
		\node[amalg face] at (2.45,1.48) {\(1267\)};
		\node[amalg face] at (3.98,0) {\(1278\)};
		\node[amalg face] at (2.45,-1.48) {\(1258\)};
	}
	
	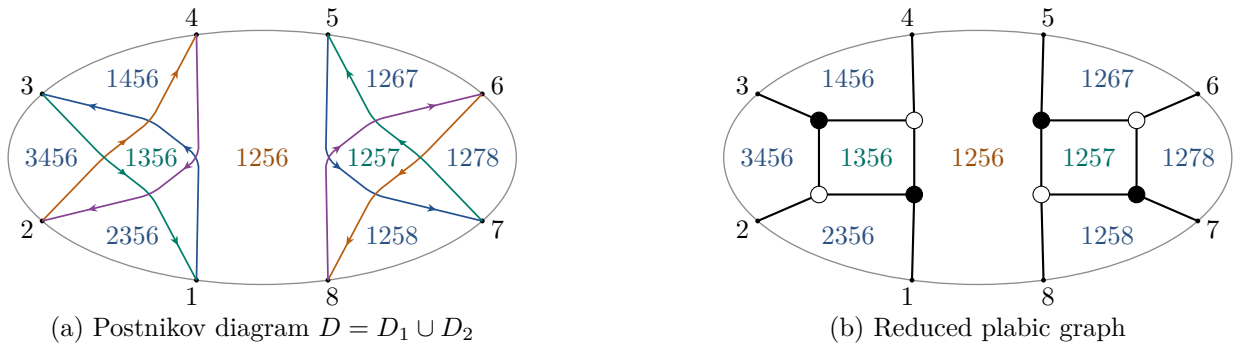
\begin{figure}[htbp]
		\centering
		\begin{tikzpicture}[scale=.70]
			\amalgGeometry
			\draw[amalg strand,amalgSOne]   (b1)--(e14)--(e34)--(b3);
			\draw[amalg strand,amalgSTwo]   (b2)--(e23)--(e34)--(b4);
			\draw[amalg strand,amalgSThree] (b3)--(e23)--(e12)--(b1);
			\draw[amalg strand,amalgSFour]  (b4)--(e14)--(e12)--(b2);
			\draw[amalg strand,amalgSOne]   (b5)--(e58)--(e78)--(b7);
			\draw[amalg strand,amalgSTwo]   (b6)--(e67)--(e78)--(b8);
			\draw[amalg strand,amalgSThree] (b7)--(e67)--(e56)--(b5);
			\draw[amalg strand,amalgSFour]  (b8)--(e58)--(e56)--(b6);
			\amalgFaceLabels
			\node[font=\small] at (0,-3.25) {(a) Postnikov diagram \(D=D_1\cup D_2\)};
		\end{tikzpicture}
		\hfill
		\begin{tikzpicture}[scale=.70]
			\amalgGeometry
			\draw[line width=.8pt] (v1)--(v2)--(v3)--(v4)--cycle;
			\draw[line width=.8pt] (v5)--(v6)--(v7)--(v8)--cycle;
			\foreach \i in {1,...,8}{\draw[line width=.8pt] (v\i)--(b\i);}
			\foreach \i in {1,3,5,7}{\node[amalg black] at (v\i) {};}
			\foreach \i in {2,4,6,8}{\node[amalg white] at (v\i) {};}
			\amalgFaceLabels
			\node[font=\small] at (0,-3.25) {(b) Reduced plabic graph};
		\end{tikzpicture}
		\caption{Two representations of the positroid \(\mathcal M_\pi\).
			The boundary labels increase clockwise. The face labelled \(1256\)
			is a single region meeting both boundary gaps \(8\)--\(1\) and
			\(4\)--\(5\).}
		\label{fig:amalgamation-diagrams}
	\end{figure}
	
	The common boundary region has local coordinates \(a_{12}\) and
	\(b_{56}\). Accordingly,
	\[
	q_{21}=b_{56},\qquad q_{12}=a_{12},\qquad
	Z_0=\Delta_{1256},\qquad \partial(Z_0)=a_{12}\otimes b_{56}.
	\]
	The component quivers are glued at the frozen vertices carrying
	\(a_{12}\) and \(b_{56}\); the resulting vertex carries \(Z_0\).
	
	\begin{figure}[htbp]
		\centering
		\begin{tikzpicture}[x=1cm,y=1cm]
			\begin{scope}[shift={(-3.6,0)}]
				\node[amalg mutable] (ax) at (0,0) {\(a_{13}\)};
				\node[amalg shared] (a0) at (1.4,0) {\(a_{12}\)};
				\node[amalg frozen] (a1) at (0,-1.1) {\(a_{23}\)};
				\node[amalg frozen] (a2) at (-1.4,0) {\(a_{34}\)};
				\node[amalg frozen] (a3) at (0,1.1) {\(a_{14}\)};
				\draw[amalg arrow] (a0)--(ax); \draw[amalg arrow] (a2)--(ax);
				\draw[amalg arrow] (ax)--(a1); \draw[amalg arrow] (ax)--(a3);
				\foreach \s/\t in {a1/a0,a1/a2,a3/a0,a3/a2}{%
					\draw[amalg frozen arrow] (\s)--(\t);}
				\node[font=\small] at (0,-1.65) {\(Q(D_1)\)};
			\end{scope}
			\begin{scope}[shift={(3.6,0)}]
				\node[amalg mutable] (by) at (0,0) {\(b_{57}\)};
				\node[amalg shared] (b0) at (-1.4,0) {\(b_{56}\)};
				\node[amalg frozen] (b4) at (0,1.1) {\(b_{67}\)};
				\node[amalg frozen] (b5) at (1.4,0) {\(b_{78}\)};
				\node[amalg frozen] (b6) at (0,-1.1) {\(b_{58}\)};
				\draw[amalg arrow] (b0)--(by); \draw[amalg arrow] (b5)--(by);
				\draw[amalg arrow] (by)--(b4); \draw[amalg arrow] (by)--(b6);
				\foreach \s/\t in {b4/b0,b4/b5,b6/b0,b6/b5}{%
					\draw[amalg frozen arrow] (\s)--(\t);}
				\node[font=\small] at (0,-1.65) {\(Q(D_2)\)};
			\end{scope}
			\draw[amalg arrow,draw=amalgShared] (0,-1.75)--(0,-3.5);
			\node[font=\small,anchor=west] at (.25,-2.6)
			{identify the two orange vertices};
			\begin{scope}[shift={(0,-5)}]
				\node[amalg mutable] (X) at (-2.1,0) {\(X\)};
				\node[amalg mutable] (Y) at (2.1,0) {\(Y\)};
				\node[amalg shared] (Z0) at (0,0) {\(Z_0\)};
				\node[amalg frozen] (Z1) at (-2.1,-1.1) {\(Z_1\)};
				\node[amalg frozen] (Z2) at (-4.2,0) {\(Z_2\)};
				\node[amalg frozen] (Z3) at (-2.1,1.1) {\(Z_3\)};
				\node[amalg frozen] (Z4) at (2.1,1.1) {\(Z_4\)};
				\node[amalg frozen] (Z5) at (4.2,0) {\(Z_5\)};
				\node[amalg frozen] (Z6) at (2.1,-1.1) {\(Z_6\)};
				\foreach \s/\t in {Z0/X,Z2/X,X/Z1,X/Z3,Z0/Y,Z5/Y,Y/Z4,Y/Z6}{%
					\draw[amalg arrow] (\s)--(\t);}
				\foreach \s/\t in {Z1/Z0,Z1/Z2,Z3/Z0,Z3/Z2,Z4/Z0,Z4/Z5,Z6/Z0,Z6/Z5}{%
					\draw[amalg frozen arrow] (\s)--(\t);}
				\node[font=\small] at (0,-1.65) {\(Q(D)\)};
			\end{scope}
		\end{tikzpicture}
		\caption{The component ice quivers and their amalgamation.
			Circles denote mutable vertices and squares frozen vertices.
			Dashed arrows join frozen vertices and do not enter the extended
			exchange matrix. The two component frozen vertices are replaced by
			the single vertex \(Z_0\), whose coordinate is specified by
			\(\partial(Z_0)=a_{12}\otimes b_{56}\).}
		\label{fig:amalgamation-quivers}
	\end{figure}
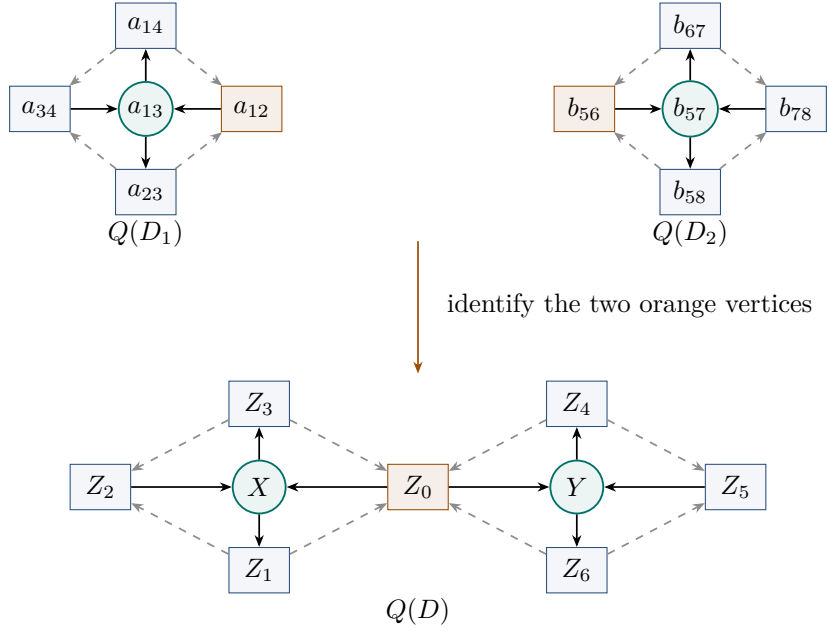
	\endgroup
	
	The global seed coordinates are
	\[
	\begin{array}{c|c@{\qquad}c|c}
		\text{Coordinate}&\text{Image under }\partial&
		\text{Coordinate}&\text{Image under }\partial\\
		\hline
		X=\Delta_{1356}&a_{13}\otimes b_{56}&
		Y=\Delta_{1257}&a_{12}\otimes b_{57}\\
		Z_0=\Delta_{1256}&a_{12}\otimes b_{56}&
		Z_4=\Delta_{1267}&a_{12}\otimes b_{67}\\
		Z_1=\Delta_{2356}&a_{23}\otimes b_{56}&
		Z_5=\Delta_{1278}&a_{12}\otimes b_{78}\\
		Z_2=\Delta_{3456}&a_{34}\otimes b_{56}&
		Z_6=\Delta_{1258}&a_{12}\otimes b_{58}\\
		Z_3=\Delta_{1456}&a_{14}\otimes b_{56}&&
	\end{array}
	\]
	In particular, the ordinary global necklace is
	\[
	(\Delta_{I_1^\pi},\ldots,\Delta_{I_8^\pi})
	=(Z_0,Z_1,Z_2,Z_3,Z_0,Z_4,Z_5,Z_6).
	\]
	There are two mutable and seven frozen variables. In the row order
	\((X,Y,Z_0,Z_1,Z_2,Z_3,Z_4,Z_5,Z_6)\), with columns \((X,Y)\),
	the convention \(b_{vw}=\#(w\to v)-\#(v\to w)\) gives
	\[
	\widetilde B=
	\begin{pmatrix}
		0&0&-1&1&-1&1&0&0&0\\
		0&0&-1&0&0&0&1&-1&1
	\end{pmatrix}^{\!T}.
	\]
	The local Pl\"ucker relations
	\[
	a_{13}a_{24}=a_{12}a_{34}+a_{14}a_{23},\qquad
	b_{57}b_{68}=b_{56}b_{78}+b_{58}b_{67}
	\]
	therefore yield
	\[
	\begin{aligned}
		XX'&=Z_0Z_2+Z_1Z_3,&X'&=\Delta_{2456},\\
		YY'&=Z_0Z_5+Z_4Z_6,&Y'&=\Delta_{1268}.
	\end{aligned}
	\]
	For example,
	\[
	\begin{aligned}
		\partial(XX')
		&=(a_{13}a_{24})\otimes b_{56}^{2}\\
		&=(a_{12}a_{34}+a_{14}a_{23})\otimes b_{56}^{2}\\
		&=\partial(Z_0Z_2+Z_1Z_3).
	\end{aligned}
	\]
	Since \(b_{XY}=b_{YX}=0\), mutations in the two components commute.
	
	Finally, choose
	\[
	p=Z_1=\Delta_{2356},\qquad p_1=a_{23},\qquad p_2=b_{56}.
	\]
	The coefficient factors in the homogeneous amalgamation are
	\[
	\begin{aligned}
		u_1&=\iota_2\!\left(\frac{b_{56}}{b_{56}}\right)=1,\\
		u_2&=\iota_1\!\left(\frac{a_{12}}{a_{23}}\right)=\frac{Z_0}{Z_1}.
	\end{aligned}
	\]
	Thus the degree-one lift formula gives
	\[
	\begin{aligned}
		X&=pu_1\,\iota_1\!\left(\frac{a_{13}}{p_1}\right)
		=Z_1\,\iota_1\!\left(\frac{a_{13}}{a_{23}}\right),\\
		Y&=pu_2\,\iota_2\!\left(\frac{b_{57}}{p_2}\right)
		=Z_0\,\iota_2\!\left(\frac{b_{57}}{b_{56}}\right).
	\end{aligned}
	\]
\end{example}

\begin{lemma}
	\label{lem:fsb-amalgamation}
	\label{prop:fsb-connected-reduction}
	Let
	\[
	[n]=S_1\sqcup\cdots\sqcup S_c
	\]
	be a noncrossing component partition.
	For each \(j\), let \(\Sigma_j\) and \(\Sigma'_j\)
	be homogeneous positroid seed patterns on \(S_j\).
	Denote their respective homogeneous amalgamations
	by \(\Sigma\) and \(\Sigma'\).
	If \(\Sigma_j\) and \(\Sigma'_j\) are quasi-coincident
	for every \(j\), then \(\Sigma\) and \(\Sigma'\)
	are quasi-coincident.
\end{lemma}

\begin{proof}
	For each \(j\), let \(\mathbb P_j\) be the common coefficient
	group of the two component seed patterns, and put
$	\mathbb P_j^0
	=
	\{c\in\mathbb P_j:\deg c=0\}.$
	Fix the degree-one units \(p_j\in\mathbb P_j\) and
	\(p\in R_\pi\) satisfying
	\[
	\partial(p)=p_1\otimes\cdots\otimes p_c.
	\]
	We use the same notation \(\iota_j\) for the extension of
	\(\iota_j:R_j^0\hookrightarrow R_\pi\) to fraction fields.
	By \eqref{eq:fsb-component-frozen-lattice}, the coefficient
	groups of the two amalgamations coincide:
	\begin{equation}
		\label{eq:fsb-amalgamation-common-coefficients}
		\mathbb P_\Sigma
		=
		\langle p\rangle
		\prod_{j=1}^c\iota_j(\mathbb P_j^0)
		=
		\mathbb P_{\Sigma'}
		=
		\mathbb P_\pi.
	\end{equation}
	
	We first verify compatibility with mutation.
	For a homogeneous \(h\in R_j\) of degree \(d\), define
	\[
	L_j(h)
	=
	p^d u_j^d
	\iota_j\left(\frac{h}{p_j^d}\right).
	\]
	This formula extends linearly to an algebra homomorphism
	\(L_j:R_j\to R_\pi\), since
	\[
	L_j(hg)=L_j(h)L_j(g)
	\]
	for homogeneous \(h,g\).
	By Definition~\ref{def:necklace-amalgamation},
	\(L_j\) sends each local seed variable to its global
	counterpart. Moreover, the exchange column at a mutable
	vertex of component \(j\) is inherited from that component.
	Consequently, applying \(L_j\) to a local exchange relation
	gives the corresponding global exchange relation:
	\[
	x_kx_k^\ast=M_++M_-
	\quad\Longrightarrow\quad
	L_j(x_k)L_j(x_k^\ast)
	=
	L_j(M_+)+L_j(M_-).
	\]
	The factors defining \(u_j\) are frozen, so this compatibility
	persists under subsequent mutations.
	The same argument applies to the second amalgamation.
	
	There are no arrows between mutable vertices of distinct
	components. Hence mutations in distinct components commute,
	and any collection of component mutation sequences lifts
	to a global mutation sequence.
	By componentwise quasi-coincidence, we may therefore choose
	corresponding component seeds such that
	\[
	x'_v=c_vx_v,
	\qquad
	c_v\in\mathbb P_j,
	\qquad
	\widehat y'_{j,k}=\widehat y_{j,k}
	\]
	for every mutable vertex \(v\) and every mutable direction
	\(k\) in component \(j\).
	
	Write
	\[
	d_v=\deg x_v,
	\qquad
	d'_v=\deg x'_v,
	\qquad
	\delta_v=d'_v-d_v=\deg c_v.
	\]
	Let \(u_j\) and \(u'_j\) be the frozen factors in the two
	amalgamations. Their corresponding global variables satisfy
	\[
	X_v
	=
	p^{d_v}u_j^{d_v}
	\iota_j\left(\frac{x_v}{p_j^{d_v}}\right),
	\qquad
	X'_v
	=
	p^{d'_v}(u'_j)^{d'_v}
	\iota_j\left(\frac{x'_v}{p_j^{d'_v}}\right).
	\]

By \eqref{eq:fsb-amalgamation-common-coefficients}, we have
$\iota_j\left(\frac{c_v}{p_j^{\delta_v}}\right)
\in\mathbb P_\pi.$
Since \(p,u_j,u'_j\in\mathbb P_\pi\) and
\(\mathbb P_\pi\) is a multiplicative group, we conclude that
\[
\frac{X'_v}{X_v}
=
p^{\delta_v}
\frac{(u'_j)^{d'_v}}{u_j^{d_v}}
\iota_j\left(\frac{c_v}{p_j^{\delta_v}}\right)
\in\mathbb P_\pi.
\]

It remains to compare exchange ratios.
Fix a mutable vertex \(k\) of component \(j\).
Let \((x_v)_v\) be its extended cluster,
let \(d_v=\deg x_v\), and let \((b_{vk})_v\)
be the exchange column at \(k\).
Homogeneity of the exchange relation gives
\[
\sum_v b_{vk}d_v=0.
\]

Let \(\nu_j(v)\) denote the global vertex corresponding
to the local vertex \(v\), and write
\(\widetilde B^{\mathrm{glob}}=(B_{wK})\)
for the global exchange matrix.
Amalgamation identifies only frozen vertices from
different components. Thus \(\nu_j\) is injective,
and the arrows incident to \(\nu_j(k)\) are precisely
the images of the arrows incident to \(k\), with the
same orientations and multiplicities. Consequently,
\[
B_{w,\nu_j(k)}
=
\begin{cases}
	b_{vk}, & w=\nu_j(v),\\
	0, & w\notin\operatorname{im}(\nu_j).
\end{cases}
\]
Since \(X_{\nu_j(v)}=L_j(x_v)\), the global exchange
ratio at \(\nu_j(k)\) is
\[
\begin{aligned}
	\widehat Y_{j,k}
	&=\prod_w X_w^{B_{w,\nu_j(k)}}\\
	&=\prod_v L_j(x_v)^{b_{vk}}\\
	&=(pu_j)^{\sum_v b_{vk}d_v}
	\iota_j\left(
	p_j^{-\sum_v b_{vk}d_v}
	\prod_v x_v^{b_{vk}}
	\right)\\
	&=\iota_j\left(\prod_v x_v^{b_{vk}}\right)
	=\iota_j(\widehat y_{j,k}).
\end{aligned}
\]
Applying the same calculation to the second amalgamation
and using local quasi-equivalence, we obtain
$\widehat Y'_{j,k}
=
\iota_j(\widehat y'_{j,k})
=
\iota_j(\widehat y_{j,k})
=
\widehat Y_{j,k}.$

Thus the coefficient groups satisfy
$\mathbb P_\Sigma=\mathbb P_{\Sigma'}=\mathbb P_\pi.$
Moreover,
$\frac{X'_v}{X_v}$ lies in $\mathbb P_\pi$
for every global mutable vertex \(v\), and
\[
\widehat Y'_{j,k}=\widehat Y_{j,k}
\]
for each component \(j\) and each mutable vertex \(k\)
of that component.
These identities establish quasi-equivalence of the chosen
global seeds by Definition~\ref{def:fsb-quasi-equivalent-seeds}.
Hence the two amalgamated seed patterns are quasi-coincident.
\end{proof}

\begin{corollary}[The general case]
	\label{general case}
	Let \(\mathcal M_\pi\) be a loopless positroid,
	not necessarily connected, and let \(\rho\) be admissible
	for \(\pi\). Put
	\[
	\mu=\rho^{-1}\pi\rho.
	\]
	Let \(G\) and \(D\) be reduced plabic graphs with trip
	permutations \(\mu\) and \(\pi\), respectively.
	Then the identity map of \(R_\pi\) is a quasi-cluster
	isomorphism from the target-labelled cluster structure
	of \(G^\rho\) to either the source-labelled or the
	target-labelled cluster structure of \(D\).
\end{corollary}

\begin{proof}
	Let \(S_1,\ldots,S_c\) be the component ground sets
	of \(\mathcal M_\pi\), and put
	\[
	A_j=\rho^{-1}(S_j).
	\]
	By Lemma~\ref{lem:fsb-local-components},
	the sets \(A_j\) are the component ground sets of
	\(\mathcal M_\mu\). Moreover, suitable cyclic
	enumerations give local permutations satisfying
	\[
	\pi_j=\rho_j\mu_j\rho_j^{-1},
	\]
	with \(\rho_j\) admissible whenever \(|S_j|>1\).
	For each such component,
	Theorem~\ref{prop:necklace-connected-comparison}
	gives the required quasi-coincidence.
	If \(|S_j|=1\), then
$	R_j=\kk[p_j^{\pm1}],$
	and the local seeds have no mutable variables, so
	the comparison is the identity.
	
	The global seed patterns are the homogeneous
	amalgamations of their component seed patterns.
	Lemma~\ref{lem:fsb-amalgamation} therefore yields
	both asserted global comparisons.
\end{proof}

\FloatBarrier


\begin{thebibliography}{99}
	
	\bibitem{ARW}
	F. Ardila, F. Rinc\'on and L. Williams,
	\emph{Positroids and non-crossing partitions},
	Trans. Amer. Math. Soc. \textbf{368} (2016), 337--363.
	\href{https://arxiv.org/abs/1308.2698}{arXiv:1308.2698}.
	
	\bibitem{BKM}
	K. Baur, A. D. King and R. J. Marsh,
	\emph{Dimer models and cluster categories of Grassmannians},
	Proc. Lond. Math. Soc. (3) \textbf{113} (2016), no.~2, 213--260.
	
	\bibitem{Buchweitz}
	R.-O. Buchweitz,
	\emph{Maximal Cohen--Macaulay modules and Tate cohomology},
	Mathematical Surveys and Monographs \textbf{262},
	American Mathematical Society, Providence, RI, 2021.
	\href{https://doi.org/10.1090/surv/262}{doi:10.1090/surv/262}.
	
	\bibitem{CKP}
	\.I. \c{C}anak\c{c}{\i}, A. King and M. Pressland,
	\emph{Perfect matching modules, dimer partition functions and
		cluster characters},
	Adv. Math. \textbf{443} (2024), 109570.
	\href{https://arxiv.org/abs/2106.15924}{arXiv:2106.15924}.
	
	\bibitem{DehyKeller}
	R. Dehy and B. Keller,
	\emph{On the combinatorics of rigid objects in $2$-Calabi--Yau categories},
	Int. Math. Res. Not. IMRN (2008), rnn029.
	\href{https://arxiv.org/abs/0709.0882}{arXiv:0709.0882}.
	
	\bibitem{DWZII}
	H. Derksen, J. Weyman and A. Zelevinsky,
	\emph{Quivers with potentials and their representations II:
		applications to cluster algebras},
	J. Amer. Math. Soc. \textbf{23} (2010), 749--790.
	\href{https://arxiv.org/abs/0904.0676}{arXiv:0904.0676}.
	
	\bibitem{FG}
	M. Farber and P. Galashin,
	\emph{Weak separation, pure domains and cluster distance},
	Selecta Math. (N.S.) \textbf{24} (2018), 2093--2127.
	\href{https://arxiv.org/abs/1612.05387}{arXiv:1612.05387}.
	
	\bibitem{FZI}
	S. Fomin and A. Zelevinsky,
	\emph{Cluster algebras I: Foundations},
	J. Amer. Math. Soc. \textbf{15} (2002), no.~2, 497--529.
	
	\bibitem{BFZIII}
	A.~Berenstein, S.~Fomin, and A.~Zelevinsky,
	\emph{Cluster algebras III: Upper bounds and double Bruhat cells},
	Duke Math. J. \textbf{126} (2005), no.~1, 1--52.
	
	\bibitem{FZIV}
	S. Fomin and A. Zelevinsky,
	\emph{Cluster algebras IV: Coefficients},
	Compos. Math. \textbf{143} (2007), 112--164.
	\href{https://arxiv.org/abs/math/0602259}{arXiv:math/0602259}.
	
	\bibitem{Fra16}
	C. Fraser,
	\emph{Quasi-homomorphisms of cluster algebras},
	Adv. in Appl. Math. \textbf{81} (2016), 40--77.
	\href{https://doi.org/10.1016/j.aam.2016.06.005}
	{doi:10.1016/j.aam.2016.06.005};
	\href{https://arxiv.org/abs/1509.05385}{arXiv:1509.05385}.
	
	\bibitem{FSB}
	C. Fraser and M. Sherman-Bennett,
	\emph{Positroid cluster structures from relabeled plabic graphs},
	Algebr. Comb. \textbf{5} (2022), 469--513.
	\href{https://doi.org/10.5802/alco.220}{doi:10.5802/alco.220}.
	
	
	\bibitem{FuGyoda}
	C. Fu and Y. Gyoda,
	\emph{Compatibility degree of cluster complexes},
	Ann. Inst. Fourier (Grenoble) \textbf{74} (2024), 663--718.
	\href{https://arxiv.org/abs/1911.07193}{arXiv:1911.07193}.
	
	\bibitem{GL}
	P. Galashin and T. Lam,
	\emph{Positroid varieties and cluster algebras},
	Ann. Sci. \'Ec. Norm. Sup\'er. (4)
	\textbf{56} (2023), no.~3, 859--884.
	\href{https://arxiv.org/abs/1906.03501}{arXiv:1906.03501}.
	
	
	\bibitem{IyamaYoshino}
	O. Iyama and Y. Yoshino,
	\emph{Mutation in triangulated categories and rigid
		Cohen--Macaulay modules},
	Invent. Math. \textbf{172} (2008), no.~1, 117--168.
	\href{https://doi.org/10.1007/s00222-007-0096-4}
	{doi:10.1007/s00222-007-0096-4}.
	
	\bibitem{JKS}
	B. T. Jensen, A. D. King and X. Su,
	\emph{A categorification of Grassmannian cluster algebras},
	Proc. Lond. Math. Soc. (3) \textbf{113} (2016), 185--212.
	\href{https://arxiv.org/abs/1309.7301}{arXiv:1309.7301}.
	
	\bibitem{JKSQuantum}
	B. T. Jensen, A. King and X. Su,
	\emph{Categorification and the quantum Grassmannian},
	Adv. Math. \textbf{406} (2022), Paper No.~108577.
	\href{https://doi.org/10.1016/j.aim.2022.108577}
	{doi:10.1016/j.aim.2022.108577}.
	
	\bibitem{KW}
	B. Keller and Y. Wu,
	\emph{Relative cluster categories and Higgs categories with
		infinite-dimensional morphism spaces},
	with an appendix by C. Fraser and B. Keller.
	\href{https://arxiv.org/abs/2307.12279}{arXiv:2307.12279}.
	
	\bibitem{Lam}
	T. Y. Lam,
	\emph{Lectures on Modules and Rings},
	Graduate Texts in Mathematics, vol.~189,
	Springer-Verlag, New York, 1999.
	
	\bibitem{MMMSV}
	K. M\'esz\'aros, G. Musiker, M. Sherman-Bennett and A. Vidinas,
	\emph{Dimer face polynomials in knot theory and cluster algebras},
	\href{https://arxiv.org/abs/2408.11156v2}{arXiv:2408.11156v2}, 2024.
	
	\bibitem{MS}
	G. Muller and D. E. Speyer,
	\emph{The twist for positroid varieties},
	Proc. Lond. Math. Soc. (3) \textbf{115} (2017), 1014--1071.
	\href{https://arxiv.org/abs/1606.08383}{arXiv:1606.08383}.
	
	\bibitem{PlaApplications}
	P.-G. Plamondon,
	\emph{Cluster algebras via cluster categories with
		infinite-dimensional morphism spaces},
	Compos. Math. \textbf{147} (2011), 1921--1954.
	\href{https://arxiv.org/abs/1004.0830}{arXiv:1004.0830}.
	
	\bibitem{Postnikov}
	A. Postnikov,
	\emph{Total positivity, Grassmannians, and networks},
	preprint, 2006.
	\href{https://arxiv.org/abs/math/0609764}{arXiv:math/0609764}.
	
	\bibitem{PresslandCY}
	M. Pressland,
	\emph{Calabi--Yau properties of Postnikov diagrams},
	Forum Math. Sigma \textbf{10} (2022), e56.
	\href{https://doi.org/10.1017/fms.2022.52}
	{doi:10.1017/fms.2022.52}.
	
	\bibitem{Pressland}
	M. Pressland,
	\emph{Quasi-coincidence of cluster structures on positroid varieties},
	J. Eur. Math. Soc., published online 2025.
	\href{https://doi.org/10.4171/JEMS/1736}{doi:10.4171/JEMS/1736};
	\href{https://arxiv.org/abs/2307.13369v3}{arXiv:2307.13369v3}.
	
	\bibitem{Rickard}
	J. Rickard,
	\emph{Morita theory for derived categories},
	J. Lond. Math. Soc. (2) \textbf{39} (1989), 436--456.
	\href{https://doi.org/10.1112/jlms/s2-39.3.436}
	{doi:10.1112/jlms/s2-39.3.436}.
	
\end{thebibliography}
\end{document}